\documentclass[11pt,letterpaper]{amsart}
\usepackage{amsfonts}
\usepackage{color}
\usepackage{amsmath}
\usepackage{amssymb, esint,tikz}
\usepackage[scr=boondox]{mathalpha}
\usepackage[foot]{amsaddr}
\usepackage{amsthm,color}
\usepackage[left=3.25cm,right=3.25cm,top=3cm,bottom=2.25cm]{geometry}
\usepackage{enumerate}
\usepackage{mathtools}
\usepackage{hyperref}
\usepackage{times}
\allowdisplaybreaks[4]
\definecolor{darkblue}{rgb}{0,0,0.6}
\hypersetup{
    colorlinks=true,       
    linkcolor=darkblue,          
    citecolor=darkblue,        
    filecolor=darkblue,      
    urlcolor=darkblue           
}
\usepackage{bropd}
\usepackage{cancel}
\usepackage{longtable}
\theoremstyle{plain}
\newtheorem{theorem}{Theorem}[section]

\newtheorem{corollary}[theorem]{Corollary}
\newtheorem{lemma}[theorem]{Lemma}
\newtheorem{proposition}[theorem]{Proposition}

\theoremstyle{definition}
\newtheorem{definition}[theorem]{Definition}

\theoremstyle{remark}

\newtheorem*{claim*}{Claim}

\newtheorem{remark}[theorem]{Remark}

\numberwithin{equation}{section}
\newcommand{\curl}{\mathop{\textup{curl}}}
\newcommand{\thmref}[1]{Theorem~\ref{#1}}
\newcommand{\secref}[1]{\S~\ref{#1}}
\newcommand{\lemref}[1]{Lemma~\ref{#1}}
\newcommand{\propref}[1]{Proposition~\ref{#1}}
\newcommand{\rmkref}[1]{Remark~\ref{#1}}
\newcommand{\appdref}[1]{Appendix~\ref{#1}}

\title[Singularities of axisymmetric compressible flows]{The singularity at  degenerate points in steady axisymmetric compressible free surface flows with gravity}
\address{$^1$Department of Mathematics, The Hong Kong University of Science and Technology, Kowloon, Hong Kong}
\address{$^2$Department of Mathematics, Sichuan University, Chengdu, China}
\address{$^3$School of Mathematical Sciences, Shenzhen University, China}
\author{Lili Du$^{2}$}
\email{dulili@scu.edu.cn}
\author{Chunlei Yang$^{1,3}$}
\email{yangcl@ust.hk}
\subjclass[2020]{Primary 35Q35, 76N30; Secondary 35R35, 35B44}
\keywords{Compressible axisymmetric free surface flows; Euler equations; Degenerate points; Free boundary problem; Monotonicity formula; Frequency formula}
\begin{document}
\begin{abstract}
    In this paper, we analyze the singular shape of the free boundary at degenerate points in a three dimensional axisymmetric compressible gravity flow. For all possible degenerate points on the free surface, we prove that  
    \begin{itemize}
        \item The only nontrivial asymptotic behavior of the free surface at the stagnation points away from the axis of symmetry is the Stokes corner flow.
        \item The possible geometries for free boundaries at the non-stagnation axis points are downward pointing or upward pointing cusps.
        \item At the origin, there are only two nontrivial asymptotics possible: the Garabedian's pointed bubble or a horizontal flat surface.
    \end{itemize}
    The problem is associated with the analysis of the degenerate points of a quasilinear free boundary problem of the Bernoulli type, and the main obstacles are the absence of a Weiss-type monotonicity formula. To achieve our goal, we establish for the first time monotonicity formulas for quasilinear Bernoulli type free boundary problems. Our formula works both when the equation becomes singular and when the free boundary condition is degenerate. 

    Moreover, we establish a new nonlinear frequency formula at the horizontal flat points at the origin and integrate it with the compensated compactness theory for Euler equations, ensuring the strong convergence of variational solutions.

    Our results resolve the Stokes conjecture [Mathematical and Physical Papers, Vol. I., 1880] in a generalized compressible, three dimensional axisymmetric framework. In addition, it can also be realized as a compressible counterpart to V\v{a}rv\v{a}ruc\v{a} and Weiss [Comm. Pure Appl. Math., (67), 2014]. Our approach is completely new
    and gives a systematic approach for studying singularities of a singular Bernoulli type quasilinear free boundary problem.
\end{abstract}
\maketitle
\tableofcontents
\section{Introduction and setups}
\subsubsection*{The physical model}
Let us consider the compressible fluids in $\mathbb{R} ^{3}$ under the influence of the gravity (e.g.  compressible gravity water waves in $\mathbb{R} ^{3}$ \cite{MR3887218,MR4439376}). Assume that the motion is steady, then the motion of the fluid is governed by the Euler system 
\begin{align}\label{axeu1}
    \left\{
        \begin{alignedat}{2}
            &(v\cdot\nabla)v=-\frac{1}{\rho}\nabla p-ge_{z}\qquad&&\text{ in }D,\\
            &(v\cdot\nabla)\rho+\rho \operatorname{div}v=0,\qquad&&\text{ in }D,
        \end{alignedat}
    \right.
\end{align}
where $D$ denotes the fluid domain, $v\in \mathbb{R} ^{3}$ the velocity, $\rho$ the density, $e_{z}=(0,0,1)^{\top}$ and $p$ the pressure of the fluid. For compressible liquids, the equation of the state (see~\cite{MR2177323,MR3812074}) $p=p(\rho)\in C^{1,\alpha}([0,\infty))$ is given by 
\begin{equation}\label{Apop}
   p'(\rho) > 0,\quad 2p'(\rho)+\rho p''(\rho)>0\qquad\text{ for }\rho >0.
\end{equation}
It is direct to check that the $\gamma$-law gas $p(\rho)=A\rho^{\gamma}$ for $\gamma>1$ and $A>0$ satisfies~\eqref{Apop}. We refer readers to the reference~\cite{MR4716737,MR92663,MR96477,MR421279,MR3437861} for discussion on this density-pressure relation. 

In this paper, we focus on axisymmetric solutions of the Euler system \eqref{axeu1}, which means that the velocity vector field can be written as 
\begin{equation*}
    v=v_{r}(x_{1},x_{2})e_{r}+v_{\theta}(x_{1},x_{2})e_{\theta}+v_{z}(x_{1},x_{2})e_{z},
\end{equation*}
where $x_{1}=\sqrt{X^{2}+Y^{2}}$ and $x_{2}=Z$ and 
\begin{equation*}
    e_{r}=\left( \frac{X}{x_{1}},\frac{Y}{x_{1}},0 \right),\qquad e_{\theta}=\left( -\frac{Y}{x_{1}},\frac{X}{x_{1}},0 \right),\qquad e_{z}=(0,0,1).
\end{equation*}
We assume that the flow is swirl--free, which means $v_{\theta}\equiv 0$. Then the Euler system~\eqref{axeu1}, which is satisfied by $v=(v_{r},v_{z})$ and $\rho=\rho(x_{1},x_{2})$, is given by 
\begin{align}\label{axeu2}
    \left\{
        \begin{alignedat}{2}
            &(v\cdot\nabla)\rho+\rho \operatorname{div}v+\rho v_{r}=0\quad&&\text{ in }D,\\
            &(v\cdot\nabla)v=-\frac{1}{\rho}\nabla p-ge_{2}\quad&&\text{ in }D,
        \end{alignedat}
    \right.
\end{align}
where $\nabla=(\partial_{x_{1}},\partial_{x_{2}})$, $\operatorname{div}v=\partial_{x_{1}}v_{r}+\partial_{x_{2}}v_{z}$, and $e_{2}=(0,1)^{\top}$. The system~\eqref{axeu2} models for example an axisymmetric compressible jet issuing from a nozzle in the influence of the gravity (see Figure~\ref{jet}) 
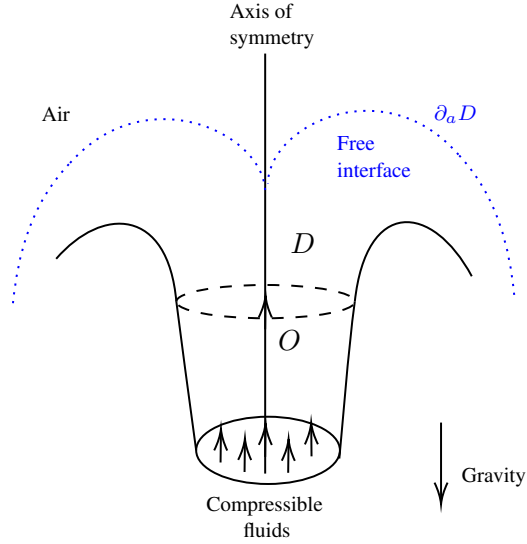
\begin{figure}[!ht]
    \centering
    \tikzset{every picture/.style={line width=0.75pt}} 

    \begin{tikzpicture}[x=0.75pt,y=0.75pt,yscale=-1,xscale=1]
    
    \draw [color={rgb, 255:red, 0; green, 0; blue, 0 }  ,draw opacity=1 ]   (351.8,270) -- (351.8,58) ;
    \draw   (317,258.27) .. controls (317,248.9) and (333.23,241.3) .. (353.25,241.3) .. controls (373.27,241.3) and (389.5,248.9) .. (389.5,258.27) .. controls (389.5,267.65) and (373.27,275.25) .. (353.25,275.25) .. controls (333.23,275.25) and (317,267.65) .. (317,258.27) -- cycle ;
    \draw  [dash pattern={on 4.5pt off 4.5pt}] (307,183.2) .. controls (307,178.56) and (327.15,174.8) .. (352,174.8) .. controls (376.85,174.8) and (397,178.56) .. (397,183.2) .. controls (397,187.84) and (376.85,191.6) .. (352,191.6) .. controls (327.15,191.6) and (307,187.84) .. (307,183.2) -- cycle ;
    \draw    (351.8,261.6) -- (351.98,246.6) ;
    \draw [shift={(352,244.6)}, rotate = 90.67] [color={rgb, 255:red, 0; green, 0; blue, 0 }  ][line width=0.75]    (10.93,-3.29) .. controls (6.95,-1.4) and (3.31,-0.3) .. (0,0) .. controls (3.31,0.3) and (6.95,1.4) .. (10.93,3.29)   ;
    \draw    (363.4,269.6) -- (363.58,254.6) ;
    \draw [shift={(363.6,252.6)}, rotate = 90.67] [color={rgb, 255:red, 0; green, 0; blue, 0 }  ][line width=0.75]    (10.93,-3.29) .. controls (6.95,-1.4) and (3.31,-0.3) .. (0,0) .. controls (3.31,0.3) and (6.95,1.4) .. (10.93,3.29)   ;
    \draw    (341.4,269.2) -- (341.58,254.2) ;
    \draw [shift={(341.6,252.2)}, rotate = 90.67] [color={rgb, 255:red, 0; green, 0; blue, 0 }  ][line width=0.75]    (10.93,-3.29) .. controls (6.95,-1.4) and (3.31,-0.3) .. (0,0) .. controls (3.31,0.3) and (6.95,1.4) .. (10.93,3.29)   ;
    \draw    (351.8,197.2) -- (351.98,182.2) ;
    \draw [shift={(352,180.2)}, rotate = 90.67] [color={rgb, 255:red, 0; green, 0; blue, 0 }  ][line width=0.75]    (10.93,-3.29) .. controls (6.95,-1.4) and (3.31,-0.3) .. (0,0) .. controls (3.31,0.3) and (6.95,1.4) .. (10.93,3.29)   ;
    \draw [color={rgb, 255:red, 0; green, 0; blue, 255 }  ,draw opacity=1 ] [dash pattern={on 0.84pt off 2.51pt}]  (253.4,113.2) .. controls (291.8,70) and (351.8,98.4) .. (351.8,128) ;
    \draw [color={rgb, 255:red, 0; green, 0; blue, 255 }  ,draw opacity=1 ] [dash pattern={on 0.84pt off 2.51pt}]  (224.6,184) .. controls (227,160.8) and (233,134.4) .. (253.4,113.2) ;
    
    \draw [color={rgb, 255:red, 0; green, 0; blue, 255 }  ,draw opacity=1 ] [dash pattern={on 0.84pt off 2.51pt}]  (449.28,109) .. controls (411.78,65.8) and (353.2,94.2) .. (353.2,123.8) ;
    \draw [color={rgb, 255:red, 0; green, 0; blue, 255 }  ,draw opacity=1 ] [dash pattern={on 0.84pt off 2.51pt}]  (477.4,179.8) .. controls (475.06,156.6) and (469.2,130.2) .. (449.28,109) ;
    
    \draw    (329.3,264.4) -- (329.48,249.4) ;
    \draw [shift={(329.5,247.4)}, rotate = 90.67] [color={rgb, 255:red, 0; green, 0; blue, 0 }  ][line width=0.75]    (10.93,-3.29) .. controls (6.95,-1.4) and (3.31,-0.3) .. (0,0) .. controls (3.31,0.3) and (6.95,1.4) .. (10.93,3.29)   ;
    \draw    (375.18,261.3) -- (375.02,246.3) ;
    \draw [shift={(375,244.3)}, rotate = 89.39] [color={rgb, 255:red, 0; green, 0; blue, 0 }  ][line width=0.75]    (10.93,-3.29) .. controls (6.95,-1.4) and (3.31,-0.3) .. (0,0) .. controls (3.31,0.3) and (6.95,1.4) .. (10.93,3.29)   ;
    \draw    (317,258.27) .. controls (314,236) and (308.8,197.8) .. (307,183.2) ;
    \draw    (389.2,262.4) .. controls (390.5,249.33) and (393.6,208) .. (397,183.2) ;
    \draw    (397,183.2) .. controls (402.8,130.8) and (436,133.2) .. (456,170.4) ;
    \draw    (246.4,161.6) .. controls (269.6,134.8) and (302.8,135.6) .. (307,183.2) ;
    \draw    (440.5,244.25) -- (440.5,283.5) ;
    \draw [shift={(440.5,285.5)}, rotate = 270] [color={rgb, 255:red, 0; green, 0; blue, 0 }  ][line width=0.75]    (10.93,-3.29) .. controls (6.95,-1.4) and (3.31,-0.3) .. (0,0) .. controls (3.31,0.3) and (6.95,1.4) .. (10.93,3.29)   ;
    
    \draw (332.4,31.2) node [anchor=north west][inner sep=0.75pt]  [font=\scriptsize] [align=left] {Axis of \\symmetry};
    \draw (320.7,280) node [anchor=north west][inner sep=0.75pt]  [font=\scriptsize] [align=left] {Compressible\\ \ \ \ \ \ \ \ fluids};
    \draw (356.8,196.3) node [anchor=north west][inner sep=0.75pt]    {$O$};
    \draw (237.65,83.5) node [anchor=north west][inner sep=0.75pt]  [font=\scriptsize] [align=left] {Air};
    \draw (386.65,98.1) node [anchor=north west][inner sep=0.75pt]  [font=\scriptsize,color={rgb, 255:red, 0; green, 0; blue, 255 }  ,opacity=1 ] [align=left] {Free \\interface};
    \draw (435.5,82.65) node [anchor=north west][inner sep=0.75pt]  [font=\scriptsize,color={rgb, 255:red, 0; green, 0; blue, 255 }  ,opacity=1 ]  {$\partial _{a} D$};
    \draw (363,147.15) node [anchor=north west][inner sep=0.75pt]    {$D$};
    \draw (449.5,265.25) node [anchor=north west][inner sep=0.75pt]  [font=\scriptsize] [align=left] {Gravity};

    \end{tikzpicture}
    \caption{Axially symmetric compressible jet rising from a nozzle}
    \label{jet}
\end{figure}

We now define some critical variables for compressible flows in gravity field. The sound speed $c(\rho)$ and the Mach number $M$ is defined by  
\begin{equation*}
    c(\rho)=\sqrt{p'(\rho)}\qquad\text{ and }\qquad M=\frac{q}{c(\rho)},
\end{equation*}
where $q:=\sqrt{v_{r}^{2}+v_{z}^{2}}$ is the speed of the flow. Then the flow is called subsonic if $q<c(\rho)$ ($M<1$), and supersonic if $q>c(\rho)$ ($M>1$). We denote the interface between compressible fluid and the air by $\partial_{a}D$ and the density on it by $\bar{\rho}_{0}:=\rho|_{\partial_{a}D}$. Note that $\bar{\rho}_{0}>0$ is a constant determined uniquely by the given atmosphere pressure and~\eqref{Apop}. We can deduce from the second equation in~\eqref{axeu2} that 
\begin{equation}\label{Blaw}
    \frac{q^{2}}{2}+h(\rho)+gx_{2}=\frac{p'(\bar{\rho}_{0})}{2}\qquad\text{ in }D,
\end{equation}
where $h(\rho):=\int_{\bar{\rho}_{0}}^{\rho}\frac{p'(s)}{s}\,ds$ denotes the enthalpy of the flow. In addition, we can also show that  
\begin{equation}\label{vort}
    (x_{1}\rho v)\cdot\nabla\left( \frac{\omega}{x_{1}\rho} \right)=0\qquad\text{ in }D,
\end{equation}
where $\omega:=\partial_{x_{1}}v_{z}-\partial_{x_{2}}v_{r}$ is the vorticity of the flow. 

It should be remark that the assumption in~\eqref{Apop} implies the existence of a local critical density $\rho_{\mathit{cr},x_{2}}$ for each fixed point $x=(x_{1},x_{2})$ in the fluid domain $D$. It is called local (depending on the height $x_{2}$) critical density in the sense that one can identify the flow to subsonic ($\rho>\rho_{\mathit{cr},x_{2}}$) or soinc-supersonic ($\rho \leqslant \rho_{\mathit{cr},x_{2}})$ flows. In fact, the inequality $2p'(\rho)+\rho p''(\rho)>0$ implies that the function $s\mapsto\left( \frac{p'(s)}{2}+h(s) \right)$ is increasing for any $s \geqslant \bar{\rho}_{0}$. We denote the unique local critical density by $\rho_{\mathit{cr},x_2}$, which is determined via the identity  
\begin{align}\label{cr}
    \frac{p'(\rho_{\mathit{cr},x_{2}})}{2}+h(\rho_{\mathit{cr},x_{2}})+gx_{2}=\frac{p'(\bar{\rho}_{0})}{2}\qquad\text{ in }D.
\end{align}
Introducing $\rho_{\mathit{cr},x_{2}}$ into~\eqref{Blaw} gives a unique local critical speed $q_{\mathit{cr},x_{2}}$ so that the flow can be classified by subsonic or sonic-supersonic by $q<q_{\mathrm{cr},x_{2}}$ or $q \geqslant q_{\mathrm{cr},x_{2}}$, respectively. Let us note that a major difference between critical variables defined above and that in the classical compressible hydrodynamics (cf.~\cite[pp.7]{MR96477}) is that, in the absence of the gravity, the critical variables are independent of the vertical coordinates and remain a constant value throughout the entire flow region $D$. In contrast, the critical momentum $\rho_{\mathit{cr},x_{2}}q_{\mathit{cr},x_{2}}$ varies at each specific point $(x_{1},x_{2})\in D$ within the water phase. 

Many results on compressible problems are achieved through conformal mappings, we are more inclined to study the problem as a free boundary problem. This has the advantage that more singular (e.g. cusps and corners) geometries can be included. In this paper, we're interested in the shape of the free boundary $\partial_{a}D$ at degenerate points. These are places where the velocity vector field on the free boundary is zero and the governing equation becomes singular. These degeneracies have a major impact on the free surface geometry, which is a typical occurrence for incompressible gravity waves~\cite{MR2810856,MR2995099,MR3225630,MR4595616}. 

It follows from the definition of $h(\rho)$ that $h|_{\partial_{a}D}=0$, and therefore~\eqref{Blaw} gives that 
\begin{equation}\label{fbcd}
    q^{2}=p'(\bar{\rho}_{0})-2gx_{2}\quad\text{ on }\partial_{a}D,
\end{equation}
and that 
\begin{equation}\label{subst}
    \frac{p'(\rho_{\mathit{cr},x_{2}})}{2}+gx_{2}=\frac{p'(\bar{\rho}_{0})}{2}\quad\text{ on }\partial_{a}D,
\end{equation}
where $\rho_{\mathit{cr},x_{2}}$ is defined in~\eqref{cr}. The equation in~\eqref{subst} indicates that $\rho_{\mathit{cr},x_{2}^{\circ}}<\bar{\rho}_{0}$ for every $x^{\circ}=(x_{1}^{\circ},x_{2}^{\circ})\in \partial_{a}D$ with $x_{2}^{\circ}>0$. In other words, the flow is subsonic at the free boundary point $x^{\circ}$ whenever $x_{2}^{\circ}>0$. On the other hand, we deduce from~\eqref{fbcd} that $q^{2}(x^{\mathit{st}})=0$ where $x^{\mathit{st}}=(x_{1}^{\mathit{st}},x_{2}^{\mathit{st}})$ with $x_{2}^{\mathit{st}}:=\frac{p'(\bar{\rho}_{0})}{2g}$. Since $x_{2}^{\mathit{st}}>0$, we may expect that the flow is subsonic at the stagnation point. However, it should be noted that~\eqref{subst} does not define $\rho_{\mathit{cr},x_{2}^{\mathit{st}}}$, so we need to define $\rho_{\mathit{cr},x_{2}^{\mathit{st}}}$ first. Differentiating the equation in~\eqref{subst} with respect to $x_{2}$ on both sides gives
\begin{equation}\label{subst1}
    \frac{p''(\rho_{\mathit{cr},x_{2}})}{2}\frac{d\rho_{\mathit{cr},x_{2}}}{dx_{2}}=-g\quad\text{ on }\partial_{a}D.
\end{equation}
Then the assumption in~\eqref{Apop} imply that $p''(\rho_{\mathit{cr},x_{2}}) >0$ gives $\frac{d\rho_{\mathit{cr},x_{2}}}{dx_{2}}<0$. We obtain that the function $x_{2}\mapsto\rho_{\mathit{cr},x_{2}}$ is strictly decreasing with $\rho_{\mathit{cr},0}=\bar{\rho}_{0}$. Thus, the number
\begin{equation}\label{crid}
    \rho_{\mathit{cr},x_{2}^{\mathit{st}}}:=\inf_{0 \leqslant x_{2} \leqslant x_{2}^{\mathit{st}}}\rho_{\mathit{cr},x_{2}}
\end{equation}
is well defined and satisfies $\rho_{\mathit{cr},x_{2}^{\mathit{st}}}<\rho_{\mathit{cr},0}=\bar{\rho}_{0}$. Consequently, the flow is uniform subsonic at the stagnation point since $\rho(x^{\mathit{st}})=\bar{\rho}_{0}$.
\subsubsection*{The quasilinear free boundary problem of the Bernoulli type}
In this subsection, we formulate the problem into a quasilinear Bernoulli type free boundary problem. The conservation of mass ensures the existence of a Stokes stream function $\psi=\psi(x_{1},x_{2})$ satisfying
\begin{equation*}
    (v_{r},v_{z})=\left( \frac{1}{x_{1}\rho}\pd{\psi}{x_{2}},-\frac{1}{x_{1}\rho}\pd{\psi}{x_{1}} \right).
\end{equation*}
Suppose that $\psi>0$ in $D$ and we extend $\psi$ by the value of zero to the region so that the fluid domain can be identified with the set $\{(x_{1},x_{2}):\psi(x_{1},x_{2})>0\}$, in short $\{\psi>0\}$. In terms of $\psi$, the Bernoulli law~\eqref{Blaw} becomes 
\begin{align}\label{Blaw1}
    \frac{|\nabla\psi|^{2}}{2x_{1}^{2}\rho^{2}}+h(\rho)+gx_{2}=\frac{p'(\bar{\rho}_{0})}{2}\quad\text{ in }\{\psi>0\},
\end{align} 
and~\eqref{fbcd} can be written as 
\begin{align}\label{Blaw2}
    \frac{|\nabla\psi|^{2}}{x_{1}^{2}\bar{\rho}_{0}^{2}}=(p'(\bar{\rho}_{0})-2gx_{2})\quad\text{ on }\partial\{\psi>0\}.
\end{align}
We can regard the equation~\eqref{Blaw1} as \(\mathscr{F}(\rho;x_{2})-t=0\), where 
\begin{equation*}
    \mathscr{F}(\rho;x_{2})=2\rho^{2}\left( \frac{p'(\bar{\rho}_{0})}{2}-h(\rho)-gx_{2} \right)\qquad\text{ and }\qquad t=\frac{|\nabla\psi|^{2}}{x_{1}^{2}}.
\end{equation*}
Then a direct calculation gives that 
\begin{equation*}
    \partial_{\rho}\mathscr{F}(\rho;x_{2})=2\rho\left[ p'(\rho_{\mathit{cr},x_{2}})+2h(\rho_{\mathit{cr},x_{2}})-(p'(\rho)-2h(\rho)) \right],
\end{equation*}
where we used the  equation~\eqref{cr}. Thanks to~\eqref{Apop}, we obtain $\partial_{\rho}\mathscr{F}(\rho;x_{2})<0$ whenever $\rho>\rho_{\mathit{cr},x_{2}}$. Due to the implicit function theorem, the density $\rho$ can be expressed as a continuously differentiable function of $\frac{|\nabla\psi|^{2}}{x_{1}^{2}}$ and $x_{2}$ in a neighborhood of any point $x^{\circ}=(x_{1}^{\circ},x_{2}^{\circ})$ satisfying $\rho(x^{\circ})>\rho_{\mathit{cr},x_{2}^{\circ}}$. Define
\begin{equation*}
    \partial_{1}\rho\left( \frac{|\nabla\psi|^{2}}{x_{1}^{2}};x_{2} \right):=\frac{\partial\rho(t;x_{2})}{\partial t}\Bigg|_{t=|\nabla\psi|^{2}/x_{1}^{2}},
\end{equation*}
and 
\begin{equation*}
    \partial_{2}\rho\left( \frac{|\nabla\psi|^{2}}{x_{1}^{2}};x_{2} \right):=\frac{\partial\rho(\tfrac{|\nabla\psi|^{2}}{x_{1}^{2}};s)}{\partial s}\Bigg|_{s=x_{2}}.
\end{equation*}
It follows from~\eqref{cr} that for any point $x=(x_{1},x_{2})\in\{\psi>0\}$ satisfying $\rho(x)>\rho_{\mathit{cr},x_{2}}$,
\begin{equation*}
    \partial_{1}\rho\left( \frac{|\nabla\psi|^{2}}{x_{1}^{2}};x_{2} \right)=\frac{1}{2\rho\left[ p'(\rho_{\mathit{cr},x_{2}})+2h(\rho_{\mathit{cr},x_{2}})-(p'(\rho)-2h(\rho)) \right]},
\end{equation*}
and 
\begin{equation*}
    \partial_{2}\rho\left( \frac{|\nabla\psi|^{2}}{x_{1}^{2}};x_{2} \right)=\frac{\rho g}{p'(\rho_{\mathit{cr},x_{2}})+2h(\rho_{\mathit{cr},x_{2}})-(p'(\rho)-2h(\rho))}.
\end{equation*}
We see that $\partial_{1}\rho<0$ and $\partial_{2}\rho<0$ for $\rho>\rho_{\mathit{cr},x_{2}}$. Recalling now the equation in~\eqref{vort} and the definition of $\omega$, for irrotational flows we have 
\begin{equation}\label{vort1}
    \operatorname{div}\left( \frac{\nabla\psi}{x_{1}\rho\left( \frac{|\nabla\psi|^{2}}{x_{1}^{2}};x_{2} \right)} \right)=0\qquad\text{ in }\{\psi>0\}.
\end{equation}
The equation~\eqref{vort1} with the free boundary condition~\eqref{Blaw2} describes a steady axisymmetric compressible irrotational flow without swirl under the influence of the gravity and with a free surface $\partial\{\psi>0\}$. 
\begin{remark}
    Combining the equation in~\eqref{vort1} and in~\eqref{Blaw2}, we obtain the following quasilinear free boundary problem of the Bernoulli-type,
    \begin{align}\label{fb1}
        \left\{
            \begin{alignedat}{2}
                &\operatorname{div}\left( \frac{\nabla\psi}{x_{1}\rho\left( \frac{|\nabla\psi|^{2}}{x_{1}^{2}};x_{2} \right)} \right)=0\qquad&&\text{ in the subsonic fluid phase }\{\psi>0\},\\
                &\frac{1}{x_{1}^{2}\bar{\rho}_{0}^{2}}|\nabla\psi|^{2}=(p'(\bar{\rho}_{0})-2gx_{2})\qquad&&\text{ on the free surface }\partial\{\psi>0\}.
            \end{alignedat}
        \right.
    \end{align}
    We concentrate on the singular behavior of the free surface near degenerate points of the free boundary problem~\eqref{fb1}. Due to the degeneracy of the free boundary condition in equation~\eqref{fb1}, these points are distributed along the sets $\{x_{2}=x_{2}^{\mathit{st}}\}$ (stagnation points), $\{x_{1}=0\}$ (non-stagnation axis points) and at the point $(0,x_{2}^{\mathit{st}})$. In the next remark, we rescale the problem so that the degenerate points are located nicely on two axes.
\end{remark}
\begin{remark}\label{Remark: resc}
    Define $x_{2}^{\mathit{st}}:=\frac{p'(\bar{\rho}_{0})}{2g}$ and introduce a new set of coordinates $y_{1}=x_{1}$ and $y_{2}=x_{2}^{\mathit{st}}-x_{2}$. In the new frame, we set
    \begin{equation*}
        u(y_{1},y_{2})=\frac{\psi(y_{1},x_{2}^{\mathit{st}}-y_{2})}{\bar{\rho}_{0}\sqrt{2g}}\qquad\text{ and }\qquad H(t;y_{2})=\rho(2\bar{\rho}_{0}^{2}gt,x_{2}^{\mathit{st}}-y_{2}).
    \end{equation*}
    After direct calculations, one can show that the problem~\eqref{fb1} is equivalent to 
    \begin{align}\label{fb2}
        \left\{
            \begin{alignedat}{2}
                \operatorname{div}\left( \frac{\nabla u}{y_{1}H\left( \frac{|\nabla u|^{2}}{y_{1}^{2}};y_{2} \right)} \right)&=0\qquad\text{ in }\{u>0\},\\
                \frac{1}{y_{1}^{2}}|\nabla u|^{2}&=y_{2}\qquad\text{ on }\partial\{u>0\},
            \end{alignedat}
        \right.
    \end{align}
    Instead of studying problem~\eqref{fb1}, we will focus on the problem~\eqref{fb2} in the rest of the paper. More specifically, we study the shape of the free surface near  degenerate points including
    \begin{enumerate}
        \item [(1).] Stagnation points: free boundary points distributed along $\{y_{2}=0\}$ but $y_{1}\neq 0$ (corresponds to $\{x_{2}=x_{2}^{\mathit{st}}\}$, $x_{1}\neq 0$ for problem~\eqref{fb1}).
        \item [(2).]  Non-stagnation axis points: free boundary points distributed along $\{y_{1}=0\}$ but $y_{2}>0$  (corresponds to $\{x_{1}=0\}$, $x_{2} <x_{2}^{\mathit{st}}$ for problem~\eqref{fb1}).
        \item [(3).] The origin $(0,0)$ ((corresponds to the point $(0,x_{2}^{\mathit{st}}))$ for problem \eqref{fb1}).
    \end{enumerate}
\end{remark}
\subsubsection*{Notations and notions of solutions}
Before we state our main results, we outline some notations and definitions for the sake of convenience. 

We denote a point in physical space $\mathbb{R}^3$ by $(X, Y, Z)$, as well as $(x_{1},x_{2})$ in the space $\mathbb{R} ^{2}$. We denote by $\mathbb{R} _{+}^{2}:=\{(x_{1},x_{2}): x_{1}>0\}$, by $|x|$ the Euclidean norm in $\mathbb{R} ^{2}$, by $B_{r}(x^{\circ}):=\{x\in \mathbb{R} ^{2}: |x-x^{\circ}|<r\}$ the ball of center $x^{\circ}$ and radius $r$, by $B_{r}^{+}(x^{\circ}):=\{x\in \mathbb{R} ^{2}: x_{1}>0, |x-x^{\circ}|<r\}$, by $\partial B_{r}^{+}(x^{\circ}):=\{x\in \mathbb{R} ^{2}: x_{1}>0, |x_{1}-x_{1}^{\circ}|=r\}$. We denote $B_{r}$ by $B_{r}(0)$ as well as $B_{r}^{+}$ by $B_{r}^{+}(0)$. The two dimensional volume of $B_{1}$ is denoted by $\omega_{2}$. 

Given any open set $A\subset \mathbb{R} ^{2}$, we denote by $\chi_{A}$ the characteristic function of $A$. For any real number $a$, we denote by $a^{\pm}:=\max\{\pm a,0\}$ for positive and negative parts of $a$. We use the notation $\mathcal{L}^{2}$ to represent the two dimensional Lebesgue measure and $\mathcal{H}^{s}$ the $s$ dimensional Hausdorff measure. For any smooth open set $E\subset \mathbb{R} ^{2}$, $|\nabla\chi_{E}|$ coincides with the surface measure of $\partial E$. The reduced boundary of $E$ will be denoted by $\partial_{\mathit{red}}E$. For any open subset $U$ of $\mathbb{R} _{+}^{2}$, we define the weighted $L^{p}$-space 
\begin{equation*}
    L_{\mathit{w}}^{p}(U):= \left\{ v\text{ measurable }: \int_{U}\frac{1}{x_{1}}|v|^{p}\,dx<+\infty \right\} 
\end{equation*}
with the norm 
\begin{equation*}
    \|v\|_{L_{\mathit{w}}^{p}(U)}:=\left( \int_{U}\frac{1}{x_{1}}|v|^{p}\,dx \right)^{1/p}.
\end{equation*}
as well as the weighted Sobolev space $W_{\mathit{w}}^{1,p}$
\begin{equation}\label{wsob}
    W_{\mathit{w}}^{1,p}(U):= \left\{ v\in L_{\mathit{w}}^{p}(U): \text{ all weak partial derivative }\nabla v\in L_{\mathit{w}}^{p}(U) \right\}.
\end{equation}
The local spaces 
\begin{equation*}
    L_{\mathit{w},\mathit{loc}}^{p}(U):= \left\{ v\text{ measurable}: \int_{K}\frac{1}{x_{1}}|v|^{p}\,dx<\infty\quad\forall\,K\subset\subset U \right\},
\end{equation*}
and 
\begin{equation*}
    W_{\mathit{w},\mathit{loc}}^{1,p}(U):=\left\{ v\in L_{\mathit{w},\mathit{loc}}^{p}(U):\text{ weak partial derivative }\nabla v\in L_{\mathit{w},\mathit{loc}}^{p}(U)\right\}.
\end{equation*}
Recall the function $H$ defined in~\rmkref{Remark: resc}, we denote by 
\begin{equation}\label{p1H}
    \partial_{1}H\left( \tfrac{|\nabla u|^{2}}{x_{1}^{2}};x_{2} \right):=\frac{\partial H(t;x_{2})}{\partial t}\Bigg|_{t=|\nabla u|^{2}/x_{1}^{2}},
\end{equation}
and 
\begin{equation}\label{p2H}
    \partial_{2}H\left( \tfrac{|\nabla u|^{2}}{x_{1}^{2}};x_{2}\right):=\frac{\partial H(\tfrac{|\nabla u|^{2}}{x_{1}^{2}};s)}{s}\Bigg|_{s=x_{2}}.
\end{equation}
Define 
\[
    F(t;s):=\int_{0}^{t}\frac{1}{H(\tau;s)}\,d\tau\qquad\text{ for }t \geqslant 0,\qquad s \geqslant 0,
\]
and we use the notation
\begin{equation*}
    \partial_{1}F(t;s):=\pd{F(t;s)}{t}=\frac{1}{H(t;s)},
\end{equation*}
and 
\begin{equation*}
    \partial_{2}F(t;s):=\pd{F(t;s)}{s}=\int_{0}^{t}\frac{-\partial_{2}H(\tau;s)}{H^{2}(\tau;s)}\,d\tau,
\end{equation*}
to denote the partial derivatives for the function $F(t;s)$. Define 
\begin{equation*}
    \Phi(t;s)=2t\partial_{1}F(t;s)-F(t;s),
\end{equation*}
and set $\lambda(x_{2}):=\Phi(x_{2};x_{2})$. A direct calculation gives that 
\begin{equation*}
    \lambda(x_{2})=\frac{x_{2}}{\bar{\rho}_{0}}+\int_{0}^{x_{2}}\pd{!}{\tau}\left( \frac{1}{H(\tau;x_{2})} \right)\tau\,d\tau.
\end{equation*}
Moreover, taking derivatives with respect to $x_{2}$ gives 
\begin{equation}\label{lbd'}
    \lambda'(x_{2})=\frac{1}{\bar{\rho}_{0}}-\int_{0}^{x_{2}}\pd{!}{x_{2}}\left( \frac{1}{H(\tau;x_{2})} \right)\,d\tau,
\end{equation}
where we used $\partial_{1}H(x_{2};x_{2})=-\partial_{2}H(x_{2};x_{2})$. Let us now introduce the notion of a \textit{subsonic variational solution} of the free boundary problem
\begin{align}\label{fb}
    \left\{
        \begin{alignedat}{2}
            \operatorname{div}\left( \frac{\nabla u}{x_{1}H\left( \frac{|\nabla u|^{2}}{x_{1}^{2}};x_{2} \right)} \right)&=0\qquad\text{ in }\Omega\cap\{u>0\},\\
            \frac{1}{x_{1}^{2}}|\nabla u|^{2}&=x_{2}\qquad\text{ on }\Omega\cap\partial\{u>0\},
        \end{alignedat}
    \right.
\end{align}
where $\Omega$ is a connected, relative open subset to the right-half plane $\{(x_{1},x_{2}): x_{1} \geqslant 0\}$. Note that compare to the problem~\eqref{fb2}, we stick with the notation $(x_{1},x_{2})$ to avoid notational ambiguity.
\begin{definition}[Subsonic variational solution]
    Let $\Omega\subset \{(x_{1},x_{2}): x_{1} \geqslant 0\}$ that is connected and relatively open. We say that $u\in W_{\mathit{w},\mathit{loc}}^{1,2}(\Omega)$ is a subsonic variational solution of~\eqref{fb}, provided that
    \begin{enumerate}
        \item [(1).] Nonnegativity: $u \geqslant 0$ in $\Omega$ and $u=0$ on $\{x_{1}=0\}$.
        \item [(2).] Regularity: $u\in C^{0}(\Omega)\cap C^{2}(\Omega\cap\{u>0\})$.
        \item [(3).] Compatibility: For any $x^{\circ}\in \Omega\cap\{x_{1}=0\}$, 
        \begin{equation}\label{com}
            \lim_{\substack{x\to x^{\circ},\\ x\in \Omega\cap\{u>0\} }}\frac{1}{x_{1}}\pd{u}{x_{2}}=0\qquad\text{ and }\qquad \lim_{\substack{x\to x^{\circ},\\ x\in \Omega\cap\{u>0\} }}\frac{1}{x_{1}}\pd{u}{x_{1}}\text{ exists }.
        \end{equation}
        \item [(4).] Subsonicity: There exists a uniform $\varepsilon_{0}>0$ so that 
        \begin{equation}\label{subs}
            H\left( \frac{|\nabla u|^{2}}{2\bar{\rho}_{0}^{2}g};x_{2}^{\mathit{st}}-x_{2} \right)-\rho_{\mathit{cr},x_{2}} \geqslant \varepsilon_{0}\qquad\text{ locally in }\Omega,
        \end{equation}
        where $x_{2}^{\mathit{st}}=\frac{p'(\bar{\rho}_{0})}{2}$ and $\rho_{\mathit{cr},x_{2}}$ is the local critical density defined via~\eqref{subst} and~\eqref{crid}.
        \item [(5).] The free surface $\partial\{u>0\}$ is contained in the plane $\{(x_{1},x_{2}): x_{1} \geqslant 0, x_{2} \geqslant 0\}$.
        \item [(6).] First domain variation formula: The first variation with respect to domain variations of the energy 
        \begin{equation}\label{eng}
            E_{F}(v;\Omega):=\int_{\Omega}x_{1}\left[ F( \tfrac{|\nabla v|^{2}}{x_{1}^{2}};x_{2} ) + \lambda(x_{2})\chi_{\left\{ v>0 \right\} }\right]\,dx
        \end{equation}
        vanishes at $v=u$, i.e.,
        \begin{align}\label{fv}
            \begin{split}
                0&=-\frac{d}{d\varepsilon}\bigg|_{\varepsilon=0}E_{F}(u(x+\varepsilon\phi(x));\Omega)\\
                &=\int_{\Omega}x_{1}\left[ F( \tfrac{|\nabla u|^{2}}{x_{1}^{2}};x_{2}) + \lambda(x_{2})\chi_{\left\{ u>0 \right\} }  \right] \operatorname{div}\phi\,dx\\
                &-2\int_{\Omega}\partial_{1}F( \tfrac{|\nabla u|^{2}}{x_{1}^{2}};x_{2})\frac{1}{x_{1}}\nabla uD\phi\nabla u\,dx\\
                &+\int_{\Omega}\left( F(\tfrac{|\nabla u|^{2}}{x_{1}^{2}};x_{2})-2\partial_{1}F\left(\tfrac{|\nabla u|^{2}}{x_{1}^{2}};x_{2}\right)\tfrac{|\nabla u|^{2}}{x_{1}^{2}} + \lambda(x_{2})\chi_{\left\{ u>0 \right\} }\right)\phi_{1}\,dx\\
                &+\int_{\Omega}x_{1}\left( \partial_{2}F(\tfrac{|\nabla u|^{2}}{x_{1}^{2}};x_{2})+\lambda'(x_{2})\chi_{\left\{ u>0 \right\} } \right)\phi_{2}\,dx
            \end{split}
        \end{align}
        for any $\phi=(\phi_{1},\phi_{2})\in C_{0}^{1}(\Omega;\mathbb{R} ^{2})$ such that $\phi_{1}=0$ on $\{x_{1}=0\}$.
    \end{enumerate}
\end{definition}
\begin{remark}
    The connected set $\Omega$ is relatively open to the right half-plane $\{(x_{1},x_{2}) \geqslant 0\}$ means that $\Omega=\tilde{\Omega}\cap\{(x_{1},x_{2}):x_{1} \geqslant 0\}$ for some open set $\tilde{\Omega}\subset \mathbb{R} ^{2}$. Besides, let us remark that our results are completely local so that we do not specify any boundary conditions on $\partial\Omega$.
\end{remark}
\begin{remark}
    The regularity assumption (2) cannot be deduced from other assumptions in the above definition using regularity theory. Therefore, the subsonic variational solutions are roughly speaking weaker than minimizers and viscosity solutions (cf.~\cite{MR772122}).
\end{remark}
\begin{remark}
    The compatibility condition~\eqref{com} is motivated by the fact that the velocity of the fluid particles on the axis of symmetry should be directed along the axis.
\end{remark}
\begin{remark}
    In the standard Sobolev space $W^{1,2}(\Omega)$, the energy functional defined in equation~\eqref{eng} is unbounded. This is evident from $\lambda(x_{2}) \geqslant 0$ and the inequality 
    \begin{equation*}
        E_{F}(v;\Omega) \geqslant -C\int_{\Omega}\frac{|\nabla u|^{2}}{x_{1}}\,dx,
    \end{equation*}
    where the integral on the right-hand side is finite only if $u$ belongs to the weighted Sobolev space $ W_{\mathit{w}}^{1,2}(\Omega)$ (recall~\eqref{wsob}). This is the reason we introduced the weighted Sobolev space. Furthermore, the first equation in~\eqref{fb} is uniformly elliptic in 
    $\Omega$ provided that the subsonicity condition~\eqref{subs} is satisfied.
\end{remark}
\begin{remark}
    A proof of formula \eqref{fv} is provided in Appendix (see~\appdref{app1}). In the incompressible case, the expression simplifies significantly: here, $F(\tfrac{|\nabla u|^{2}}{x_{1}^{2}};x_{2})\equiv\frac{|\nabla u|^{2}}{\bar{\rho}_{0}x_{1}^{2}}$, $\lambda(x_{2})=\frac{x_{2}}{\bar{\rho}_{0}}$. When $\bar{\rho}_{0}\equiv 1$, this formula aligns with the formula for incompressible fluid in~\cite[Definition 3.1]{MR3225630}. Furthermore, if $u$ and $\partial\{u>0\}$ are smooth,  integration by parts reveals that any solution $u$ satisfying~\eqref{fv} also solves the problem \eqref{fb} in the classical sense.
\end{remark}
We last introduce the notion of subsonic weak solutions.
\begin{definition}[Subsonic weak solutions.]\label{defsubweak}
    We define $u\in W_{\mathit{w},\mathit{loc}}^{1,2}(\Omega)$ to be a weak solution of~\eqref{fb} provided that 
    \begin{enumerate}
        \item [(1).] $u$ is a subsonic variational solution of~\eqref{fb}.
        \item [(2).] The topological free boundary $\partial\{u>0\}\cap\Omega\cap\{x_{1}>0\}\cap\{x_{2}\neq 0\}$ is locally a $C^{2,\alpha}$ curve.
    \end{enumerate}
\end{definition}
\begin{remark}
    For any subsonic weak solution of~\eqref{fb}, for each $x^{\circ}\in \Omega\cap\{x_{1}>0\}\cap\{x_{2}\neq 0\}\cap\partial\{u>0\}$, there exists an open neighborhood $V$ of $x^{\circ}$ so that $u\in C^{1}(V\cap\overline{\{u>0\}})$ satisfies 
    \begin{equation*}
        \frac{1}{x_{1}^{2}}|\nabla u|^{2}=x_{2}\quad\text{ on }V\cap\partial_{\mathit{red}}\{u>0\}.
    \end{equation*}
\end{remark}
\section{Main Results}
In this article, we investigate the singular behavior of the free boundary near degenerate points of the free boundary problem~\eqref{fb}. Due to the degeneracy of the free boundary condition in equation~\eqref{fb}, these points are distributed along the sets $\{x_{2}=0\}$, $\{x_{1}=0\}$ or at the origin $(0,0)$ (Recall~\rmkref{Remark: resc}). We first study the singular profile of the free boundary near the stagnation points of the problem~\eqref{fb}. Our first main result provides a complete characterization of the blow-up limits for this problem at such points and is formally stated below.
\begin{theorem}\label{thm:stp}
    Let $u$ be a subsonic weak solution of~\eqref{fb}, and assume that 
    \begin{equation*}
        \frac{|\nabla u|^{2} }{x_{1}^{2}}\leqslant Cx_{2}^{+}\quad\text{ locally in }\Omega\cap\{u>0\},
    \end{equation*}
    where $\Omega\subset\{(x_{1},x_{2}): x_{1} \geqslant 0\}$ so that $\Omega\cap\{x_{2}=0\}\neq\varnothing$. Then for any stagnation points $x^{\circ}=(x_{1}^{\circ},0)$ with $x_{1}^{\circ}\neq 0$ and any blow-up sequence 
    \begin{equation}\label{thm:bls}
        u_{r}(x):=\frac{u(x^{\circ}+rx)}{r^{3/2}}.
    \end{equation}
    The following statements hold:
    \begin{enumerate}
        \item [(1).] Either the blow-up limit $u_{0}:=\lim_{r\to 0^{+}}u_{r}\equiv 0$, or 
        \begin{equation}\label{thm1}
            u_{0}(x)=\frac{\sqrt{2}}{3}(x_{1}^{2}+x_{2}^{2})^{\tfrac{3}{4}}\cos\left( \frac{3}{2}\arctan\left( \frac{x_{1}}{x_{2}} \right) \right)\chi_{\left\{ -\frac{\pi}{3}<\arctan\left( \tfrac{x_{1}}{x_{2}} \right)<\frac{\pi}{3} \right\} }.
        \end{equation} 
        In the case $u_{0}\not\equiv 0$, the weighted density is  uniquely determined by 
        \begin{equation}\label{thmwd}
            \mathcal{D}^{x_{2}}(0^{+}):=\lim_{r\to 0^{+}}r^{-3}\int_{B_{r}(x^{\circ})}x_{2}^{+}\chi_{\left\{ u>0 \right\} }\,dx=\frac{\sqrt{3}}{3}.
        \end{equation}
        \item [(2)] If the blow-up limit $u_{0}\equiv 0$, the weighted density satisfies  either 
        \begin{equation*}
            \mathcal{D}^{x_{2}}(0^{+}):=\lim_{r\to 0^{+}}r^{-3}\int_{B_{r}(x^{\circ})}x_{2}^{+}\chi_{\left\{ u>0 \right\} }\,dx=0,
        \end{equation*}
        or 
        \begin{equation*}
            \mathcal{D}^{x_{2}}(0^{+}):=\lim_{r\to 0^{+}}r^{-3}\int_{B_{r}(x^{\circ})}x_{2}^{+}\chi_{\left\{ u>0 \right\} }\,dx=\frac{2}{3}.
        \end{equation*}
    \end{enumerate}
\end{theorem}
Subsonic weak solution (defined in Definition~\ref{defsubweak}) are roughly speaking solutions that are constructed via a domain/inner variation of the energy functional associated with problem \eqref{fb}. Subsonic weak solutions for problem~\eqref{fb} near non-degenerate points (where $x^{\circ}\in\partial\{\psi>0\}$ with $x_{1}^{\circ}\neq 0$ and $x_{2}^{\circ}>0$) were first introduced by Alt, Caffarelli, and Friedman~\cite{MR752578,MR772122}. They established that the free surface of problem~\eqref{fb} has finite $\mathcal{H}^{1}$ Hausdorff measure and is smooth in a neighborhood of every non-degenerate point. The foundational results of Alt, Caffarelli, and Friedman on the quasilinear Bernoulli-type free boundary problem have spurred numerous extensions and generalizations. These include studies involving quasilinear operators in Orlicz spaces~\cite{MR2431665}, one-phase 
$p$-Laplace variants~\cite{MR2133664}, two-phase 
$p$-Laplace analogs~\cite{MR4773610}, strongly nonlinear degenerate elliptic operators~\cite{MR4591831}, and degenerate operators in Orlicz–Sobolev spaces~\cite{MR4163981}. Besides, there are some physical generalizations based on the work of~\cite{MR772122}, including the two dimensional compressible subsonic jet~\cite{MR3842050,MR3814594}, and the compressible subsonic jet for three dimensional axisymmetric flows~\cite{MR4246821}. However, all the mentioned primary studies have been focused on the Hausdorff measure and regularity properties of the reduced free boundary near non-degenerate points, in which case the gradient is strictly positive at every free boundary points. To the best of authors' knowledge, this paper presents the first systematic analysis of the profile of the free surface near free boundary points where the gradient of the solution is vanishing.
\begin{figure}[!ht]
    \center
    \tikzset{every picture/.style={line width=0.75pt}} 

    \begin{tikzpicture}[x=0.75pt,y=0.75pt,yscale=-1,xscale=1]
    
    \draw  [fill={rgb, 255:red, 248; green, 231; blue, 28 }  ,fill opacity=1 ] (286.18,8) -- (364.6,81.13) -- (286.18,154.25) -- (207.75,81.13) -- cycle ;
    \draw    (338.85,107.5) -- (400,169.58) ;
    \draw [shift={(401.4,171)}, rotate = 225.43] [color={rgb, 255:red, 0; green, 0; blue, 0 }  ][line width=0.75]    (10.93,-3.29) .. controls (6.95,-1.4) and (3.31,-0.3) .. (0,0) .. controls (3.31,0.3) and (6.95,1.4) .. (10.93,3.29)   ;
    \draw    (240.2,111) -- (185.61,166.18) ;
    \draw [shift={(184.2,167.6)}, rotate = 314.69] [color={rgb, 255:red, 0; green, 0; blue, 0 }  ][line width=0.75]    (10.93,-3.29) .. controls (6.95,-1.4) and (3.31,-0.3) .. (0,0) .. controls (3.31,0.3) and (6.95,1.4) .. (10.93,3.29)   ;
    \draw  [fill={rgb, 255:red, 126; green, 211; blue, 33 }  ,fill opacity=1 ] (147.25,178.35) .. controls (147.25,173.24) and (151.39,169.1) .. (156.5,169.1) -- (213.5,169.1) .. controls (218.61,169.1) and (222.75,173.24) .. (222.75,178.35) -- (222.75,206.1) .. controls (222.75,211.21) and (218.61,215.35) .. (213.5,215.35) -- (156.5,215.35) .. controls (151.39,215.35) and (147.25,211.21) .. (147.25,206.1) -- cycle ;
    \draw  [fill={rgb, 255:red, 126; green, 211; blue, 33 }  ,fill opacity=1 ] (362.75,181.25) .. controls (362.75,176.14) and (366.89,172) .. (372,172) -- (429.5,172) .. controls (434.61,172) and (438.75,176.14) .. (438.75,181.25) -- (438.75,209) .. controls (438.75,214.11) and (434.61,218.25) .. (429.5,218.25) -- (372,218.25) .. controls (366.89,218.25) and (362.75,214.11) .. (362.75,209) -- cycle ;
    \draw    (183.25,215.85) -- (182.76,324.1) ;
    \draw [shift={(182.75,326.1)}, rotate = 270.26] [color={rgb, 255:red, 0; green, 0; blue, 0 }  ][line width=0.75]    (10.93,-3.29) .. controls (6.95,-1.4) and (3.31,-0.3) .. (0,0) .. controls (3.31,0.3) and (6.95,1.4) .. (10.93,3.29)   ;
    \draw    (404.75,218.75) -- (352.7,322.41) ;
    \draw [shift={(351.8,324.2)}, rotate = 296.66] [color={rgb, 255:red, 0; green, 0; blue, 0 }  ][line width=0.75]    (10.93,-3.29) .. controls (6.95,-1.4) and (3.31,-0.3) .. (0,0) .. controls (3.31,0.3) and (6.95,1.4) .. (10.93,3.29)   ;
    \draw  [fill={rgb, 255:red, 248; green, 231; blue, 28 }  ,fill opacity=1 ] (133.4,358.05) .. controls (133.4,341.67) and (155.79,328.4) .. (183.4,328.4) .. controls (211.01,328.4) and (233.4,341.67) .. (233.4,358.05) .. controls (233.4,374.43) and (211.01,387.7) .. (183.4,387.7) .. controls (155.79,387.7) and (133.4,374.43) .. (133.4,358.05) -- cycle ;
    \draw    (404.75,218.75) -- (468.95,323.3) ;
    \draw [shift={(470,325)}, rotate = 238.45] [color={rgb, 255:red, 0; green, 0; blue, 0 }  ][line width=0.75]    (10.93,-3.29) .. controls (6.95,-1.4) and (3.31,-0.3) .. (0,0) .. controls (3.31,0.3) and (6.95,1.4) .. (10.93,3.29)   ;
    \draw  [fill={rgb, 255:red, 248; green, 231; blue, 28 }  ,fill opacity=1 ] (301.8,353.85) .. controls (301.8,337.47) and (324.19,324.2) .. (351.8,324.2) .. controls (379.41,324.2) and (401.8,337.47) .. (401.8,353.85) .. controls (401.8,370.23) and (379.41,383.5) .. (351.8,383.5) .. controls (324.19,383.5) and (301.8,370.23) .. (301.8,353.85) -- cycle ;
    \draw  [fill={rgb, 255:red, 248; green, 231; blue, 28 }  ,fill opacity=1 ] (420,354.65) .. controls (420,338.27) and (442.39,325) .. (470,325) .. controls (497.61,325) and (520,338.27) .. (520,354.65) .. controls (520,371.03) and (497.61,384.3) .. (470,384.3) .. controls (442.39,384.3) and (420,371.03) .. (420,354.65) -- cycle ;
    
    \draw (256.55,43) node [anchor=north west][inner sep=0.75pt]  [font=\scriptsize] [align=left] {Let $\displaystyle u_{r}$ be \\the blow-up\\sequence\\at the stagnation\\point $\displaystyle x^{\circ }$};
    \draw (192.15,139.82) node [anchor=north west][inner sep=0.75pt]  [font=\footnotesize,rotate=-314.09]  {$r\rightarrow 0^{+}$};
    \draw (362.46,110.27) node [anchor=north west][inner sep=0.75pt]  [font=\footnotesize,rotate=-46.67]  {$r\rightarrow 0^{+}$};
    \draw (162.15,181.5) node [anchor=north west][inner sep=0.75pt]    {$u_{0} \not\equiv 0$};
    \draw (376.75,186.4) node [anchor=north west][inner sep=0.75pt]    {$u_{0} \equiv 0$};
    \draw (145.15,348.35) node [anchor=north west][inner sep=0.75pt]  [font=\tiny]  {$\mathcal{D}^{x_{2}}(0^{+}) =\frac{\sqrt{3}}{3}$};
    \draw (151.45,392.4) node [anchor=north west][inner sep=0.75pt]  [font=\scriptsize] [align=left] {Stokes corner\\flow};
    \draw (320.15,348.35) node [anchor=north west][inner sep=0.75pt]  [font=\tiny]  {$\mathcal{D}^{x_{2}}(0^{+}) =\frac{2}{3}$};
    \draw (331.15,388.9) node [anchor=north west][inner sep=0.75pt]  [font=\scriptsize] [align=left] {Horizontal\\flatness};
    \draw (440.25,348.35) node [anchor=north west][inner sep=0.75pt]  [font=\tiny]  {$\mathcal{D}^{x_{2}}( 0^{+}) =0$};
    \draw (462.05,392.8) node [anchor=north west][inner sep=0.75pt]  [font=\scriptsize] [align=left] {Cusp};

    \end{tikzpicture}
    \caption{The trichotomy principle for stagnation point}
    \label{fig:tp}
\end{figure}
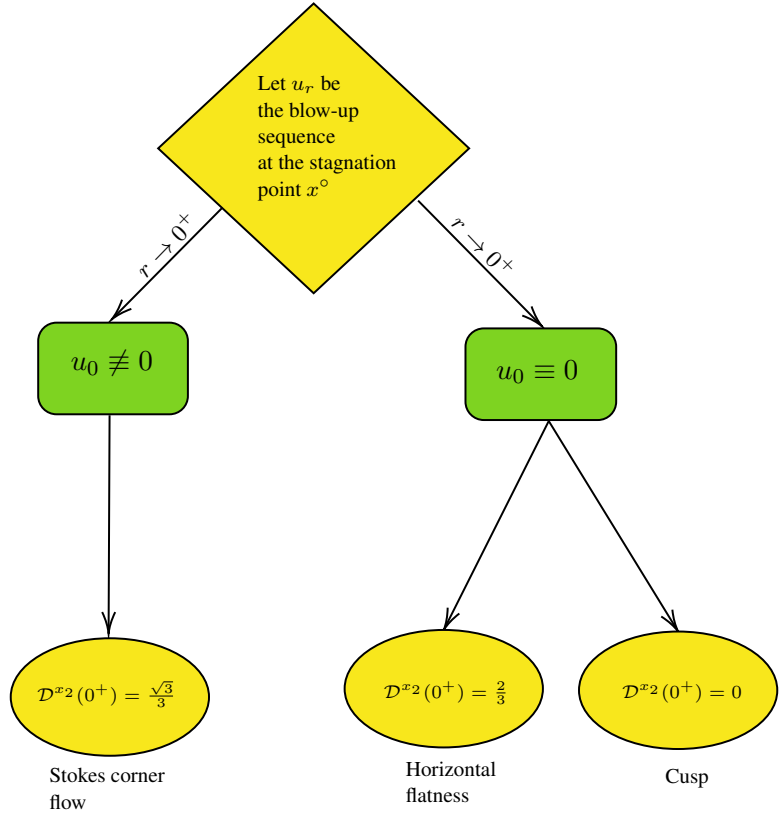

\begin{remark}
    We make some comments on~\thmref{thm:stp}.
    \begin{enumerate}
        \item In Theorem~\ref{thm:stp}, all convergences are strong in $W_{\mathit{loc}}^{1,2}(\mathbb{R} ^{2})$ and locally uniformly in $\mathbb{R} ^{2}$. The term ``weighted density'' introduced in Theorem~\ref{thm:stp} serves as a geometric interpretation of the limit appearing on the right-hand side of the equation in~\eqref{thmwd}, which has two components: The word ``weighted'' corresponds to the integration weight $x_{2}^{+}$ in the expression. The term ``density'' reflects the scenario where, if the weight equals unity (i.e., $x_{2}^{+}=1$), the limit quantifies the volume fraction of $\{u>0\}$ within a ball of radius $r$ as $r\to 0^{+}$. This fraction represents the density of $\{u>0\}$ relative to the entire ball in the limit, implying that the set $\{u>0\}$ occupies the full measure of the ball asymptotically. The superscript $x_{2}$ in the notation $\mathcal{D}^{x_{2}}(0^{+})$ indicates that the degeneracy arises precisely when $x_{2}=0$.
        \item Due to the degeneracy of the free boundary condition
        \begin{equation*}
            \frac{|\nabla u(x_{1},x_{2})|^{2}}{x_{1}^{2}}=x_{2}\quad\text{ on }\Omega\cap\partial\{u>0\}
        \end{equation*}
        at the stagnation point $x^{\circ}=(x_{1}^{\circ},0)$ with $x_{1}^{\circ}\neq 0$, we obtain the invariant scaling defined in equation~\eqref{thm:stp}. Note that in this case the velocity would rescale like $|x_{2}|^{1/2}$. This is the reason we employ the rescaling in equation~\eqref{thm:bls}.
        \item A stagnation point $x^{\circ}=(x_{1}^{\circ},0)\in\partial\{u>0\}$ is called a nontrivial stagnation point if, up to a subsequence of radii, $u_{r}(x)$ defined in equation~\eqref{thm:bls} converges to $u_{0}(x)$ defined in equation~\eqref{thm1}. Conversely, a stagnation point is called a trivial stagnation point if, up to a subsequence of radii, $u_{r}(x)\to u_{0}(x)\equiv 0$ as $r\to 0^{+}$. Furthermore, a trivial stagnation point is called with non-zero density if $u_{0}(x)=0$ and the ``weighted density'' equals to $\frac{2}{3}$, and with zero density if $u_{0}(x)=0$ and the ``weighted density'' equals to $0$ (see Figure~\ref{classtp} for the classification). 
        \item In this work, we establish (presented in~\propref{propcc}) that if the free surface is assumed to be an injective curve, then near each non-degenerate stagnation point, the singular asymptotic behavior is uniquely given by the Stokes corner flow, characterized by a corner with a $120^{\circ}$ opening angle. In contrast, for degenerate trivial stagnation points, the singular asymptotics is a cusp. Finally, in the case of degenerate nontrivial stagnation points, the singular asymptotics is described by a horizontal flat.
        \item  Prior to this work, the possible forms of blow-up limits for the sequence defined in equation~\eqref{thm:bls} for axisymmetric compressible problem~\eqref{fb} were undetermined, and their mutual exclusivity remained unresolved. Theorem~\ref{thm:stp} establishes a \emph{trichotomy principle} for such limits at stagnation points of the free boundary problem~\eqref{fb}, classifying them into three distinct, mutually exclusive asymptotic configurations (see Figure~\ref{fig:tp}).
        \item Degenerate stagnation points—both trivial and nontrivial—are physically implausible at stagnation points of the free boundary problem. To exclude such pathological cases, we prove in Lemma~\ref{lem:cuspstp} that trivial degenerate stagnation points cannot occur under the strong Bernstein assumption
        \begin{equation*}
            \frac{|\nabla u|^{2}}{x_{1}^{2}} \leqslant x_{2}^{+}\quad\text{ locally in }\Omega\cap\{u>0\}.
        \end{equation*}
        From the point view of physics, this corresponds to the situation that the flow or water region is situated beneath the air region.
        \item The exclusion of non-trivial degenerate stagnation points represents a non-trivial challenge. This problem is closely related to analyzing free boundary points for non-minimizers with maximal weighted density, for which results exist in the context of elliptic Bernoulli problems~\cite{MR4850026,MR2748622} and their parabolic counterparts~\cite{MR1989835}. However, prior analyses in this area have universally assumed the existence of a positive lower bound for the gradient on the free boundary. While cases where the gradient vanishes have been only studied for elliptic equations~\cite{MR2810856,MR3225630} and semilinear elliptic equations~\cite{MR2995099}, no existing literature addresses the quasilinear elliptic Bernoulli problem in this context. To address this gap, we introduce a novel frequency formula suited to degenerate non-trivial stagnation singularities on the axis of symmetry in~\secref{Sec:fre}. This is achieved through the application of compensated compactness framework for compressible Euler equations, which facilitates taking limits in the domain variation formula.
    \end{enumerate}
\end{remark}
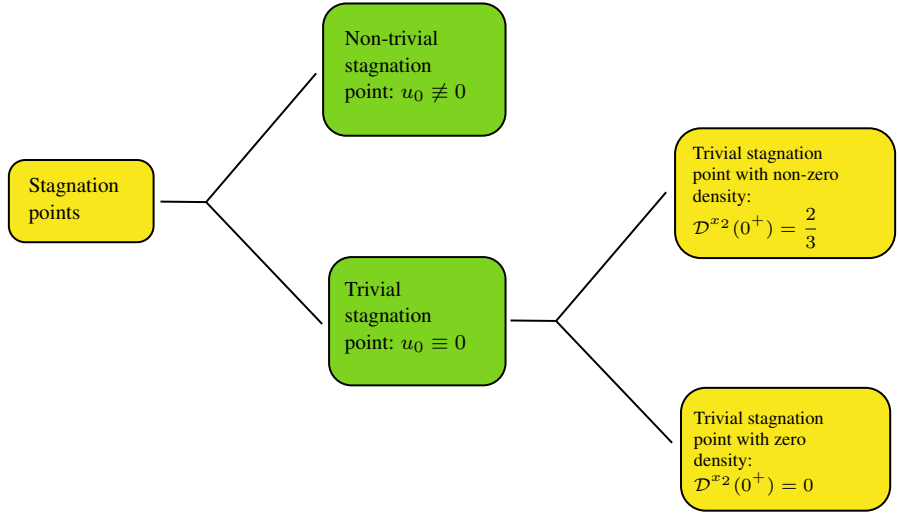
\begin{figure}[!ht]
    \center
    \tikzset{every picture/.style={line width=0.75pt}} 

    \begin{tikzpicture}[x=0.75pt,y=0.75pt,yscale=-1,xscale=1]
    
    \draw  [fill={rgb, 255:red, 248; green, 231; blue, 28 }  ,fill opacity=1 ] (79.5,133.8) .. controls (79.5,129.22) and (83.22,125.5) .. (87.8,125.5) -- (144.2,125.5) .. controls (148.78,125.5) and (152.5,129.22) .. (152.5,133.8) -- (152.5,158.7) .. controls (152.5,163.28) and (148.78,167) .. (144.2,167) -- (87.8,167) .. controls (83.22,167) and (79.5,163.28) .. (79.5,158.7) -- cycle ;
    \draw    (156,146.5) -- (179,146.75) ;
    \draw    (179,146.75) -- (234.5,81.75) ;
    \draw    (179,146.75) -- (237.5,208.25) ;
    
    \draw  [fill={rgb, 255:red, 126; green, 211; blue, 33 }  ,fill opacity=1 ] (238,58.18) .. controls (238,51.73) and (243.23,46.5) .. (249.68,46.5) -- (318.32,46.5) .. controls (324.77,46.5) and (330,51.73) .. (330,58.18) -- (330,101.57) .. controls (330,108.02) and (324.77,113.25) .. (318.32,113.25) -- (249.68,113.25) .. controls (243.23,113.25) and (238,108.02) .. (238,101.57) -- cycle ;
    \draw  [fill={rgb, 255:red, 126; green, 211; blue, 33 }  ,fill opacity=1 ] (241,185.83) .. controls (241,179.57) and (246.07,174.5) .. (252.33,174.5) -- (318.67,174.5) .. controls (324.93,174.5) and (330,179.57) .. (330,185.83) -- (330,227.92) .. controls (330,234.18) and (324.93,239.25) .. (318.67,239.25) -- (252.33,239.25) .. controls (246.07,239.25) and (241,234.18) .. (241,227.92) -- cycle ;
    \draw    (332.5,206.5) -- (355.5,206.75) ;
    \draw    (355.5,206.75) -- (411,141.75) ;
    \draw    (355.5,206.75) -- (414,268.25) ;
    
    \draw  [fill={rgb, 255:red, 248; green, 231; blue, 28 }  ,fill opacity=1 ] (415.5,122.6) .. controls (415.5,115.37) and (421.37,109.5) .. (428.6,109.5) -- (513.4,109.5) .. controls (520.63,109.5) and (526.5,115.37) .. (526.5,122.6) -- (526.5,161.9) .. controls (526.5,169.13) and (520.63,175) .. (513.4,175) -- (428.6,175) .. controls (421.37,175) and (415.5,169.13) .. (415.5,161.9) -- cycle ;
    \draw  [color={rgb, 255:red, 0; green, 0; blue, 0 }  ,draw opacity=1 ][fill={rgb, 255:red, 248; green, 231; blue, 28 }  ,fill opacity=1 ] (418.5,252.9) .. controls (418.5,246.05) and (424.05,240.5) .. (430.9,240.5) -- (511.6,240.5) .. controls (518.45,240.5) and (524,246.05) .. (524,252.9) -- (524,290.1) .. controls (524,296.95) and (518.45,302.5) .. (511.6,302.5) -- (430.9,302.5) .. controls (424.05,302.5) and (418.5,296.95) .. (418.5,290.1) -- cycle ;
    
    \draw (88,133) node [anchor=north west][inner sep=0.75pt]  [font=\scriptsize] [align=left] {Stagnation\\points};
    \draw (247.5,58.77) node [anchor=north west][inner sep=0.75pt]  [font=\scriptsize] [align=left] {Non-trivial \\stagnation\\point: $\displaystyle u_{0} \not\equiv 0$};
    \draw (247.5,185.27) node [anchor=north west][inner sep=0.75pt]  [font=\scriptsize] [align=left] {Trivial \\stagnation\\point: $\displaystyle u_{0} \equiv 0$};
    \draw (423,117.27) node [anchor=north west][inner sep=0.75pt]  [font=\tiny] [align=left] {Trivial stagnation\\point with non-zero\\density:\\$\displaystyle \mathcal{D}^{x_{2}}( 0^{+}) =\frac{2}{3}$};
    \draw (423.5,250.77) node [anchor=north west][inner sep=0.75pt]  [font=\tiny] [align=left] {Trivial stagnation\\point with zero\\density:\\$\displaystyle \mathcal{D}^{x_{2}}( 0^{+}) =0$};

    \end{tikzpicture}
    \caption{Classification of stagnation points}
    \label{classtp}
\end{figure}
We next investigate the non-stagnation points on the axis of symmetry. A major challenge in this scenario arises because the first equation in~\eqref{fb} reduces to a singular quasilinear elliptic equation along the axis of symmetry.
\begin{theorem}\label{thm:ad}
    Let $u$ be a subsonic weak solution of~\eqref{fb}, and assume that 
    \begin{equation*}
        \left|\frac{|\nabla u|^{2} }{x_{1}^{2}}-x_{2}\right|\leqslant Cx_{1}\quad\text{ locally in }\Omega\cap\{u>0\},
    \end{equation*}
    where $\Omega\subset\{(x_{1},x_{2}): x_{1} \geqslant 0\}$ so that $\Omega\cap\{x_{1}=0\}\neq\varnothing$. Then for any non-stagnation axis points $x^{\circ}=(0,x_{2}^{\circ})$ with $x_{2}^{\circ}>0$ and any blow-up sequence 
    \begin{equation*}
        u_{r}(x):=\frac{u(x^{\circ}+rx)}{r^{2}}.
    \end{equation*}
    The following statements hold:
    \begin{enumerate}
        \item [(1).] Either the blow-up limit $u_{0}:=\lim_{r\to 0^{+}}u_{r}\equiv 0$, or 
        \begin{equation*}
            u_{0}(x_{1},x_{2})=\alpha x_{1}^{2},
        \end{equation*}
        where $\alpha>0$ is a uniquely determined constant. In the case $u_{0}\not\equiv 0$, the weighted density is uniquely determined by 
        \begin{equation*}
            \mathcal{D}^{x_{1}}(0^{+}):=\lim_{r\to 0^{+}}r^{-3}\int_{B_{r}^{+}(x^{\circ})}x_{1}\chi_{\left\{ u>0 \right\} }\,dx=\frac{2}{3}.
        \end{equation*}
        \item If the blow-up limit $u_{0}\equiv 0$, then either 
        \begin{equation*}
            \mathcal{D}^{x_{1}}(0^{+})=\lim_{r\to 0^{+}}r^{-3}\int_{B_{r}^{+}(x^{\circ})}x_{1}\chi_{\left\{ u>0 \right\} }\,dx=0,
        \end{equation*}
        or 
        \begin{equation*}
            \mathcal{D}^{x_{1}}(0^{+})=\lim_{r\to 0^{+}}r^{-3}\int_{B_{r}^{+}(x^{\circ})}x_{1}\chi_{\left\{ u>0 \right\} }\,dx=\frac{2}{3}.
        \end{equation*}
    \end{enumerate}
\end{theorem}
\begin{figure}[!ht]
    \center
    \tikzset{every picture/.style={line width=0.75pt}} 

    \begin{tikzpicture}[x=0.75pt,y=0.75pt,yscale=-1,xscale=1]
    
    \draw  [fill={rgb, 255:red, 248; green, 231; blue, 28 }  ,fill opacity=1 ] (296.18,27.2) -- (374.6,100.32) -- (296.18,173.45) -- (217.75,100.32) -- cycle ;
    \draw    (348.85,126.7) -- (410,188.78) ;
    \draw [shift={(411.4,190.2)}, rotate = 225.43] [color={rgb, 255:red, 0; green, 0; blue, 0 }  ][line width=0.75]    (10.93,-3.29) .. controls (6.95,-1.4) and (3.31,-0.3) .. (0,0) .. controls (3.31,0.3) and (6.95,1.4) .. (10.93,3.29)   ;
    \draw    (250.2,130.2) -- (195.61,185.38) ;
    \draw [shift={(194.2,186.8)}, rotate = 314.69] [color={rgb, 255:red, 0; green, 0; blue, 0 }  ][line width=0.75]    (10.93,-3.29) .. controls (6.95,-1.4) and (3.31,-0.3) .. (0,0) .. controls (3.31,0.3) and (6.95,1.4) .. (10.93,3.29)   ;
    \draw  [fill={rgb, 255:red, 126; green, 211; blue, 33 }  ,fill opacity=1 ] (157.25,197.55) .. controls (157.25,192.44) and (161.39,188.3) .. (166.5,188.3) -- (223.5,188.3) .. controls (228.61,188.3) and (232.75,192.44) .. (232.75,197.55) -- (232.75,225.3) .. controls (232.75,230.41) and (228.61,234.55) .. (223.5,234.55) -- (166.5,234.55) .. controls (161.39,234.55) and (157.25,230.41) .. (157.25,225.3) -- cycle ;
    \draw  [fill={rgb, 255:red, 126; green, 211; blue, 33 }  ,fill opacity=1 ] (372.75,200.45) .. controls (372.75,195.34) and (376.89,191.2) .. (382,191.2) -- (439.5,191.2) .. controls (444.61,191.2) and (448.75,195.34) .. (448.75,200.45) -- (448.75,228.2) .. controls (448.75,233.31) and (444.61,237.45) .. (439.5,237.45) -- (382,237.45) .. controls (376.89,237.45) and (372.75,233.31) .. (372.75,228.2) -- cycle ;
    \draw    (193.25,235.05) -- (192.76,343.3) ;
    \draw [shift={(192.75,345.3)}, rotate = 270.26] [color={rgb, 255:red, 0; green, 0; blue, 0 }  ][line width=0.75]    (10.93,-3.29) .. controls (6.95,-1.4) and (3.31,-0.3) .. (0,0) .. controls (3.31,0.3) and (6.95,1.4) .. (10.93,3.29)   ;
    \draw    (414.75,237.95) -- (362.7,341.61) ;
    \draw [shift={(361.8,343.4)}, rotate = 296.66] [color={rgb, 255:red, 0; green, 0; blue, 0 }  ][line width=0.75]    (10.93,-3.29) .. controls (6.95,-1.4) and (3.31,-0.3) .. (0,0) .. controls (3.31,0.3) and (6.95,1.4) .. (10.93,3.29)   ;
    \draw  [fill={rgb, 255:red, 248; green, 231; blue, 28 }  ,fill opacity=1 ] (143.4,377.25) .. controls (143.4,360.87) and (165.79,347.6) .. (193.4,347.6) .. controls (221.01,347.6) and (243.4,360.87) .. (243.4,377.25) .. controls (243.4,393.63) and (221.01,406.9) .. (193.4,406.9) .. controls (165.79,406.9) and (143.4,393.63) .. (143.4,377.25) -- cycle ;
    \draw    (414.75,237.95) -- (478.95,342.5) ;
    \draw [shift={(480,344.2)}, rotate = 238.45] [color={rgb, 255:red, 0; green, 0; blue, 0 }  ][line width=0.75]    (10.93,-3.29) .. controls (6.95,-1.4) and (3.31,-0.3) .. (0,0) .. controls (3.31,0.3) and (6.95,1.4) .. (10.93,3.29)   ;
    \draw  [fill={rgb, 255:red, 248; green, 231; blue, 28 }  ,fill opacity=1 ] (311.8,373.05) .. controls (311.8,356.67) and (334.19,343.4) .. (361.8,343.4) .. controls (389.41,343.4) and (411.8,356.67) .. (411.8,373.05) .. controls (411.8,389.43) and (389.41,402.7) .. (361.8,402.7) .. controls (334.19,402.7) and (311.8,389.43) .. (311.8,373.05) -- cycle ;
    \draw  [fill={rgb, 255:red, 248; green, 231; blue, 28 }  ,fill opacity=1 ] (430,373.85) .. controls (430,357.47) and (452.39,344.2) .. (480,344.2) .. controls (507.61,344.2) and (530,357.47) .. (530,373.85) .. controls (530,390.23) and (507.61,403.5) .. (480,403.5) .. controls (452.39,403.5) and (430,390.23) .. (430,373.85) -- cycle ;
    
    \draw (268.05,54.7) node [anchor=north west][inner sep=0.75pt]  [font=\scriptsize] [align=left] {Let $\displaystyle u_{r}$ be \\the blow-up\\sequence\\at the \\non-stagnation\\axis point $\displaystyle x^{\circ }$};
    \draw (202.15,159.02) node [anchor=north west][inner sep=0.75pt]  [font=\footnotesize,rotate=-314.09]  {$r\rightarrow 0^{+}$};
    \draw (372.46,129.47) node [anchor=north west][inner sep=0.75pt]  [font=\footnotesize,rotate=-46.67]  {$r\rightarrow 0^{+}$};
    \draw (172.15,200.7) node [anchor=north west][inner sep=0.75pt]    {$u_{0} \not\equiv 0$};
    \draw (386.75,205.6) node [anchor=north west][inner sep=0.75pt]    {$u_{0} \equiv 0$};
    \draw (164.15,367.55) node [anchor=north west][inner sep=0.75pt]  [font=\tiny]  {$\mathcal{D}^{x_{1}}(0^{+}) =\frac{2}{3}$};
    \draw (330.65,367.55) node [anchor=north west][inner sep=0.75pt]  [font=\tiny]  {$\mathcal{D}^{x_{1}}( 0^{+}) =\frac{2}{3}$};
    \draw (341.15,408.1) node [anchor=north west][inner sep=0.75pt]  [font=\scriptsize] [align=left] {Horizontal\\flatness};
    \draw (450.65,367.55) node [anchor=north west][inner sep=0.75pt]  [font=\tiny]  {$\mathcal{D}^{x_{1}}( 0^{+}) =0$};
    \draw (472.05,411.5) node [anchor=north west][inner sep=0.75pt]  [font=\scriptsize] [align=left] {Cusp};
    \draw (178.05,412.5) node [anchor=north west][inner sep=0.75pt]  [font=\scriptsize] [align=left] {Cusp};

    \end{tikzpicture}
    \caption{The trichotomy principal for non-stagnation axis points}
    \label{fig: trpad}
\end{figure}
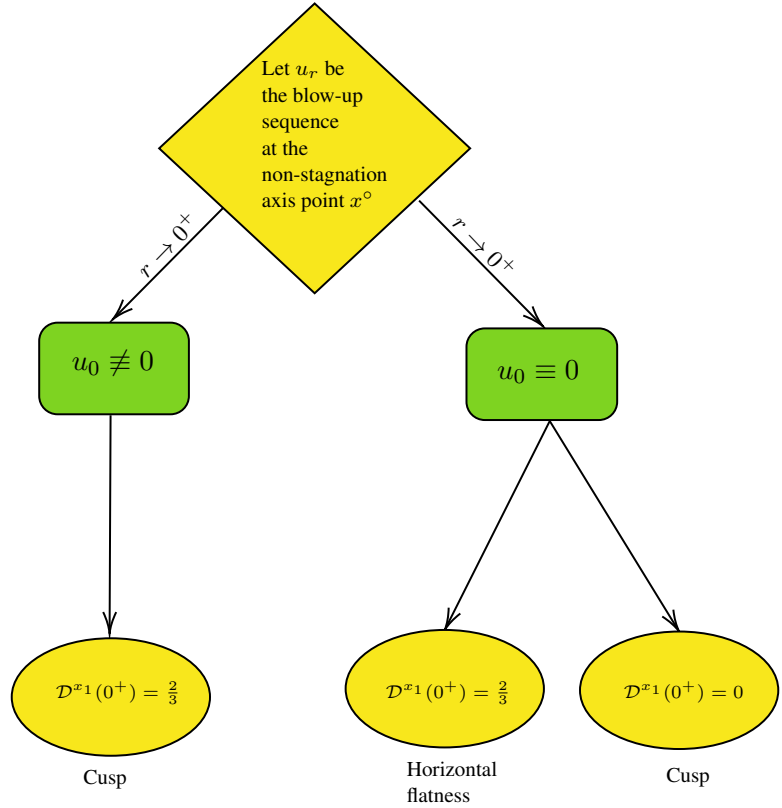
\begin{remark}
    We now remark several special points in~\thmref{thm:ad}, contrasting their properties with those established in the preceding theorem with respect to stagnation points to highlight key differences in the results.
    \begin{enumerate}
        \item In~\thmref{thm:ad}, all convergence occurs in the weighted Sobolev space $W_{\mathit{w},\mathit{loc}}^{1,2}(\mathbb{R} _{+}^{2})$ and locally uniform in $\mathbb{R}_{+} ^{2}$. This is due to the degeneracy of the quasilinear operator on the axis of symmetry $\{x_{1}=0\}$.
        \item It is important to note that the rescaling chosen for non-stagnation axis points differs from that applied to stagnation points away from the axis of symmetry. This distinction arises primarily due to the degeneracy of the free boundary condition $|\nabla u|^{2}=x_{1}^{2}x_{2}$. Specifically, when the non-stagnation axis point $x^{\circ}$ is not a stagnation point, the invariant scaling suggested by the free boundary condition is given $u(x^{\circ}+rx)/r^{2}$.
        \item Following the classification in the analysis of stagnation points, we classify non-stagnation axis points into two categories: nontrivial non-stagnation axis points and trivial axis non-stagnation points. The latter class is further subdivided into trivial non-stagnation axis points with zero density and trivial non-stagnation axis  points with non-zero density. Since the distinction between trivial/non-trivial densities relies on weighted densities, one may define a trivial non-stagnation axis point with zero density if it satisfies $u_{0}=0$ and $\mathcal{D}^{x_{1}}(0^{+})=0$, and with non-zero density if $u_{0}=0$ and $\mathcal{D}^{x_{1}}(0^{+})=\frac{2}{3}$ (see Figure~\ref{Classaxi}). In~\propref{propccad}, we give explicit singular asymptotics of the free surface for each possible non-stagnation axis point.
        \item Compared to the singular asymptotics of the free boundary at the non-degenerate stagnation points away from the axis of symmetry, which corresponds to the fact that $\{u_{0}>0\}$ is a cone solution (with opening angle of $\tfrac{2}{3}\pi$, cf. Corollary~\ref{corbl}), the singular profile at the non-degenerate axis point is cusp like (which has dramatic difference to any cone solution). From the point of view of mathematics, this is mainly due to the fact that the singularities appearing on the axis of symmetry forces that the governing equations for the blow-up limit $u_{0}$ is no longer Laplace, but a degenerate elliptic equation of the form $\operatorname{div}\left( \frac{1}{x_{1}}\nabla u_{0} \right)=0$ (cf. Corollary~\ref{corblad}). To solve the partial differential equation governing $u_{0}$, we employ the velocity potential for incompressible flow, which reduces the problem to a Legendre differential equation. The solutions are thus expressed in terms of Legendre polynomials. This is the main difference to the case of stagnation points.
        \item Similar to the trichotomy principle we have displayed for stagnation points, we have the corresponding trichotomy for non-stagnation axis points (see Figure~\ref{fig: trpad}). Let us remark that the trichotomy framework established for compressible axisymmetric free boundary problems applies equally to the incompressible case~\cite{MR3225630}. Consequently, our analysis along the axis of symmetry constitutes a compressible counterpart to the existing incompressible theory.
        \item For degenerate nontrivial stagnation points, we apply the developed frequency formula (a similar one to~\secref{Sec:fre} for nontrivial degenerate stagnation points) $u_{0}=0$ and $\mathcal{D}^{x_{1}}(0^{+})=\frac{2}{3}$ cannot occur. The question of whether the analogous case $u_{0}=0$ and $\mathcal{D}^{x_{1}}(0^{+})=\frac{2}{3}$ persists for axisymmetric degenerate points remains unresolved, even in the incompressible setting. We will defer this investigation to future studies.
    \end{enumerate}
\end{remark}
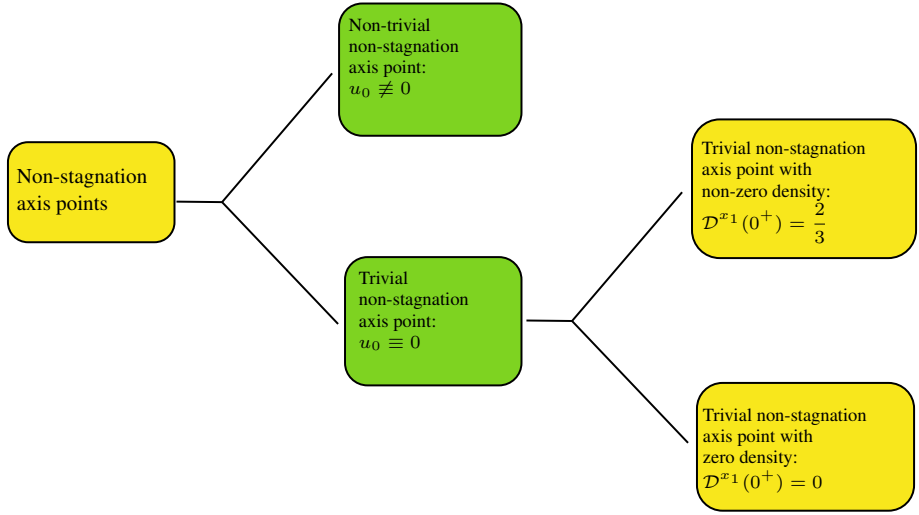
\begin{figure}[!ht]
    \center
    \tikzset{every picture/.style={line width=0.75pt}} 

    \begin{tikzpicture}[x=0.75pt,y=0.75pt,yscale=-1,xscale=1]
    
    \draw  [fill={rgb, 255:red, 248; green, 231; blue, 28 }  ,fill opacity=1 ] (79,133.6) .. controls (79,128.02) and (83.52,123.5) .. (89.1,123.5) -- (152.9,123.5) .. controls (158.48,123.5) and (163,128.02) .. (163,133.6) -- (163,163.9) .. controls (163,169.48) and (158.48,174) .. (152.9,174) -- (89.1,174) .. controls (83.52,174) and (79,169.48) .. (79,163.9) -- cycle ;
    \draw    (164,153.5) -- (187,153.75) ;
    \draw    (187,153.75) -- (242.5,88.75) ;
    \draw    (187,153.75) -- (245.5,215.25) ;
    
    \draw  [fill={rgb, 255:red, 126; green, 211; blue, 33 }  ,fill opacity=1 ] (246,65.18) .. controls (246,58.73) and (251.23,53.5) .. (257.68,53.5) -- (326.32,53.5) .. controls (332.77,53.5) and (338,58.73) .. (338,65.18) -- (338,108.57) .. controls (338,115.02) and (332.77,120.25) .. (326.32,120.25) -- (257.68,120.25) .. controls (251.23,120.25) and (246,115.02) .. (246,108.57) -- cycle ;
    \draw  [fill={rgb, 255:red, 126; green, 211; blue, 33 }  ,fill opacity=1 ] (249,192.83) .. controls (249,186.57) and (254.07,181.5) .. (260.33,181.5) -- (326.67,181.5) .. controls (332.93,181.5) and (338,186.57) .. (338,192.83) -- (338,234.92) .. controls (338,241.18) and (332.93,246.25) .. (326.67,246.25) -- (260.33,246.25) .. controls (254.07,246.25) and (249,241.18) .. (249,234.92) -- cycle ;
    \draw    (340.5,213.5) -- (363.5,213.75) ;
    \draw    (363.5,213.75) -- (419,148.75) ;
    \draw    (363.5,213.75) -- (422,275.25) ;
    
    \draw  [fill={rgb, 255:red, 248; green, 231; blue, 28 }  ,fill opacity=1 ] (424,126) .. controls (424,118.27) and (430.27,112) .. (438,112) -- (523,112) .. controls (530.73,112) and (537,118.27) .. (537,126) -- (537,168) .. controls (537,175.73) and (530.73,182) .. (523,182) -- (438,182) .. controls (430.27,182) and (424,175.73) .. (424,168) -- cycle ;
    \draw  [fill={rgb, 255:red, 248; green, 231; blue, 28 }  ,fill opacity=1 ] (426.5,257.9) .. controls (426.5,250.78) and (432.28,245) .. (439.4,245) -- (523.1,245) .. controls (530.22,245) and (536,250.78) .. (536,257.9) -- (536,296.6) .. controls (536,303.72) and (530.22,309.5) .. (523.1,309.5) -- (439.4,309.5) .. controls (432.28,309.5) and (426.5,303.72) .. (426.5,296.6) -- cycle ;
    
    \draw (82,135.5) node [anchor=north west][inner sep=0.75pt]  [font=\scriptsize] [align=left] {Non-stagnation\\axis points};
    \draw (249.5,59.27) node [anchor=north west][inner sep=0.75pt]  [font=\tiny] [align=left] {Non-trivial \\non-stagnation\\axis point: \\$\displaystyle u_{0} \not\equiv 0$};
    \draw (254.5,187.33) node [anchor=north west][inner sep=0.75pt]  [font=\tiny] [align=left] {Trivial \\non-stagnation\\axis point:\\ $\displaystyle u_{0}\equiv 0$};
    \draw (427.5,121.77) node [anchor=north west][inner sep=0.75pt]  [font=\tiny] [align=left] {Trivial non-stagnation\\axis point with \\non-zero density:\\$\displaystyle \mathcal{D}^{x_{1}}( 0^{+}) =\frac{2}{3}$};
    \draw (428,256.27) node [anchor=north west][inner sep=0.75pt]  [font=\tiny] [align=left] {Trivial non-stagnation\\axis point with \\zero density:\\$\displaystyle \mathcal{D}^{x_{1}}( 0^{+}) =0$};

    \end{tikzpicture}
    \caption{Classification of non-stagnation axis point}
    \label{Classaxi}
\end{figure}
Finally, we examine the intersection of stagnation points and axisymmetric degenerate points, which yields the unique point $x^{\circ}=(0,0)$. Our analysis in this case establishes the following result.
\begin{theorem}\label{thm:orig}
    Let $u$ be a subsonic variational solution of~\eqref{fb}, and assume that 
    \begin{equation*}
        \frac{|\nabla u|^{2}}{x_{1}^{2}} \leqslant Cx_{2}^{+}\quad\text{ locally in }\Omega\cap\{u>0\}\text{ with }(0,0)\in \Omega.
    \end{equation*}
    Consider the blow-up sequence 
    \begin{equation}\label{thm: org}
        u_{r}(x):=\frac{u(rx)}{r^{5/2}}.
    \end{equation}
    Then the following statements hold.
    \begin{enumerate}
        \item [(1).] Either the blow-up limit $u_{0}\equiv 0$ or 
        \begin{align}\label{thm:blor}
            \begin{split}
                u_{0}(x)=\beta_{0}(x_{1}^{2}+x_{2}^{2})^{1/4}P_{3/2}'\left(-\tfrac{x_{2}}{\sqrt{x_{1}^{2}+x_{2}^{2}}}\right)\chi_{\left\{\pi-\theta^{*}<\arctan\left(\tfrac{x_{1}}{x_{2}}\right)<\pi \right\} },
            \end{split}
        \end{align}
        where $\beta_{0}>0$ is a unique positive constant, $\theta^{*}=\arccos s^{*}$, $s^{*}\in(-1,0)$ is the unique solution $s\in(0,1)$ of $P_{3/2}'(s)=0$ (where $P_{3/2}$ is the Legendre function of the first kind). When the blow-up limit $u_{0}$ has the form~\eqref{thm:blor},
        \begin{align*}
            \mathcal{D}^{x_{1}x_{2}}(0^{+})&:=\lim_{r\to 0^{+}}r^{-4}\int_{B_{r}^{+}}x_{1}x_{2}\chi_{\left\{ u>0 \right\} }\,dx\\
            &=\int_{B_{1}^{+}\cap \left\{  \pi-\theta^{*}<\arctan\left( \tfrac{x_{1}}{x_{2}} \right)<\pi\right\}  }x_{1}x_{2}^{+}\,dx:=m_{0}.
        \end{align*} 
        \item If $u_{0}(x)\equiv 0$, then either 
        \begin{equation*}
            \mathcal{D}^{x_{1}x_{2}}(0^{+})=\lim_{r\to 0^{+}}r^{-4}\int_{B_{r}^{+}}x_{1}x_{2}\chi_{\left\{ u>0 \right\} }\,dx=0,
        \end{equation*}
        or 
        \begin{equation*}
            \mathcal{D}^{x_{1}x_{2}}(0^{+})=\lim_{r\to 0^{+}}r^{-4}\int_{B_{r}^{+}}x_{1}x_{2}\chi_{\left\{ u>0 \right\} }\,dx=\int_{B_{1}}x_{1}x_{2}^{+}\,dx=\frac{1}{8}.
        \end{equation*}
    \end{enumerate}
    If we further assume the following strong Bernstein estimate that
    \begin{equation*}
        \frac{|\nabla u|^{2}}{x_{1}^{2}} \leqslant x_{2}^{+}\quad\text{ locally in }\Omega\cap\{u>0\}\text{ with }(0,0)\in\Omega,
    \end{equation*}
    and that the free boundary $B_{1}^{+}\cap\partial\{u>0\}$ is in a neighborhood of $0$ a continuous injective curve $\sigma:I\to \mathbb{R} ^{2}$, where $I$ is an interval of $\mathbb{R} $ containing $0$, such that $\sigma=(\sigma_{1},\sigma_{2})$ and $\sigma(0)=0$. Then $\sigma_{1}(t)\neq 0$ in $[0,t_{1})\setminus\{0\}$, 
    \begin{equation*}
        \lim_{t\to 0^{+}}\frac{\sigma_{2}(t)}{\sigma_{1}(t)}=0,
    \end{equation*}
    and 
    \begin{equation*}
        \frac{u(rx)}{\|u\|_{L_{\mathit{w}}^{2}(\partial B_{r}^{+})}}\to \beta x_{1}^{2}x_{2}^{+}\quad\text{ as }r\to 0^{+},
    \end{equation*}
    where $\beta$ is a unique positive constant given by 
    \begin{equation*}
        \beta:=\frac{1}{\sqrt{\int_{\partial B_{1}^{+}}x_{1}^{3}(x_{2}^{+})^{2}\,d\mathcal{H}^{1}}}.
    \end{equation*} 
\end{theorem}
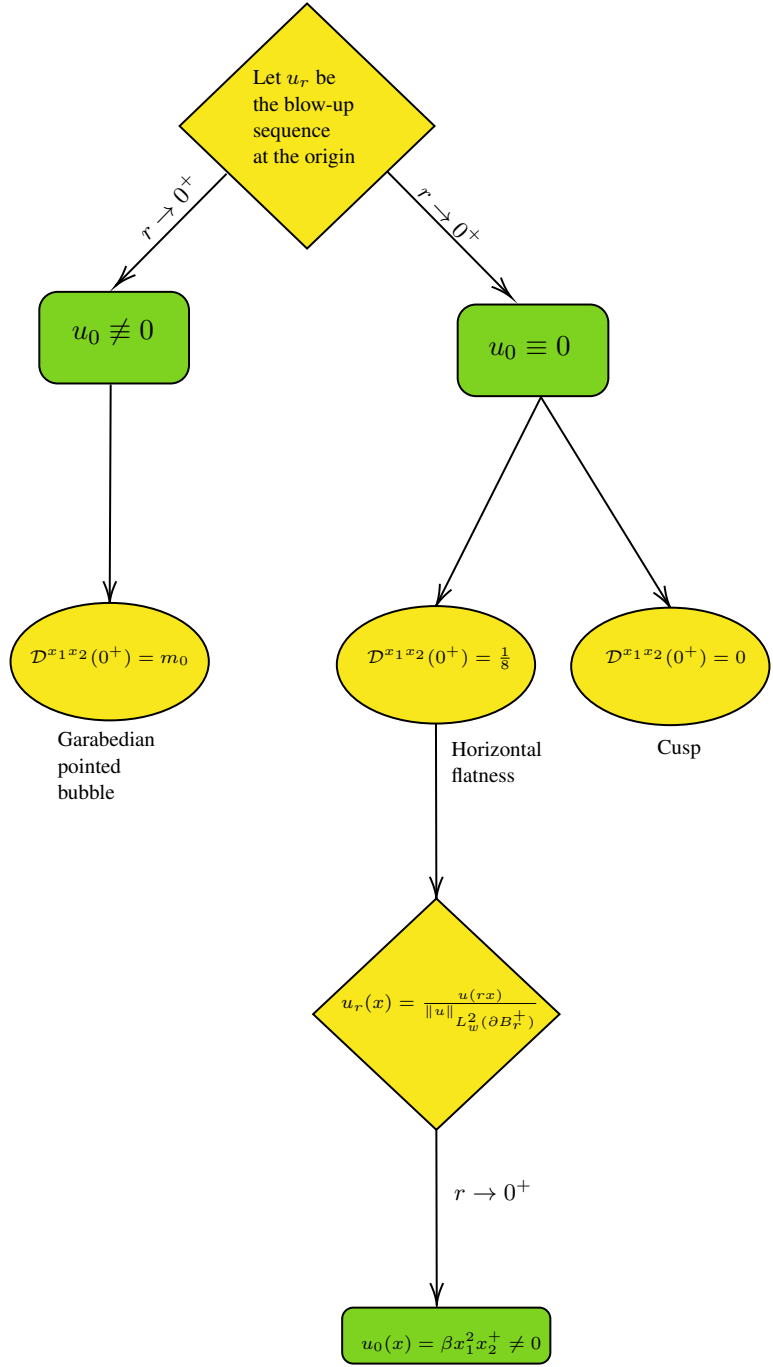
\begin{figure}[!ht]
    \center
    \tikzset{every picture/.style={line width=0.75pt}} 

    \begin{tikzpicture}[x=0.75pt,y=0.75pt,yscale=-1,xscale=1]
    
    \draw  [fill={rgb, 255:red, 248; green, 231; blue, 28 }  ,fill opacity=1 ] (321.8,18.75) -- (386.1,80.6) -- (321.8,142.45) -- (257.5,80.6) -- cycle ;
    \draw    (362.35,103.7) -- (423.5,165.78) ;
    \draw [shift={(424.9,167.2)}, rotate = 225.43] [color={rgb, 255:red, 0; green, 0; blue, 0 }  ][line width=0.75]    (10.93,-3.29) .. controls (6.95,-1.4) and (3.31,-0.3) .. (0,0) .. controls (3.31,0.3) and (6.95,1.4) .. (10.93,3.29)   ;
    \draw    (281.2,104.7) -- (226.61,159.88) ;
    \draw [shift={(225.2,161.3)}, rotate = 314.69] [color={rgb, 255:red, 0; green, 0; blue, 0 }  ][line width=0.75]    (10.93,-3.29) .. controls (6.95,-1.4) and (3.31,-0.3) .. (0,0) .. controls (3.31,0.3) and (6.95,1.4) .. (10.93,3.29)   ;
    \draw  [fill={rgb, 255:red, 126; green, 211; blue, 33 }  ,fill opacity=1 ] (186.75,173.55) .. controls (186.75,168.44) and (190.89,164.3) .. (196,164.3) -- (253,164.3) .. controls (258.11,164.3) and (262.25,168.44) .. (262.25,173.55) -- (262.25,201.3) .. controls (262.25,206.41) and (258.11,210.55) .. (253,210.55) -- (196,210.55) .. controls (190.89,210.55) and (186.75,206.41) .. (186.75,201.3) -- cycle ;
    \draw  [fill={rgb, 255:red, 126; green, 211; blue, 33 }  ,fill opacity=1 ] (397.75,179.95) .. controls (397.75,174.84) and (401.89,170.7) .. (407,170.7) -- (464.5,170.7) .. controls (469.61,170.7) and (473.75,174.84) .. (473.75,179.95) -- (473.75,207.7) .. controls (473.75,212.81) and (469.61,216.95) .. (464.5,216.95) -- (407,216.95) .. controls (401.89,216.95) and (397.75,212.81) .. (397.75,207.7) -- cycle ;
    \draw    (222.75,211.05) -- (222.26,319.3) ;
    \draw [shift={(222.25,321.3)}, rotate = 270.26] [color={rgb, 255:red, 0; green, 0; blue, 0 }  ][line width=0.75]    (10.93,-3.29) .. controls (6.95,-1.4) and (3.31,-0.3) .. (0,0) .. controls (3.31,0.3) and (6.95,1.4) .. (10.93,3.29)   ;
    \draw    (439.75,217.45) -- (387.7,321.11) ;
    \draw [shift={(386.8,322.9)}, rotate = 296.66] [color={rgb, 255:red, 0; green, 0; blue, 0 }  ][line width=0.75]    (10.93,-3.29) .. controls (6.95,-1.4) and (3.31,-0.3) .. (0,0) .. controls (3.31,0.3) and (6.95,1.4) .. (10.93,3.29)   ;
    \draw  [fill={rgb, 255:red, 248; green, 231; blue, 28 }  ,fill opacity=1 ] (172.25,350.95) .. controls (172.25,334.57) and (194.64,321.3) .. (222.25,321.3) .. controls (249.86,321.3) and (272.25,334.57) .. (272.25,350.95) .. controls (272.25,367.33) and (249.86,380.6) .. (222.25,380.6) .. controls (194.64,380.6) and (172.25,367.33) .. (172.25,350.95) -- cycle ;
    \draw    (439.75,217.45) -- (503.95,322) ;
    \draw [shift={(505,323.7)}, rotate = 238.45] [color={rgb, 255:red, 0; green, 0; blue, 0 }  ][line width=0.75]    (10.93,-3.29) .. controls (6.95,-1.4) and (3.31,-0.3) .. (0,0) .. controls (3.31,0.3) and (6.95,1.4) .. (10.93,3.29)   ;
    \draw  [fill={rgb, 255:red, 248; green, 231; blue, 28 }  ,fill opacity=1 ] (336.8,352.55) .. controls (336.8,336.17) and (359.19,322.9) .. (386.8,322.9) .. controls (414.41,322.9) and (436.8,336.17) .. (436.8,352.55) .. controls (436.8,368.93) and (414.41,382.2) .. (386.8,382.2) .. controls (359.19,382.2) and (336.8,368.93) .. (336.8,352.55) -- cycle ;
    \draw  [fill={rgb, 255:red, 248; green, 231; blue, 28 }  ,fill opacity=1 ] (455,353.35) .. controls (455,336.97) and (477.39,323.7) .. (505,323.7) .. controls (532.61,323.7) and (555,336.97) .. (555,353.35) .. controls (555,369.73) and (532.61,383) .. (505,383) .. controls (477.39,383) and (455,369.73) .. (455,353.35) -- cycle ;
    \draw    (386.8,382.2) -- (387,468.5) ;
    \draw [shift={(387,470.5)}, rotate = 269.87] [color={rgb, 255:red, 0; green, 0; blue, 0 }  ][line width=0.75]    (10.93,-3.29) .. controls (6.95,-1.4) and (3.31,-0.3) .. (0,0) .. controls (3.31,0.3) and (6.95,1.4) .. (10.93,3.29)   ;
    \draw  [fill={rgb, 255:red, 248; green, 231; blue, 28 }  ,fill opacity=1 ] (387,470.5) -- (449.15,528.88) -- (387,587.25) -- (324.85,528.88) -- cycle ;
    \draw    (387,587.25) -- (387.2,673.55) ;
    \draw [shift={(387.2,675.55)}, rotate = 269.87] [color={rgb, 255:red, 0; green, 0; blue, 0 }  ][line width=0.75]    (10.93,-3.29) .. controls (6.95,-1.4) and (3.31,-0.3) .. (0,0) .. controls (3.31,0.3) and (6.95,1.4) .. (10.93,3.29)   ;
    \draw  [fill={rgb, 255:red, 126; green, 211; blue, 33 }  ,fill opacity=1 ] (339.5,682.65) .. controls (339.5,679.53) and (342.03,677) .. (345.15,677) -- (438.85,677) .. controls (441.97,677) and (444.5,679.53) .. (444.5,682.65) -- (444.5,699.6) .. controls (444.5,702.72) and (441.97,705.25) .. (438.85,705.25) -- (345.15,705.25) .. controls (342.03,705.25) and (339.5,702.72) .. (339.5,699.6) -- cycle ;
    
    \draw (293.05,50.7) node [anchor=north west][inner sep=0.75pt]  [font=\scriptsize] [align=left] {Let $\displaystyle u_{r}$ be \\the blow-up\\sequence\\at the origin};
    \draw (232.15,131.52) node [anchor=north west][inner sep=0.75pt]  [font=\footnotesize,rotate=-314.09]  {$r\rightarrow 0^{+}$};
    \draw (384.46,105.97) node [anchor=north west][inner sep=0.75pt]  [font=\footnotesize,rotate=-46.67]  {$r\rightarrow 0^{+}$};
    \draw (201.65,176.7) node [anchor=north west][inner sep=0.75pt]    {$u_{0} \not\equiv 0$};
    \draw (411.75,185.1) node [anchor=north west][inner sep=0.75pt]    {$u_{0} \equiv 0$};
    \draw (180.65,341.55) node [anchor=north west][inner sep=0.75pt]  [font=\tiny]  {$\mathcal{D}^{x_{1} x_{2}}( 0^{+}) =m_{0}$};
    \draw (350.65,341.55) node [anchor=north west][inner sep=0.75pt]  [font=\tiny]  {$\mathcal{D}^{x_{1} x_{2}}( 0^{+}) =\frac{1}{8}$};
    \draw (393.15,389.6) node [anchor=north west][inner sep=0.75pt]  [font=\scriptsize] [align=left] {Horizontal\\flatness};
    \draw (470.75,341.55) node [anchor=north west][inner sep=0.75pt]  [font=\tiny]  {$\mathcal{D}^{x_{1} x_{2}}( 0^{+}) =0$};
    \draw (497.05,389) node [anchor=north west][inner sep=0.75pt]  [font=\scriptsize] [align=left] {Cusp};
    \draw (194.55,385) node [anchor=north west][inner sep=0.75pt]  [font=\scriptsize] [align=left] {Garabedian\\pointed\\bubble};
    \draw (337.5,514.65) node [anchor=north west][inner sep=0.75pt]  [font=\tiny]  {$u_{r}( x) =\frac{u( rx)}{\|u\|_{L_{\mathit{w}}^{2}( \partial B_{r}^{+})}}$};
    \draw (347.5,687.4) node [anchor=north west][inner sep=0.75pt]  [font=\tiny]  {$u_{0}( x) =\beta x_{1}^{2} x_{2}^{+} \neq 0$};
    \draw (394.5,610.9) node [anchor=north west][inner sep=0.75pt]  [font=\footnotesize]  {$r\rightarrow 0^{+}$};

    \end{tikzpicture}
    \caption{The analysis flow at the origin}
    \label{fig: ori}
\end{figure}
Analogous to the analysis for axisymmetric degenerate points, the convergence established in~\thmref{thm:orig} holds within  $W_{\mathit{w},\mathit{loc}}^{1,2}(\mathbb{R} _{+}^{2})$ and locally uniform in $\mathbb{R} _{+}^{2}$. Furthermore, the rescaling employed in Equation~\eqref{thm: org} is indicated by the degeneracy of the free boundary condition. The trichotomy framework (illustrated in Figure~\ref{fig: ori}) parallels prior classifications. The origin are categorized into non-trivial origin and trivial origin classes. Within the trivial origin class, a further distinction is made between trivial origin with non-zero density and trivial origin with zero density, which are differentiated by their weighted density (see Figure~\ref{Classori}). We note that the novel contribution of our analysis for nontrivial degenerate axis stagnation points in the introduction of a nonlinear frequency formula and the use of a compensated compactness result for compressible Euler equations. We defer this discussion to the subsequent section, where we outline the general approach to the proofs of~\thmref{thm:stp},~\thmref{thm:ad} and~\thmref{thm:orig}. 

\begin{figure}[!ht]
    \center
    \tikzset{every picture/.style={line width=0.75pt}} 

    \begin{tikzpicture}[x=0.75pt,y=0.75pt,yscale=-1,xscale=1]
    
    \draw  [fill={rgb, 255:red, 248; green, 231; blue, 28 }  ,fill opacity=1 ] (89.5,126) .. controls (89.5,122.41) and (92.41,119.5) .. (96,119.5) -- (133.5,119.5) .. controls (137.09,119.5) and (140,122.41) .. (140,126) -- (140,145.5) .. controls (140,149.09) and (137.09,152) .. (133.5,152) -- (96,152) .. controls (92.41,152) and (89.5,149.09) .. (89.5,145.5) -- cycle ;
    \draw    (140,137) -- (163,137.25) ;
    \draw    (163,137.25) -- (218.5,72.25) ;
    \draw    (163,137.25) -- (221.5,198.75) ;
    
    \draw  [fill={rgb, 255:red, 126; green, 211; blue, 33 }  ,fill opacity=1 ] (219,56.18) .. controls (219,51.66) and (222.66,48) .. (227.18,48) -- (276.82,48) .. controls (281.34,48) and (285,51.66) .. (285,56.18) -- (285,86.57) .. controls (285,91.09) and (281.34,94.75) .. (276.82,94.75) -- (227.18,94.75) .. controls (222.66,94.75) and (219,91.09) .. (219,86.57) -- cycle ;
    \draw  [fill={rgb, 255:red, 126; green, 211; blue, 33 }  ,fill opacity=1 ] (227.5,186.66) .. controls (227.5,182.43) and (230.93,179) .. (235.16,179) -- (276.84,179) .. controls (281.07,179) and (284.5,182.43) .. (284.5,186.66) -- (284.5,215.09) .. controls (284.5,219.32) and (281.07,222.75) .. (276.84,222.75) -- (235.16,222.75) .. controls (230.93,222.75) and (227.5,219.32) .. (227.5,215.09) -- cycle ;
    \draw    (287.5,200) -- (310.5,200.25) ;
    \draw    (310.5,200.25) -- (366,135.25) ;
    \draw    (310.5,200.25) -- (369,261.75) ;
    
    \draw  [fill={rgb, 255:red, 248; green, 231; blue, 28 }  ,fill opacity=1 ] (370,118.6) .. controls (370,110.26) and (376.76,103.5) .. (385.1,103.5) -- (453.4,103.5) .. controls (461.74,103.5) and (468.5,110.26) .. (468.5,118.6) -- (468.5,163.9) .. controls (468.5,172.24) and (461.74,179) .. (453.4,179) -- (385.1,179) .. controls (376.76,179) and (370,172.24) .. (370,163.9) -- cycle ;
    \draw  [fill={rgb, 255:red, 248; green, 231; blue, 28 }  ,fill opacity=1 ] (374,254.4) .. controls (374,245.89) and (380.89,239) .. (389.4,239) -- (453.6,239) .. controls (462.11,239) and (469,245.89) .. (469,254.4) -- (469,300.6) .. controls (469,309.11) and (462.11,316) .. (453.6,316) -- (389.4,316) .. controls (380.89,316) and (374,309.11) .. (374,300.6) -- cycle ;
    
    \draw (96.5,129) node [anchor=north west][inner sep=0.75pt]  [font=\scriptsize] [align=left] {Origin};
    \draw (223.5,55.77) node [anchor=north west][inner sep=0.75pt]  [font=\tiny] [align=left] {Non-trivial \\origin\\$\displaystyle u_{0} \not\equiv 0$};
    \draw (231.16,187) node [anchor=north west][inner sep=0.75pt]  [font=\tiny] [align=left] {Trivial \\origin:\\ $\displaystyle u_{0} \equiv 0$};
    \draw (374.9,108.07) node [anchor=north west][inner sep=0.75pt]  [font=\tiny] [align=left] {Trivial origin \\with \\non-zero \\density:\\$\displaystyle \mathcal{D}^{x_{1}x_{2}}( 0^{+}) =\frac{2}{3}$};
    \draw (379.3,247.47) node [anchor=north west][inner sep=0.75pt]  [font=\tiny] [align=left] {Trivial origin \\with \\zero \\density:\\$\displaystyle \mathcal{D}^{x_{1}x_{2}}( 0^{+}) =0$};

    \end{tikzpicture}
    \caption{Possible cases at the origin}
    \label{Classori}
\end{figure}
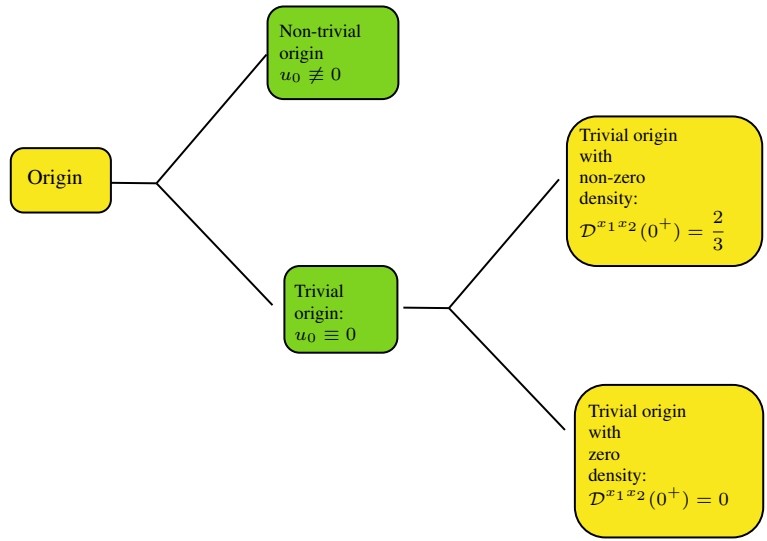

We now highlight the main contributions of this work. This paper focuses on analyzing singular profiles of the free boundary near degenerate points in the context of one-phase quasilinear Bernoulli free boundary problems. Here, a degenerate point refers to a free boundary point where the gradient of the solution vanishes. We summarize as follows.
\begin{itemize}
    \item The study on degenerate points for Bernoulli-type free boundary problems arises naturally in physical situations, the most famous one is in two dimensional gravity water waves. A notable connection exists with the Stokes conjecture—a seminal problem posed by G.G Stokes in 1880 \cite{MR2858161}—which remains deeply intertwined with the analysis of degenerate free boundary behavior (see \cite{MR4238496} for an overview). The conjecture was resolved for the incompressible gravity waves in \cite{MR666110,MR2810856}, and subsequent work extended its implications to more complex physical scenarios~\cite{MR2995099,MR3225630,MR4595616}.
    
    However, the previously mentioned works have been restricted to the incompressible case, with few considerations for the compressible one. To the best of our knowledge, only the recent work \cite{du2024proofstokesconjecturecompressible} addresses degenerate points in two dimensional compressible gravity water waves, but no existing literature treats the Stokes conjecture for compressible flows in higher dimensional models (e.g. three dimensional axisymmetric case). From this perspective, our results establish the first analysis of the Stokes conjecture for compressible axisymmetric flows in the influence of the gravity.
    \item In this paper, we develop for the first time several  monotonicity formulas for quasilinear elliptic equations of the Bernoulli type. They are particularly constructed for stagnation points away from the axis of symmetry, for  degenerate points on the axis that are non-stagnation and for the origin. The monotonicity formula for stagnation points  can be viewed as an extension of the two dimensional compressible case~\cite{du2024proofstokesconjecturecompressible} to the three dimensional axisymmetric case. We point out that the monotonicity formulas for axis free boundary points that are non-stagnation and for the origin are novel in the context for monotonicity formulas for Bernoulli-type free boundary problems~\cite{MR1620644,MR1759450}. Furthermore, we establish a new frequency formula which deals degenerate stagnation points at the origin with nontrivial density, together with the compensated compactness for compressible Euler system, we prove the strong convergence of subsonic variational solutions. The elegant mathematical theory of compensated compactness method was developed by Murat~\cite{MR506997} and Tartar~\cite{MR584398} in 1970s to solve the weak convergence in nonlinear partial differential equations. This compactness argument helps us to avoid possible concentration issues that can lead to the failure of strong convergence.
    \item Our results establish a complete, mutually exclusive classification of blow-up limits for three distinct types of degenerate points on the free boundary of the quasilinear Bernoulli free boundary problem \eqref{fb}. Such classifications are fundamental to the study of free boundary problems. In fact, the classification of mutually exclusive blow-up limits is core in the area of  free boundary problems, not only in Bernoulli-type \cite{MR1759450}, but also in other types of free boundary problems, like obstacle problems \cite{MR454350,MR3904453,MR4848670}. The classification of the blow-up limits in this paper paves the way for investigating finer geometric properties of the free boundary of the problem~\eqref{fb}, such as rectifiability, etc.
\end{itemize}
\subsubsection*{Related literatures on axisymmetric free surface flows.}
We conclude the introduction with a brief review of mathematical results pertaining to axisymmetric free surface flows. The axisymmetric model serves as an effective framework for studying the dynamics of three dimensional waves \cite{MR112435}. This model simplifies the complex three dimensional problem by reducing it to a two dimensional setting while retaining key features of the original system. Consequently, it has been a valuable tool in the analysis of free surface flows, enabling mathematicians to derive significant insights into the behavior of such flows. The earliest study of axisymmetric motion can be traced back to the work of Gilbarg and Serrin \cite{MR67650}, who investigated the uniqueness of axisymmetric subsonic flows past a given body. This pioneering research laid the foundation for subsequent studies in the field. Following their work, numerous existence and uniqueness results for global axisymmetric solutions through an infinitely long nozzle have been established. Notably, Xie and Xin \cite{MR2375709,MR2644144,MR2607929} demonstrate the existence and uniqueness of such solutions under certain upstream conditions. Additionally,  quantitative properties for solutions of axisymmetric flows in infinitely long nozzle was investigated in \cite{MR2737815,MR3537905}.
From the perspective of free boundary problems, Alt, Caffarelli, and Friedman investigated axially symmetric jets with gravity for irrotational incompressible waves \cite{MR647374,MR682265,MR642623} (see \cite{MR702728} for axisymmetric rotational flows). They established the existence and uniqueness of the jet and addressed the regularity of the free boundary, building upon the classical work of Alt and Caffarelli \cite{MR618549}. Their ideas have been extensively developed, leading to numerous researches on this topic. For further reading on incompressible axisymmetric free boundary problems, we refer the reader to the works \cite{MR4127792,MR4246821,MR4595616,MR4808256}.
\section{Strategy of the proof}
In this section, we outline the basic idea and the fundamental approach to proving Theorems \ref{thm:stp}, \ref{thm:ad}, and \ref{thm:orig}. Firstly, it should be noted that the free boundary problem \eqref{fb} admits a variational structure \cite{MR752578,MR772122}, suggesting that solutions to the problem~\eqref{fb} can be regarded as critical points of the corresponding energy functional. The minimizers of the energy functional can easily be shown to exist for the energy functional. However, we focus not on minimizers but on subsonic variational solutions of the problem~\eqref{fb}, since non-minimizers are better suited for analyzing singularities in free boundary problems \cite{MR2915865,MR4238496}. These solutions satisfy the equation in the sense of first domain variations (cf.~\eqref{fv}), while additionally enforcing conditions such as nonnegativity, a uniform subsonicity constraint, and compatibility conditions for $\nabla u/x_{1}$ along the axis of symmetry. We then aim to classify subsonic variational solutions at three distinct types of free boundary points. However, we encounter an initial challenge: the governing equation for axisymmetric compressible irrotational flows is quasilinear, and thus the associated energy functional is nonhomogeneous (cf. Equation \eqref{eng}). To overcome this difficulty, we decompose the energy into its homogeneous component (cf. Equation~\eqref{engEH}). The key contribution of this work is the discovery that:
\begin{enumerate}
    \item The homogeneous component \textit{enables the construction of a monotonicity formula} while preserving the scaling invariance of the equation, thereby enabling blow-up analysis.
    \item The discrepancy between the homogeneous component and the original energy is controlled near all degenerate points.
\end{enumerate}
Motivated by the recent work of~\cite{du2024proofstokesconjecturecompressible}, we introduce the Weiss boundary-adjusted $L^{2}$ energy 
\begin{equation*}
    J(r):=\int_{\partial B_{r}^{+}(x^{\circ})}\frac{u^{2}}{x_{1}\bar{\rho}_{0}}\,d\mathcal{H}^{1}
\end{equation*}
to facilitate our analysis. Based on these key observations, we establish  three distinct monotonicity formulas (presented in Proposition~\ref{propstg},~\ref{propadp} and~\ref{prop0}), each worked for the specific geometry of stagnation points, axisymmetric degenerate points, and the origin, respectively. The monotonicity formulas serve as preliminary tools for blow-up analysis. Under the growth assumption we imposed on each specific degenerate points, we are able to prove that the homogeneity of blow-up limits (presented in Lemma~\ref{lemstp},~\ref{lemad} and~\ref{lemo}). We now choose the analysis for the non-stagnation axis  points for example. We classify each axis stagnation points as non-stagnation degenerate stagnation points or axis non-degenerate stagnation points, and we study specifically the non-trivial degenerate axis stagnation points. In this case, we study 
the rescaled solutions of the new form
\begin{align*}
	\psi_{r}(x):=\frac{\psi(x^{\circ}+rx)}{\|\psi\|_{L_{\mathit{w}}^{2}(\partial B_{r}^{+})} }.
\end{align*}
We aim to study the limit of $\psi_{r}$ as $r\to 0^{+}$. To do so, we first construct a new frequency formula (presented in~\lemref{lemma:freq}) to examine the compactness of $\psi_{r}$. We apply our frequency formula to prove that $\psi_{r}$ is bounded in $W_{\mathit{w}}^{1,2}(B_{1})$ (presented in~\propref{prop:propfre} and~\propref{prop:freq}). This allows $\psi_{r}$ to have a weak limit $\psi_{0}$ in the space $W_{\mathit{w}}^{1,2}(B_{1})$. However, in order to pass to the limit in the domain variation formula for $\psi_{r}$, we need the strong convergence. Here we encounter our another obstacle: how to improve the weak convergence of $\psi_{r}$ to strong convergence? In the case of incompressible axisymmetric flows, the concentration compactness method is applied~\cite{MR1176677}, which relies on the structure of incompressible Euler equations and achieves convergence of the non--linear quadratic terms $\operatorname{div}((u_{1}^{m},u_{2}^{m})\otimes(u_{1}^{m},u_{2}^{m}))$ in the sense of distributions as $m\to+\infty$. To preserve such property in compressible flows, we adopt the ideas of compensated compactness for the compressible Euler system~\cite{MR2291790}. The elegant mathematical theory of compensated compactness method was developed by Murat~\cite{MR506997} and Tartar~\cite{MR584398} in 1970s to solve nonlinear partial differential equations. This compactness argument helps us to avoid possible concentration issues that can lead to the failure of strong convergence. The strong convergence of $\psi_{r}$ to $\psi_{0}$ in $W_{\mathit{w}}^{1,2}(B_{1}^{+}\setminus\{0\})$ gives us the governing equation of $\psi_{0}$ and proves the last part of~\thmref{thm:orig}.
\section{Monotonicity formulas for axisymmetric compressible problems}
In this section, we present the primary tools for analyzing stagnation points and degenerate points of the problem~\eqref{fb}. These tools are monotonicity formulas. Prior to introducing them, we establish several preparatory identities.
\begin{lemma}[Poho\v{z}aev-type identity]
    Let $x^{\circ}\in \Omega$ be such that $x_{1}^{\circ}x_{2}^{\circ}=0$, and let 
    \begin{align}\label{d}
        \delta=\left\{
            \begin{alignedat}{2}
                &\min\{|x_{1}^{\circ}|,\operatorname{dist}(x^{\circ},\partial\Omega)\}/2\qquad&&\text{ if }x_{1}^{\circ}>0,\\
                &\operatorname{dist}(x^{\circ},\partial\Omega)/2\qquad&&\text{ if }x_{1}^{\circ}=0.
            \end{alignedat}
        \right.
    \end{align}
    Then the following identities hold for a.e. $r\in(0,\delta)$:

    \noindent \textup{(1).} If $x_{1}^{\circ}\neq 0$ and $x_{2}^{\circ}=0$, then 
    \begin{align}\label{tys}
        \begin{split}
            &3E_{F}(u;B_{r}(x^{\circ}))-rE_{F}(u;\partial B_{r}(x^{\circ}))\\
            &=3\int_{B_{r}(x^{\circ})}\frac{|\nabla u|^{2}}{x_{1}H(\tfrac{|\nabla u|^{2}}{x_{1}^{2}};x_{2})}\,dx-2r\int_{\partial B_{r}(x^{\circ})}\frac{(\nabla u\cdot\nu)^{2}}{x_{1}H(\tfrac{|\nabla u|^{2}}{x_{1}^{2}};x_{2})}\,d\mathcal{H}^{1}\\
            &-\sum_{i=1}^{3}K_{i}^{x_{2}}(r),
        \end{split}
    \end{align}
    where 
    \begin{align}\label{K1x2r}
        \begin{split}
            K_{1}^{x_{2}}(r)&=\int_{B_{r}(x^{\circ})}x_{1}\Bigg[ \int_{0}^{\tfrac{|\nabla u|^{2}}{x_{1}^{2}}}\tfrac{\partial}{\partial\tau}\left( \tfrac{1}{H(\tau;x_{2})} \right)\tau\,d\tau\\
            &\qquad\qquad\qquad-\left(  \int_{0}^{x_{2}}\tfrac{\partial}{\partial\tau}\left( \tfrac{1}{H(\tau;x_{2})} \right)\tau\,d\tau\right)\chi_{\left\{ u>0 \right\} }\Bigg]\,dx,
        \end{split}
    \end{align}
    \begin{align}\label{K2x2r}
        \begin{split}
            K_{2}^{x_{2}}(r)&=\int_{B_{r}(x^{\circ})}x_{1}\Bigg[ \int_{0}^{\tfrac{|\nabla u|^{2}}{x_{1}^{2}}}\tfrac{\partial}{\partial x_{2}}\left( \tfrac{1}{H(\tau;x_{2})} \right)x_{2}\,d\tau\\
            &\qquad\qquad\qquad-\left( \int_{0}^{x_{2}}\tfrac{\partial}{\partial x_{2}}\left( \tfrac{1}{H(\tau;x_{2})} \right)x_{2}\,d\tau \right)\chi_{\left\{ u>0 \right\} } \Bigg]\,dx,
        \end{split}
    \end{align}
    and
    \begin{align}\label{K3x2r}
        \begin{split}
            K_{3}^{x_{2}}(r)&=\int_{B_{r}(x^{\circ})}(x_{1}-x_{1}^{\circ})\Bigg[ F(\tfrac{|\nabla u|^{2}}{x_{1}^{2}};x_{2})-\frac{2}{H(\tfrac{|\nabla u|^{2}}{x_{1}^{2}};x_{2})}\frac{|\nabla u|^{2}}{x_{1}^{2}}\\
            &\qquad\qquad\qquad\qquad\quad+\lambda(x_{2})\chi_{\left\{ u>0 \right\} } \Bigg]\,dx.
        \end{split}
    \end{align}
    \noindent \textup{(2).} If $x_{1}^{\circ}=0$ and $x_{2}^{\circ}\neq 0$, then 
    \begin{align}\label{tya}
        \begin{split}
            &3E_{F}(u;B_{r}^{+}(x^{\circ}))-rE_{F}(u;\partial B_{r}^{+}(x^{\circ}))\\
            &=4\int_{B_{r}^{+}(x^{\circ})}\frac{|\nabla u|^{2}}{x_{1}H(\tfrac{|\nabla u|^{2}}{x_{1}^{2}};x_{2})}\,dx-2r\int_{\partial B_{r}^{+}(x^{\circ})}\frac{(\nabla u\cdot\nu)^{2}}{x_{1}H(\tfrac{|\nabla u|^{2}}{x_{1}^{2}};x_{2})}\,d\mathcal{H}^{1}\\
            &-K_{1}^{x_{1}}(r),
        \end{split}
    \end{align}
    where 
    \begin{equation}\label{K1x1r}
        K_{1}^{x_{1}}(r)=\int_{B_{r}^{+}(x^{\circ})}x_{1}(x_{2}-x_{2}^{\circ})\left[ \partial_{2}F(\tfrac{|\nabla u|^{2}}{x_{1}^{2}};x_{2})+\lambda'(x_{2})\chi_{\left\{ u>0 \right\} } \right]\,dx.
    \end{equation}
    \noindent \textup{(3).} If $x_{1}^{\circ}=x_{2}^{\circ}=0$, then 
    \begin{align}\label{ty0}
        \begin{split}
            &4E_{F}(u;B_{r}^{+})-rE_{F}(u;\partial B_{r}^{+})\\
            &=5\int_{ B_{r}^{+}}\frac{|\nabla u|^{2}}{x_{1}H(\tfrac{|\nabla u|^{2}}{x_{1}^{2}};x_{2})}\,dx-2r\int_{\partial B_{r}^{+}}\frac{\left( \nabla u\cdot\nu \right)^{2}}{x_{1}H(\tfrac{|\nabla u|^{2}}{x_{1}^{2}};x_{2})}\,d\mathcal{H}^{1}\\
            &-K_{1}^{x_{2}}(r)-K_{1}^{x_{1}x_{2}}(r),
        \end{split}
    \end{align}
    where 
    \begin{equation}\label{K1x12r}
        K_{1}^{x_{1}x_{2}}(r)=\int_{B_{r}^{+}}x_{1}x_{2}\left[ \partial_{2}F(\tfrac{|\nabla u|^{2}}{x_{1}^{2}};x_{2})+\left(\lambda'(x_{2})-\frac{1}{\bar{\rho}_{0}}\right)\chi_{\left\{ u>0 \right\} } \right]\,dx.
    \end{equation}

    \noindent In the equations~\eqref{tys}, \eqref{tya} and \eqref{ty0}, $E_{F}(u;B_{r}^{+}(x^{\circ}))$ is defined in equation \eqref{eng} with $\Omega$ replaced by $B_{r}^{+}(x^{\circ})$ (with respect to $B_{r}(x^{\circ})$ if $x_{1}^{\circ}\neq 0$). Similarly, $E_{F}(u;\partial B_{r}^{+}(x^{\circ}))$ is defined by 
    \begin{equation*}
        E_{F}(u;\partial B_{r}^{+}(x^{\circ}))=\int_{\partial B_{r}^{+}(x^{\circ})}x_{1}\left[ F(\tfrac{|\nabla u|^{2}}{x_{1}^{2}};x_{2})+\lambda(x_{2})\chi_{\left\{ u>0 \right\} } \right]\,d\mathcal{H}^{1}.
    \end{equation*}    
\end{lemma}
\begin{proof}
    Let us use $\phi(x):=(x-x^{\circ})\eta(|x-x^{\circ}|)$ as the test function in the first domain variation formula~\eqref{fv}, where $\eta\in C^{0}([0,r])\cap C^{1}([0,r))$, $\eta(r)=0$ is a cut off function. A simple computation gives 
    \begin{equation*}
        \operatorname{div}\phi=2\eta(|x-x^{\circ}|)+|x-x^{\circ}|\eta'(|x-x^{\circ}|),
    \end{equation*}
    and 
    \begin{equation*}
        D\phi= \begin{pmatrix}
            \eta(|x-x^{\circ}|)&0\\
            0&\eta(|x-x^{\circ}|)
        \end{pmatrix} +\eta'(|x-x^{\circ}|)(x-x^{\circ})\otimes(x-x^{\circ}).
    \end{equation*}
    Consequently, we have 
    \begin{align*}
        0&=2\int_{\Omega}\left[ x_{1}F( \tfrac{|\nabla u|^{2}}{x_{1}^{2}};x_{2} ) +x_{1}\lambda(x_{2})\chi_{\left\{ u>0 \right\} } \right]\eta(|x-x^{\circ}|)\,dx\\
        &+\int_{\Omega}|x-x^{\circ}|\left[ x_{1}F( \tfrac{|\nabla u|^{2}}{x_{1}^{2}};x_{2} ) +x_{1}\lambda(x_{2})\chi_{\left\{ u>0 \right\} } \right]\eta'(|x-x^{\circ}|) \,dx\\
        &-2\int_{\Omega}\frac{|\nabla u|^{2}\eta(|x-x^{\circ}|)}{x_{1}H( \tfrac{|\nabla u|^{2}}{x_{1}^{2}};x_{2} )}\,dx\\
        &-2\int_{\Omega}\frac{\eta'(|x-x^{\circ}|)|x-x^{\circ}|}{x_{1}H( \tfrac{|\nabla u|^{2}}{x_{1}^{2}};x_{2} )}\left( \nabla u\cdot\frac{x-x^{\circ}}{|x-x^{\circ}|} \right)^{2}\,dx\\
        &+\int_{\Omega}(x_{1}-x_{1}^{\circ})\left[ F(\tfrac{|\nabla u|^{2}}{x_{1}^{2}};x_{2})-2\frac{1}{H(\tfrac{|\nabla u|^{2}}{x_{1}^{2}};x_{2})}\frac{|\nabla u|^{2}}{x_{1}^{2}} \right]\eta(|x-x^{\circ}|)\,dx\\
        &+\int_{\Omega}(x_{1}-x_{1}^{\circ})\eta(|x-x^{\circ}|)\lambda(x_{2})\chi_{\left\{ u>0 \right\} }\,dx\\
        &+\int_{\Omega}x_{1}(x_{2}-x_{2}^{\circ})\left[\partial_{2}F(\tfrac{|\nabla u|^{2}}{x_{1}^{2}};x_{2})+\lambda'(x_{2})\chi_{\left\{ u>0 \right\} } \right]\eta(|x-x^{\circ}|)\,dx.
    \end{align*}
    Choose a sequence of $\eta_{m}$ with $\eta_{m}=1$ for $s\in[0,r]$, $\eta_{m}(s)=0$ for $s>r+\tfrac{1}{m}$, and $\eta_{m}(s)$ is linear between $(r,r+\tfrac{1}{m})$. Given that $\nu=\frac{x-x^{\circ}}{r}$ for $x\in \partial B_{r}^{+}(x^{\circ})$, we have that 
    \begin{align*}
        0&=2\int_{B_{r}^{+}(x^{\circ})}\left[ x_{1}F(\tfrac{|\nabla u|^{2}}{x_{1}^{2}};x_{2})+x_{1}\lambda(x_{2})\chi_{\left\{ u>0 \right\} } \right]\,dx\\
        &-m\int_{r}^{r+1/m}\int_{\partial B_{s}^{+}(x^{\circ})}s\left[ x_{1}F(\tfrac{|\nabla u|^{2}}{x_{1}^{2}};x_{2})+x_{1}\lambda(x_{2})\chi_{\left\{ u>0 \right\} } \right]\,d\mathcal{H}^{1}ds\\
        &-2\int_{B_{r}^{+}(x^{\circ})}\frac{|\nabla u|^{2}}{H(\tfrac{|\nabla u|^{2}}{x_{1}^{2}};x_{2})}\,dx\\
        &+2m\int_{r}^{r+1/m}\int_{\partial B_{s}^{+}(x^{\circ})}\frac{\left( \nabla u\cdot\nu \right)^{2}}{x_{1}H(\tfrac{|\nabla u|^{2}}{x_{1}^{2}};x_{2})}\,d\mathcal{H}^{1}ds\\
        &+\int_{B_{r}^{+}(x^{\circ})}(x_{1}-x_{1}^{\circ})\left[ F(\tfrac{|\nabla u|^{2}}{x_{1}^{2}};x_{2})-2\frac{1}{H(\tfrac{|\nabla u|^{2}}{x_{1}^{2}};x_{2})}\frac{|\nabla u|^{2}}{x_{1}^{2}}\right]\,dx\\
        &+\int_{B_{r}^{+}(x^{\circ})}(x_{1}-x_{1}^{\circ})\lambda(x_{2})\chi_{\left\{ u>0 \right\} } \,dx\\
        &+\int_{B_{r}^{+}(x^{\circ})}x_{1}(x_{2}-x_{2}^{\circ})\left[\partial_{2}F(\tfrac{|\nabla u|^{2}}{x_{1}^{2}};x_{2})+\lambda'(x_{2})\chi_{\left\{ u>0 \right\} } \right]\,dx.
    \end{align*}
    Passing to the limit as $m\to\infty$, we obtain for a.e. $r\in (0,\delta)$ that 
    \begin{align}\label{pi1}
        \begin{split}
            0&=2\int_{B_{r}^{+}(x^{\circ})}\left[ x_{1}F(\tfrac{|\nabla u|^{2}}{x_{1}^{2}};x_{2})+x_{1}\lambda(x_{2})\chi_{\left\{ u>0 \right\} } \right]\,dx\\
            &-r\int_{\partial B_{r}^{+}(x^{\circ})}\left[ x_{1}F(\tfrac{|\nabla u|^{2}}{x_{1}^{2}};x_{2})+x_{1}\lambda(x_{2})\chi_{\left\{ u>0 \right\} } \right]\,d\mathcal{H}^{1}\\
            &-2\int_{B_{r}^{+}(x^{\circ})}\frac{|\nabla u|^{2}}{x_{1}H(\tfrac{|\nabla u|^{2}}{x_{1}^{2}};x_{2})}\,dx+2r\int_{\partial B_{r}^{+}(x^{\circ})}\frac{\left( \nabla u\cdot\nu \right)^{2}}{x_{1}H(\tfrac{|\nabla u|^{2}}{x_{1}^{2}};x_{2})}\,d\mathcal{H}^{1}\\
            &+\int_{B_{r}^{+}(x^{\circ})}(x_{1}-x_{1}^{\circ})\left[F(\tfrac{|\nabla u|^{2}}{x_{1}^{2}};x_{2}) -\frac{2}{H(\tfrac{|\nabla u|^{2}}{x_{1}^{2}};x_{2})}\frac{|\nabla u|^{2}}{x_{1}^{2}} \right] \,dx\\
            &+\int_{B_{r}^{+}(x^{\circ})}(x_{1}-x_{1}^{\circ})+\lambda(x_{2})\chi_{\left\{ u>0 \right\} }\,dx\\
            &+\int_{B_{r}^{+}(x^{\circ})}x_{1}(x_{2}-x_{2}^{\circ})\left[\partial_{2}F(\tfrac{|\nabla u|^{2}}{x_{1}^{2}};x_{2})+\lambda'(x_{2})\chi_{\left\{ u>0 \right\} } \right]\,dx.
        \end{split}
    \end{align}
    We are now ready to prove identities in the  equations~\eqref{tys},~\eqref{tya} and~\eqref{ty0}.  We first prove the identity in~\eqref{tys}. Let us define  
    \begin{equation}\label{engEH}
        E_{H}(u;B_{r}^{+}(x^{\circ}))=\int_{B_{r}^{+}(x^{\circ})}x_{1}\left[ \frac{|\nabla u|^{2}}{x_{1}^{2}H(\tfrac{|\nabla u|^{2}}{x_{1}^{2}};x_{2})}+\frac{x_{2}}{\bar{\rho}_{0}}\chi_{\left\{ u>0 \right\} } \right]\,dx.
    \end{equation}
    Note that $E_{H}(u;B_{r}^{+}(x^{\circ}))$ is well defined for each $r\in(0,\delta)$, since variational solution $u\in W_{\mathit{w},\mathit{loc}}^{1,2}(\Omega)$. We then deduce from the identity in~\eqref{pi1} that for a.e. $r\in(0,\delta)$,
    \begin{align}\label{tys1}
        \begin{split}
            &3E_{F}(u;B_{r}(x^{\circ}))-rE_{H}(u;B_{r}(x^{\circ}))\\
            &=3\int_{B_{r}(x^{\circ})}\frac{|\nabla u|^{2}}{x_{1}H(\tfrac{|\nabla u|^{2}}{x_{1}^{2}};x_{2})}\,dx\\
            &-2r\int_{\partial B_{r}(x^{\circ})}\frac{(\nabla u\cdot\nu)^{2}}{x_{1}H(\tfrac{|\nabla u|^{2}}{x_{1}^{2}};x_{2})}\,d\mathcal{H}^{1}\\
            &+E_{F}(u;B_{r}(x^{\circ}))-E_{H}(u;B_{r}(x^{\circ}))\\
            &-\int_{B_{r}(x^{\circ})}(x_{1}-x_{1}^{\circ})\left[ F(\tfrac{|\nabla u|^{2}}{x_{1}^{2}};x_{2})-\frac{2}{H(\tfrac{|\nabla u|^{2}}{x_{1}^{2}};x_{2})}\frac{|\nabla u|^{2}}{x_{1}^{2}}\right]\,dx\\
            &-\int_{B_{r}(x^{\circ})}(x_{1}-x_{1}^{\circ})\lambda(x_{2})\chi_{\left\{ u>0 \right\} }\,dx+\frac{1}{\bar{\rho}_{0}}\int_{B_{r}(x^{\circ})}x_{1}x_{2}\chi_{\left\{ u>0 \right\} }\,dx\\ 
            &-\int_{B_{r}(x^{\circ})}x_{1}x_{2}\left[ \partial_{2}F(\tfrac{|\nabla u|^{2}}{x_{1}^{2}};x_{2})+\lambda'(x_{2})\chi_{\left\{ u>0 \right\} } \right]\,dx.
        \end{split}
    \end{align}
    Set 
    \begin{equation}\label{defK1x2r}
        K_{1}^{x_{2}}(r):=\begin{cases}
            E_{H}(u;B_{r}^{+}(x^{\circ}))-E_{F}(u;B_{r}^{+}(x^{\circ}))&\text{ if }x_{1}^{\circ}=0,\\
            E_{H}(u;B_{r}(x^{\circ}))-E_{F}(u;B_{r}(x^{\circ})),&\text{ if }x_{1}^{\circ}\neq 0.
        \end{cases}
    \end{equation}
    and 
    \begin{equation}\label{defK2x2r}
        K_{2}^{x_{2}}(r)=\int_{B_{r}(x^{\circ})}x_{1}x_{2}\left[ \partial_{2}F(\tfrac{|\nabla u|^{2}}{x_{1}^{2}};x_{2})+\left(  \lambda'(x_{2})-\frac{1}{\bar{\rho}_{0}}\right)\chi_{\left\{ u>0 \right\} } \right]\,dx.
    \end{equation}
    From equations~\eqref{defK1x2r} and~\eqref{defK2x2r} and direct calculation, it follows that
    \begin{align}\label{EH-EF}
        \begin{split}
            K_{1}^{x_{2}}(r)&=\int_{B_{r}^{+}(x^{\circ})}x_{1}\int_{0}^{\tfrac{|\nabla u|^{2}}{x_{1}^{2}}}\pd{!}{\tau}\left( \frac{\tau}{H(\tau;x_{2})} \right)\,d\tau dx\\
            &-\int_{B_{r}^{+}(x^{\circ})}x_{1}\int_{0}^{\tfrac{|\nabla u|^{2}}{x_{1}^{2}}}\frac{1}{H(\tau;x_{2})}\,d\tau dx\\
            &-\int_{B_{r}^{+}(x^{\circ})}x_{1}\left[\int_{0}^{x_{2}}\pd{!}{\tau}\left(\frac{1}{H(\tau;x_{2})} \right) \tau\,d\tau \right]\chi_{\left\{ u>0 \right\} }\,dx\\
            &=\int_{B_{r}^{+}(x^{\circ})}x_{1}\Bigg\{ \int_{0}^{\tfrac{|\nabla u|^{2}}{x_{1}^{2}}}\tfrac{\partial}{\partial\tau}\left( \tfrac{1}{H(\tau;x_{2})} \right)\tau\,d\tau \\
            &\qquad\qquad\qquad-\left[  \int_{0}^{x_{2}}\tfrac{\partial}{\partial\tau}\left( \tfrac{1}{H(\tau;x_{2})} \right)\tau\,d\tau\right]\chi_{\left\{ u>0 \right\} }\Bigg\}\,dx,
        \end{split}
    \end{align}
    and 
    \begin{align}\label{K2x2r1}
        \begin{split}
            K_{2}^{x_{2}}(r)&=\int_{B_{r}(x^{\circ})}x_{1}x_{2}\Bigg[ \int_{0}^{\tfrac{|\nabla u|^{2}}{x_{1}^{2}}}\tfrac{\partial}{\partial x_{2}}\left( \tfrac{1}{H(\tau;x_{2})} \right)\,d\tau\\
            &\qquad\qquad\qquad-\left( \int_{0}^{x_{2}}\tfrac{\partial}{\partial x_{2}}\left( \tfrac{1}{H(\tau;x_{2})} \right)\,d\tau \right)\chi_{\left\{ u>0 \right\} } \Bigg]\,dx,
        \end{split}
    \end{align}
    where we used the equation~\eqref{lbd'}. Then the identity~\eqref{tys} is a combination of equations in~\eqref{tys1}, \eqref{defK1x2r}, \eqref{defK2x2r}, \eqref{EH-EF} and \eqref{K2x2r1}.

    Now we prove the identity in~\eqref{tya}. We observe from the identity in~\eqref{pi1} that for a.e. $r\in(0,\delta)$
    \begin{align*}
        &3E_{F}(u;B_{r}^{+}(x^{\circ}))-rE_{H}(u;\partial B_{r}^{+}(x^{\circ}))\\
        &=4\int_{B_{r}^{+}(x^{\circ})}\frac{|\nabla u|^{2}}{x_{1}H(\tfrac{|\nabla u|^{2}}{x_{1}^{2}};x_{2})}\,dx-2r\int_{\partial B_{r}^{+}(x^{\circ})}\frac{(\nabla u\cdot\nu)^{2}}{x_{1}H(\tfrac{|\nabla u|^{2}}{x_{1}^{2}};x_{2})}\,d\mathcal{H}^{1}\\
        &-\int_{B_{r}^{+}(x^{\circ})}x_{1}(x_{2}-x_{2}^{\circ})\left[\partial_{2}F(\tfrac{|\nabla u|^{2}}{x_{1}^{2}};x_{2})+\lambda'(x_{2})\chi_{\left\{ u>0 \right\} } \right] \,dx.
    \end{align*}
    This immediately proves the identity in~\eqref{tya} by defining $K_{1}^{x_{1}}(r)$ as specified in~\eqref{K1x1r}.

    We finish our proof by proving the identity~\eqref{ty0}. It follows from the identity in~\eqref{pi1} that for a.e. $r\in(0,\delta)$
    \begin{align*}
        &4E_{F}(u;B_{r}^{+})-rE_{F}(u;\partial B_{r}^{+})\\
        &=5\int_{B_{r}^{+}}\frac{|\nabla u|^{2}}{x_{1}H(\tfrac{|\nabla u|^{2}}{x_{1}^{2}};x_{2})}\,dx-2r\int_{\partial B_{r}^{+}}\frac{\left( \nabla u\cdot\nu \right)^{2}}{x_{1}H(\tfrac{|\nabla u|^{2}}{x_{1}^{2}};x_{2})}\,d\mathcal{H}^{1}\\
        &-K_{1}^{x_{2}}(r)-\int_{B_{r}^{+}}x_{1}\left(\partial_{2}F(\tfrac{|\nabla u|^{2}}{x_{1}^{2}};x_{2})+\lambda'(x_{2})\chi_{\left\{ u>0 \right\} } \right)x_{2}\,dx\\
        &+\frac{1}{\bar{\rho}_{0}}\int_{B_{r}^{+}}x_{1}x_{2}\chi_{\left\{ u>0 \right\} }\,dx.
    \end{align*}
    Define $K_{1}^{x_{1}x_{2}}(r)$ as in~\eqref{K1x12r} and we obtain~\eqref{ty0}.
\end{proof}
\begin{remark}
    Let us remark that the energy functional  $E_{H}(u; B_{r}^{+}(x^{\circ}))$ defined in equation~\eqref{engEH} plays a pivotal role for two key reasons:
    \begin{enumerate}
        \item [(1)] When $H(\tfrac{|\nabla u|^{2}}{x_{1}^{2}};x_{2})\equiv\bar{\rho}_{0}$, the functional reduces to the energy for the incompressible axisymmetric problem (cf.~\cite{MR3225630}).
        \item [(2)] Unlike the functional $E_{F}(u;B_{r}^{+}(x^{\circ}))$ defined in~\eqref{eng}, the functional $E_{H}(u;B_{r}^{+}(x^{\circ}))$ inherently satisfies a scaling invariance property. More specifically, 
        \begin{itemize}
            \item [(a).] If $x^{\circ}=(x_{1}^{\circ},0)$ with $x_{1}^{\circ}\neq 0$, then for subsonic variational solution $u\in W^{1,2}(B_{R}(x^{\circ}))$ with $R>0$, we set $u_{R}(x):=\frac{u(x^{\circ}+Rx)}{R^{3/2}}$ and we have 
            \begin{align*}
                &E_{H}(u;B_{r})\\
                &=R^{3}\int_{B_{r/R}}\frac{|\nabla u_{R}|^{2}}{(x_{1}^{\circ}+Rx_{1})H_{R}(x)}+\frac{(x_{1}^{\circ}+Rx_{1})x_{2}}{\bar{\rho}_{0}}\chi_{\left\{ u_{R}>0 \right\} }\,dx,
            \end{align*}
            where $H_{R}(x):=H(\tfrac{R|\nabla u_{R}|^{2}}{(x_{1}^{\circ}+Rx_{1})^{2}};Rx_{2})$. Define
            \begin{align*}
                &E_{H}(u_{R};B_{r/R})\\
                &:=\int_{B_{r/R}}\frac{|\nabla u_{R}|^{2}}{(x_{1}^{\circ}+Rx_{1})H_{R}(x)}+\frac{(x_{1}^{\circ}+Rx_{1})x_{2}}{\bar{\rho}_{0}}\chi_{\left\{ u_{R}>0 \right\} }\,dx.
            \end{align*} 
            In this way, we see that $E_{H}(u;B_{r})$ has the scaling property in the following sense.
            \begin{equation*}
                r^{-3}E_{H}(u;B_{r})=(r/R)^{-3}E_{H}(u_{R};B_{r/R}).
            \end{equation*}
            If we assume further that $u_{R}$ converges strongly in $W_{\mathit{loc}}^{1,2}(\mathbb{R} ^{2})$ to a blow-up limit $u_{0}$, it follows that 
            $u_{0}$ is a subsonic variational solution of 
            \begin{equation*}
                E_{H}(v):=\frac{1}{\bar{\rho}_{0}}\int_{\mathbb{R} ^{2}}\left[ \frac{|\nabla v|^{2}}{x_{1}^{\circ}}+x_{1}^{\circ}x_{2}\chi_{\left\{ v>0 \right\} }\right]\,dx.
            \end{equation*}
            This suggests that the singular asymptotics at $x^{\circ}$ exhibit Stokes-type behavior (cf.~\cite{du2024proofstokesconjecturecompressible}).
            \item [(b).] If $x^{\circ}=(0,x_{2}^{\circ})$ with $x_{2}^{\circ}\neq 0$, a similar argument as previous shows that 
            \begin{align*}
                &E_{H}(u;B_{r}^{+})\\
                &=R^{3}\int_{B_{r/R}^{+}}\frac{|\nabla u_{R}|^{2}}{x_{1}H_{R}(x)}+\frac{x_{1}(x_{2}^{\circ}+Rx_{2})}{\bar{\rho}_{0}}\chi_{\left\{ u_{R}>0 \right\} }\,dx,
            \end{align*}
            where $u_{R}:=\frac{u(x^{\circ}+Rx)}{R^{2}}$ and $H_{R}(x):=H(\tfrac{|\nabla u_{R}|^{2}}{x_{1}^{2}};x_{2}^{\circ}+Rx_{2})$ in this case. It follows that  
            \begin{equation*}
                r^{-3}E_{H}(u;B_{r}^{+})=(r/R)^{-3}E_{H}(u_{R};B_{r/R}^{+}).
            \end{equation*}
           In contrast to our recent investigation of singularities in compressible waves~\cite{du2024proofstokesconjecturecompressible}, the phenomenon under consideration represents a novel development in the study of degenerate points for axisymmetric compressible flows. We will rigorously analyze this case in~\secref{sec:ad}.
            \item [(c).] If $x^{\circ}=(0,0)$, let us consider $u_{R}(x)=\frac{u(Rx)}{R^{5/2}}$ and $H_{R}(x):=H(\tfrac{R|\nabla u_{R}|^{2}}{x_{1}^{2}};Rx_{2})$, then the corresponding situation becomes 
            \begin{equation*}
                r^{-4}E_{H}(u;B_{r}^{+})=(r/R)^{-4}E_{H}(u;B_{r/R}^{+}).
            \end{equation*}
            This particular case is also considered in~\secref{sec:0}.
        \end{itemize}
    \end{enumerate}
\end{remark}
\begin{lemma}[Energy identity]
    Let $x^{\circ}\in \Omega$, and let $\delta$ be defined as in~\eqref{d}. Then 
    \begin{equation}\label{eni}
        \int_{B_{r}^{+}(x^{\circ})}\frac{|\nabla u|^{2}}{x_{1}H(\tfrac{|\nabla u|^{2}}{x_{1}^{2}};x_{2})}\,dx=\int_{\partial B_{r}^{+}(x^{\circ})}\frac{u\nabla u\cdot\nu}{x_{1}H(\tfrac{|\nabla u|^{2}}{x_{1}^{2}};x_{2})}\,d\mathcal{H}^{1}.
    \end{equation}
\end{lemma}
\begin{proof}
    Let us consider the following two cases: 
    
    \noindent \textit{Case} 1. $x_{1}^{\circ}>0$. A direct computation gives 
    \begin{align*}
        0&=\int_{B_{r}(x^{\circ})}u \operatorname{div}\left( \frac{\nabla u}{x_{1}H(\tfrac{|\nabla u|^{2}}{x_{1}^{2}};x_{2})} \right)\,dx\\
        &=-\int_{B_{r}(x^{\circ})}\frac{|\nabla u|^{2}}{x_{1}H(\tfrac{|\nabla u|^{2}}{x_{1}^{2}};x_{2})}\,dx+\int_{\partial B_{r}(x^{\circ})}\frac{u\nabla u\cdot\nu}{x_{1}H(\tfrac{|\nabla u|^{2}}{x_{1}^{2}};x_{2})}\,d\mathcal{H}^{1}.
    \end{align*}
    \noindent\textit{Case} 2. $x_{1}^{\circ}=0$. Then 
    \begin{align*}
        0&=\int_{B_{r}^{+}(x^{\circ})}u \operatorname{div}\left( \frac{\nabla u}{x_{1}H(\tfrac{|\nabla u|^{2}}{x_{1}^{2}};x_{2})} \right)\,dx\\
        &=-\int_{B_{r}^{+}(x^{\circ})}\frac{|\nabla u|^{2}}{x_{1}H(\tfrac{|\nabla u|^{2}}{x_{1}^{2}};x_{2})}\,dx+\int_{\partial B_{r}^{+}(x^{\circ})}\frac{u\nabla u\cdot\nu}{x_{1}H(\tfrac{|\nabla u|^{2}}{x_{1}^{2}};x_{2})}\,d\mathcal{H}^{1}\\
        &-\lim_{\varepsilon\to 0^{+}}\int_{B_{r}(x^{\circ})\cap\{x_{1}=\varepsilon\}}\frac{u\partial_{1}u}{x_{1}H(\tfrac{|\nabla u|^{2}}{x_{1}^{2}};x_{2})}\,d\mathcal{H}^{1}.
    \end{align*}
    The last limit equals to zero due to the compatible condition~\eqref{com}.
\end{proof}
By employing a Poho\v{z}aev-type identity and the energy identity, we establish the following monotonicity formulas. They are presented in~\propref{propstg},~\propref{propadp} and~\propref{prop0} and work for the stagnation points, the non-stagnation axis  points and the origin, respectively.
\begin{proposition}[Monotonicity formula for the stagnation points]\label{propstg}
    Let $u$ be a subsonic variational solution to~\eqref{fb}, let $x^{\circ}=(x_{1}^{\circ},0)\in \Omega$ with $x_{1}^{\circ}\neq 0$, and let $\delta$ defined as in~\eqref{d}. Define, for any $r\in(0,\delta)$, 
    \begin{equation}\label{Ir}
        I(r)=\int_{B_{r}(x^{\circ})}x_{1}\left[ F(\tfrac{|\nabla u|^{2}}{x_{1}^{2}};x_{2})+\lambda(x_{2})\chi_{\left\{ u>0 \right\} } \right]\,dx,
    \end{equation}
    \begin{equation}\label{Jr}
        J(r)=\int_{\partial B_{r}(x^{\circ})}\frac{1}{x_{1}\bar{\rho}_{0}}u^{2}\,d\mathcal{H}^{1},
    \end{equation}
    \begin{equation}\label{Mx2r}
        M^{x_{2}}(r)=r^{-3}I(r)-\frac{3}{2}r^{-4}J(r).
    \end{equation}
    Then for a.e. $r\in(0,\delta)$,
    \begin{align}\label{Mx2r1}
        \begin{split}
            (M^{x_{2}}(r))'&=2r^{-3}\int_{\partial B_{r}(x^{\circ})}\frac{1}{x_{1}H(\tfrac{|\nabla u|^{2}}{x_{1}^{2}};x_{2})}\left( \nabla u\cdot\nu-\frac{3}{2}\frac{u}{r} \right)^{2}\,d\mathcal{H}^{1}\\
            &+\sum_{i=1}^{6}r^{-4}K_{i}^{x_{2}}(r),
        \end{split}
    \end{align}
    where $K_{i}^{x_{2}}$ for $i=1$, $2$ and $3$ are defined in equations~\eqref{K1x2r},~\eqref{K2x2r} and~\eqref{K3x2r}, respectively, and 
    \begin{equation*}
        K_{4}^{x_{2}}(r)=3\int_{\partial B_{r}(x^{\circ})}\left( \frac{1}{x_{1}H(\tfrac{|\nabla u|^{2}}{x_{1}^{2}};x_{2})}-\frac{1}{x_{1}\bar{\rho}_{0}} \right)u\nabla u\cdot\nu\,d\mathcal{H}^{1},
    \end{equation*}
    \begin{equation*}
        K_{5}^{x_{2}}(r)=\frac{9}{2r}\int_{\partial B_{r}(x^{\circ})}\left( \frac{1}{x_{1}\bar{\rho}_{0}}-\frac{1}{x_{1}H(\tfrac{|\nabla u|^{2}}{x_{1}^{2}};x_{2})} \right)u^{2}\,d\mathcal{H}^{1},
    \end{equation*}
    and 
    \begin{equation*}
        K_{6}^{x_{2}}(r)=\frac{3}{2r}\int_{\partial B_{r}(x^{\circ})}\frac{x_{1}-x_{1}^{\circ}}{x_{1}^{2}\bar{\rho}_{0}}u^{2}\,d\mathcal{H}^{1}.
    \end{equation*}
\end{proposition}
\begin{proof}
    Notice that for a.e. $r\in(0,\delta)$, 
    \begin{align*}
        (r^{-3}I(r))'&=-3r^{-4}\int_{B_{r}(x^{\circ})}x_{1}\left[ F(\tfrac{|\nabla u|^{2}}{x_{1}^{2}};x_{2})+\lambda(x_{2})\chi_{\left\{ u>0 \right\} } \right]\,dx\\
        &+r^{-3}\int_{\partial B_{r}(x^{\circ})}x_{1}\left[ F(\tfrac{|\nabla u|^{2}}{x_{1}^{2}};x_{2})+\lambda(x_{2})\chi_{\left\{ u>0 \right\} } \right]\,d\mathcal{H}^{1}\\
        &=-3 r^{-4}E_{F}(u;B_{r}(x^{\circ}))+r^{-3}E_{F}(u;\partial B_{r}(x^{\circ})).
    \end{align*}
    Thanks to the Poho\v{z}aev-type identity in~\eqref{tys}, we obtain 
    \begin{align}\label{Ir1st}
        \begin{split}
            (r^{-3}I(r))'&=2r^{-3}\int_{\partial B_{r}(x^{\circ})}\frac{(\nabla u\cdot\nu)^{2}}{x_{1}H(\tfrac{|\nabla u|^{2}}{x_{1}^{2}};x_{2})}\,d\mathcal{H}^{1}\\
            &-3r^{-4}\int_{\partial B_{r}(x^{\circ})}\frac{u\nabla u\cdot\nu}{x_{1}H(\tfrac{|\nabla u|^{2}}{x_{1}^{2}};x_{2})}\,d\mathcal{H}^{1}\\
            &+r^{-4}\sum_{i=1}^{3}K_{i}^{x_{2}}(r),
        \end{split}
    \end{align}
    where we also used~\eqref{eni}. Moreover, a direct calculation gives that 
    \begin{align}\label{Jr1st}
        \begin{split}
            (r^{-4}J(r))'&=2r^{-4}\int_{\partial B_{r}(x^{\circ})}\frac{u\nabla u\cdot\nu}{x_{1}\bar{\rho}_{0}}\,d\mathcal{H}^{1}\\
            &-3r^{-5}\int_{\partial B_{r}(x^{\circ})}\frac{u^{2}}{x_{1}\bar{\rho}_{0}}\,d\mathcal{H}^{1}\\
            &-r^{-5}\int_{\partial B_{r}(x^{\circ})}\frac{x_{1}-x_{1}^{\circ}}{x_{1}^{2}\bar{\rho}_{0}}u^{2}\,d\mathcal{H}^{1}.
        \end{split}
    \end{align} 
    Then the monotonicity formula~\eqref{Mx2r1} follows from~\eqref{Ir1st} and~\eqref{Jr1st}.
\end{proof}
The following monotonicity formula holds for the axis-degenerate points.
\begin{proposition}[Monotonicity formula for the non-stagnation axis  points]\label{propadp}
    Let $u$ be a subsonic variational solution to \eqref{fb}, let $x^{\circ}=(0,x_{2}^{\circ})\in \Omega$ with $x_{2}^{\circ}\neq 0$, and let $\delta$ defined as in~\eqref{d}. Define, for any $r\in(0,\delta)$, 
    \begin{equation}\label{Mx1r}
        M^{x_{1}}(r)=r^{-3}I(r)-2r^{-4}J(r),
    \end{equation}
    where $I(r)$ and $J(r)$ are defined in~\eqref{Ir} and~\eqref{Jr}, respectively. Then for a.e. $r\in(0,\delta)$, 
    \begin{align}\label{Mx1r1}
        \begin{split}
            (M^{x_{1}}(r))'&=2r^{-3}\int_{\partial B_{r}^{+}(x^{\circ})}\frac{1}{x_{1}H(\tfrac{|\nabla u|^{2}}{x_{1}^{2}};x_{2})}\left( \nabla u\cdot\nu-2\frac{u}{r} \right)^{2}\,d\mathcal{H}^{1}\\
            &+\sum_{i=1}^{3}r^{-4}K_{i}^{x_{2}}(r),
        \end{split}
    \end{align}
    where $K_{1}^{x_{1}}$ is defined in equation~\eqref{K1x1r} and 
    \begin{equation}\label{K2x1r}
        K_{2}^{x_{1}}(r)=4\int_{\partial B_{r}^{+}(x^{\circ})}\left( \frac{1}{x_{1}H(\tfrac{|\nabla u|^{2}}{x_{1}^{2}};x_{2})}-\frac{1}{x_{1}\bar{\rho}_{0}} \right)u\nabla u\cdot\nu\,d\mathcal{H}^{1},
    \end{equation}
    and 
    \begin{equation}\label{K3x1r}
        K_{3}^{x_{1}}(r)=\frac{8}{r}\int_{\partial B_{r}^{+}(x^{\circ})}\left( \frac{1}{x_{1}\bar{\rho}_{0}}-\frac{1}{x_{1}H(\tfrac{|\nabla u|^{2}}{x_{1}^{2}};x_{2})} \right)u^{2}\,d\mathcal{H}^{1}.
    \end{equation}
\end{proposition}
\begin{proof}
    In the case $x_{1}^{\circ}=0$ and $x_{2}^{\circ}\neq 0$, we apply the Poho\v{z}aev-type identity~\eqref{tya}, and it follows from~\eqref{eni} that 
    \begin{align}\label{Ir1ad}
        \begin{split}
            (r^{-3}I(r))'&=2r^{-3}\int_{\partial B_{r}^{+}(x^{\circ})}\frac{(\nabla u\cdot\nu)^{2}}{x_{1}H(\tfrac{|\nabla u|^{2}}{x_{1}^{2}};x_{2})}\,d\mathcal{H}^{1}\\
            &-4r^{-4}\int_{\partial B_{r}^{+}(x^{\circ})}\frac{u\nabla u\cdot\nu}{x_{1}H(\tfrac{|\nabla u|^{2}}{x_{1}^{2}};x_{2})}\,d\mathcal{H}^{1}\\
            &+r^{-4}K_{1}^{x_{1}}(r),
        \end{split}
    \end{align}
    and 
    \begin{align}\label{Jr1ad}
        \begin{split}
            (r^{-4}J(r))'&=2r^{-4}\int_{\partial B_{r}(x^{\circ})}\frac{u\nabla u\cdot\nu}{x_{1}\bar{\rho}_{0}}\,d\mathcal{H}^{1}\\
            &-4r^{-5}\int_{\partial B_{r}(x^{\circ})}\frac{u^{2}}{x_{1}\bar{\rho}_{0}}\,d\mathcal{H}^{1}.
        \end{split}
    \end{align}
    Then~\eqref{Mx1r1} follows from~\eqref{Ir1ad} and~\eqref{Jr1ad}, with $K_{2}^{x_{1}}(r)$ and $K_{3}^{x_{1}}(r)$ defined in~\eqref{K2x1r} and~\eqref{K3x1r}, respectively.
\end{proof}
Finally, we establish the following monotonicity formula for the origin.
\begin{proposition}[Monotonicity formula for the origin]\label{prop0}
    Let $u$ be a subsonic variational solution to \eqref{fb}, let $x^{\circ}=(0,0)\in \Omega$, and let $\delta$ defined as in~\eqref{d}. Define, for any $r\in(0,\delta)$, 
    \begin{equation}\label{Mx12r}
        M^{x_{1}x_{2}}(r)=r^{-4}I(r)-\frac{5}{2}r^{-5}J(r).
    \end{equation}
    Then for a.e. $r\in(0,\delta)$,
    \begin{align}\label{Mx12r1}
        \begin{split}
            (M^{x_{1}x_{2}}(r))'&=2r^{-4}\int_{\partial B_{r}^{+}}\frac{1}{x_{1}H(\tfrac{|\nabla u|^{2}}{x_{1}^{2}};x_{2})}\left( \nabla u\cdot\nu-\frac{5}{2}\frac{u}{r} \right)^{2}\,d\mathcal{H}^{1}\\
            &+r^{-5}K_{1}^{x_{2}}(r)+r^{-5}\sum_{i=1}^{3}K_{i}^{x_{1}x_{2}}(r),
        \end{split}
    \end{align}
    where $K_{1}^{x_{2}}(r)$ and $K_{1}^{x_{1}x_{2}}(r)$ are defined in~\eqref{K1x2r} and~\eqref{K1x12r}, respectively, and 
    \begin{equation}\label{K2x12r}
        K_{2}^{x_{1}x_{2}}(r)=5\int_{\partial B_{r}^{+}}\left( \frac{1}{x_{1}H(\tfrac{|\nabla u|^{2}}{x_{1}^{2}};x_{2})}-\frac{1}{x_{1}\bar{\rho}_{0}} \right)u\nabla u\cdot\nu\,d\mathcal{H}^{1},
    \end{equation}
    and 
    \begin{equation}\label{K3x12r}
        K_{3}^{x_{1}x_{2}}(r)=\frac{25}{2r}\int_{\partial B_{r}^{+}}\left( \frac{1}{x_{1}\bar{\rho}_{0}}-\frac{1}{x_{1}H(\tfrac{|\nabla u|^{2}}{x_{1}^{2}};x_{2})} \right)u^{2}\,d\mathcal{H}^{1}.
    \end{equation}
\end{proposition}
\begin{proof}
    For the origin, we use the Poho\v{z}eav-type identity~\eqref{ty0}. A direct calculation gives that 
    \begin{align}\label{Ir10}
        \begin{split}
            (r^{-4}I(r))&=2r^{-4}\int_{\partial B_{r}^{+}}\frac{(\nabla u\cdot\nu)^{2}}{x_{1}H(\tfrac{|\nabla u|^{2}}{x_{1}^{2}};x_{2})}\,d\mathcal{H}^{1}\\
            &-5r^{-5}\int_{\partial B_{r}^{+}}\frac{u\nabla u\cdot\nu}{x_{1}H(\tfrac{|\nabla u|^{2}}{x_{1}^{2}};x_{2})}\,d\mathcal{H}^{1}\\
            &+r^{-5}K_{1}^{x_{2}}(r)+r^{-5}K_{1}^{x_{1}x_{2}}(r),
        \end{split}
    \end{align}
    and 
    \begin{align}\label{Jr10}
        \begin{split}
            (r^{-5}J(r))'&=2r^{-5}\int_{\partial B_{r}^{+}}\frac{1}{x_{1}\bar{\rho}_{0}}u\nabla u\cdot\nu\,d\mathcal{H}^{1}\\
            &-5r^{-6}\int_{\partial B_{r}^{+}}\frac{1}{x_{1}\bar{\rho}_{0}}u^{2}\,d\mathcal{H}^{1}.
        \end{split}
    \end{align}
    Then~\eqref{Mx12r1} follows directly from~\eqref{Ir10} and~\eqref{Jr10}, with $K_{2}^{x_{1}x_{2}}(r)$ and $K_{3}^{x_{1}x_{2}}(r)$ defined in~\eqref{K2x12r} and~\eqref{K3x12r}.
\end{proof}
\section{The blow-up analysis at the stagnation points}
In this section, we analyze the singular asymptotics near the stagnation points of the problem~\eqref{fb}. To be specific, we define 
\begin{equation*}
    S^{u}= \left\{ (x_{1},x_{2})\in \Omega\cap\partial\{u>0\}: x_{1}>0, x_{2}=0 \right\}.
\end{equation*}
Our first result states that if one considers the rescaling $u_{r}(x):=\frac{u(x^{\circ}+rx)}{r^{3/2}}$ for each $x^{\circ}\in S^{u}$, then the blow-up limit $u_{0}$ is a homogeneous function of degree $3/2$.
\begin{lemma}\label{lemstp}
    Let $u$ be a subsonic variational solution of~\eqref{fb}, let $x^{\circ}\in \Omega$ and assume that 
    \begin{equation}\label{gastp}
        \frac{|\nabla u|^{2}}{x_{1}^{2}} \leqslant Cx_{2}^{+}\quad\text{ locally in }\Omega.
    \end{equation}
    Then 
    \begin{enumerate}
        \item [(1).] The limit $M^{x_{2}}(0^{+}):=\lim_{r\to 0^{+}}M^{x_{2}}(r)$ exists and is finite (Recall the definition of $M^{x_{2}}(r)$ in~\eqref{Mx2r}).
        \item [(2).] Let $r_{m}\to 0^{+}$ as $m\to \infty$ be a vanishing sequence so that the blow-up sequence 
        \begin{equation}\label{blstp}
            u_{m}(x):=\frac{u(x^{\circ}+r_{m}x)}{r_{m}^{3/2}}
        \end{equation}
        converges weakly in $W_{\mathit{loc}}^{1,2}(\mathbb{R} _{2})$ to a blow-up limit $u_{0}$, then $u_{0}$ is a homogeneous function of degree $3/2$, i.e., $u_{0}(\lambda x)=\lambda^{3/2}u_{0}(x)$ for all $\lambda>0$.
        \item [(3).] Let $u_{m}$ be a converging sequence as in (2). Then $u_{m}$ converges to $u_{0}$ strongly in $W_{\mathit{loc}}^{1,2}(\mathbb{R} ^{2})$.
    \end{enumerate}
\end{lemma}
\begin{proof}
    The proof is divided into three steps.

    \noindent\textit{Step} 1. We will verify the integrability of $r\mapsto r^{-4}\sum_{i=1}^{6}K_{i}^{x_{2}}(r)$. By the uniform subsonic condition for variational solution $u$, we have that 
        \begin{equation}\label{esm1}
            \left|\pd{!}{\tau}\left( \frac{1}{H(\tau;x_{2})} \right) \right|\leqslant C,\quad\text{ and }\quad\left|\pd{!}{x_{2}}\left( \frac{1}{H(\tau;x_{2})} \right)\right|\leqslant C\quad\text{ in }B_{r}(x^{\circ}),
        \end{equation}
        This together with the growth assumption~\eqref{gastp} gives that 
        \begin{align}\label{esmstp1}
            \begin{split}
                |K_{1}^{x_{2}}(r)| &\leqslant C\int_{B_{r}(x^{\circ})}|x_{1}|\left( \int_{0}^{\tfrac{|\nabla u|^{2}}{x_{1}^{2}}}\tau\,d\tau \right)+|x_{1}|\left( \int_{0}^{x_{2}}\tau\,d\tau \right)\,dx \\
                & \leqslant C\int_{B_{r}(x^{\circ})}\left[  \left( \frac{|\nabla u|^{2}}{x_{1}^{2}} \right)^{2}+x_{2}^{2}\right]\,dx\\
                &\leqslant Cr^{4},
            \end{split}
        \end{align}
        and 
        \begin{align}\label{esmstp2}
            \begin{split}
                |K^{x_{2}}(r)| &\leqslant C\int_{B_{r}(x^{\circ})}|x_{1}|\left( \int_{0}^{\tfrac{|\nabla u|^{2}}{x_{1}^{2}}}x_{2}\,d\tau \right)+|x_{1}|\left( \int_{0}^{x_{2}}x_{2}\,d\tau \right)\,dx\\
                & \leqslant Cr^{4}.
            \end{split}
        \end{align}
        Here, we used the fact that $|x_{1}| \leqslant C$ and $|x_{2}| \leqslant Cr$ in $B_{r}(x^{\circ})$ for each $x^{\circ}\in S^{u}$. Thanks to the fact that $H(\tfrac{|\nabla u|^{2}}{x_{1}^{2}};x_{2})$ is uniformly bounded from above and below by a constant $C$ depending only on $\bar{\rho}_{0}$, in $B_{r}(x^{\circ})$ with $x^{\circ}\in S^{u}$, we have  
        \begin{align}\label{esmstp3}
            \begin{split}
                |K_{3}^{x_{2}}(r)| &\leqslant Cr\int_{B_{r}(x^{\circ})}\frac{|\nabla u|^{2}}{x_{1}^{2}}+|x_{2}|+\left( \int_{0}^{x_{2}}\tau\,d\tau \right)\,dx \\
                &\leqslant Cr^{4}.
            \end{split}
        \end{align}
        Observe that 
        \begin{align}\label{esmH}
            \begin{split}
                &\frac{1}{H(\tfrac{|\nabla u|^{2}}{x_{1}^{2}};x_{2})}-\frac{1}{\bar{\rho}_{0}}\\
                &=\int_{0}^{1}\frac{d}{d\theta}\left( \frac{1}{H(\tfrac{\theta|\nabla u|^{2}}{x_{1}^{2}};x_{2})} \right)\,d\theta+\int_{0}^{1}\frac{d}{ds}\left( \frac{1}{H(0;sx_{2})} \right)\,ds\\
                &=\left(\int_{0}^{1}\frac{-\partial_{1}H(\tfrac{\theta|\nabla u|^{2}}{x_{1}^{2}};x_{2})}{H^{2}(\tfrac{\theta|\nabla u|^{2}}{x_{1}^{2}};x_{2})}\,d\theta\right)\frac{|\nabla u|^{2}}{x_{1}^{2}}+\left( \int_{0}^{1}\frac{-\partial_{2}H(0;sx_{2})}{H^{2}(0;sx_{2})}\,ds \right)x_{2},
            \end{split}
        \end{align}
        where $\theta$, $s\in(0,1)$ and $\partial_{1}H$, $\partial_{2}H$ are defined in equations~\eqref{p1H} and~\eqref{p2H} respectively. Thanks to~\eqref{esm1} and the growth condition~\eqref{gastp}, we obtain 
        \begin{equation}\label{esmH1}
            \left|\frac{1}{H(\tfrac{|\nabla u|^{2}}{x_{1}^{2}};x_{2})}-\frac{1}{\bar{\rho}_{0}}\right| \leqslant Cr\qquad\text{ in }B_{r}(x^{\circ}).
        \end{equation}
        It follows that 
    \begin{equation}\label{esmstp4}
        |K_{4}^{x_{2}}(r)| \leqslant Cr\int_{\partial B_{r}(x^{\circ})}\frac{1}{|x_{1}|}|u||\nabla u|\,d\mathcal{H}^{1} \leqslant Cr^{4},
    \end{equation}
    and 
    \begin{equation}\label{esmstp5}
        |K_{5}^{x_{2}}(r)| \leqslant C\int_{\partial B_{r}(x^{\circ})}\frac{1}{|x_{1}|}u^{2}\,d\mathcal{H}^{1} \leqslant Cr^{4}.
    \end{equation}
    Finally, given that $|x_{1}-x_{1}^{\circ}| \leqslant Cr$ on $\partial B_{r}(x^{\circ})$, we have that 
    \begin{equation}\label{esmstp6}
        |K_{6}^{x_{2}}(r)| \leqslant C\int_{\partial B_{r(x^{\circ})}}\frac{1}{x_{1}^{2}}u^{2}\,d\mathcal{H}^{1} \leqslant Cr^{4}.
    \end{equation}
    Collecting the estimates~\eqref{esmstp1},~\eqref{esmstp2},~\eqref{esmstp3},~\eqref{esmstp4},~\eqref{esmstp5} and~\eqref{esmstp6} together, we have that 
    \begin{equation*}
        \left|r^{-4}\sum_{i=1}^{6}K_{i}^{x_{2}}(r)\right| \leqslant C_{0}.
    \end{equation*}
    This proves the integrability of the function $r\mapsto\sum_{i=1}^{6}r^{-4}K_{i}^{x_{2}}(r)$. Moreover, we have that the function $r\mapsto M^{x_{2}}(r)$ has a right limit $M^{x_{2}}(0^{+})$. The finiteness of $M^{x_{2}}(0^{+})$ follows from the growth condition~\eqref{gastp} and $u(x^{\circ})=0$.

    \noindent\textit{Step} 2. For any $0<\sigma<\tau<\infty$ and any vanishing sequence $r_{m}\to 0^{+}$, we integrate the equation~\eqref{Mx2r1} from $r_{m}\sigma$ to $r_{m}\tau$ and we get 
    \begin{align}\label{rescl1}
        \begin{split}
            &2\int_{B_{\tau}\setminus B_{\sigma}}|x|^{-5}\frac{(\nabla u_{m}\cdot x-\tfrac{3}{2}u_{m})^{2}}{(x_{1}^{\circ}+r_{m}x_{1})H(\tfrac{r_{m}|\nabla u_{m}|^{2}}{(x_{1}^{\circ}+r_{m}x_{1})^{2}};r_{m}x_{2})}\,dx\\
            &=M^{x_{2}}(r_{m}\tau)-M^{x_{2}}(r_{m}\sigma)-\int_{r_{m}\sigma}^{r_{m}\tau}r^{-4}\sum_{i=1}^{6}K_{i}^{x_{2}}(r)\,dr.
        \end{split}
    \end{align}
    It follows from the existence and finiteness of $M^{x_{2}}(0^{+})$, and that the function $r\mapsto r^{-4}\sum_{i=1}^{6}K_{i}^{x_{2}}(r)$ is integrable for $r\in(0,\delta)$ that 
    \begin{equation*}
        \lim_{m\to \infty}\left( M^{x_{2}}(r_{m}\tau)-M^{x_{2}}(r_{m}\sigma)-\int_{r_{m}\sigma}^{r_{m}\tau}r^{-4}\sum_{i=1}^{6}K_{i}^{x_{2}}(r)\,dr \right)=0.
    \end{equation*}
    On the other hand, since
    \begin{equation}\label{limhm}
        \lim_{m\to\infty}H\bigg(\tfrac{r_{m}|\nabla u_{m}|^{2}}{(x_{1}^{\circ}+r_{m}x_{1})^{2}};r_{m}x_{2}\bigg)=H(0;0)=\bar{\rho}_{0}.
    \end{equation}
    By passing to the limit as $m\to\infty$ in equation~\eqref{rescl1} and applying the weak lower semicontinuity property in $W^{1,2}$, we deduce that $u_{0}$ is a homogeneous function of degree $3/2$.

    \noindent\textit{Step} 3. Let $u_{m}$ be the sequence defined in~\eqref{blstp}. We first prove that 
    \begin{equation}\label{cvghm}
        h_{m}:=\int_{B_{1}}\left|\frac{1}{H(\tfrac{r_{m}|\nabla u_{m}|^{2}}{(x_{1}^{\circ}+r_{m}x_{1})^{2}};r_{m}x_{2})}-\frac{1}{\bar{\rho}_{0}}\right|^{2}\,dx\to 0,
    \end{equation}
    as $m\to \infty$. Thanks to~\eqref{limhm}, we have 
    \begin{align*}
        h_{m}& \leqslant C\int_{B_{1}}\left|\frac{1}{H_{m}}-\frac{1}{\bar{\rho}_{0}}\right|\,dx\\
        & \leqslant C\int_{B_{1}}\int_{0}^{1}\frac{\left|\partial_{1}H(\tfrac{r_{m}\theta|\nabla u_{m}|^{2}}{(x_{1}^{\circ}+r_{m}x_{1})^{2}};r_{m}x_{2})\right|}{H^{2}(\tfrac{r_{m}\theta|\nabla u_{m}|^{2}}{(x_{1}^{\circ}+r_{m}x_{1})^{2}};r_{m}x_{2})}\,d\theta\cdot\frac{r_{m}|\nabla u_{m}|^{2}}{(x_{1}^{\circ}+r_{m}x_{1})^{2}}\,dx\\
        &+C\int_{B_{1}}\int_{0}^{1}\frac{\left|\partial_{2}H(0;r_{m}sx_{2})\right|}{H^{2}(0;r_{m}sx_{2})}\,ds\cdot r_{m}x_{2}\,dx,
    \end{align*}
    where $\theta$, $s\in(0,1)$. It follows from the subsonic condition~\eqref{subs} for subsonic variational solution that 
    \begin{equation*}
        h_{m} \leqslant Cr_{m}\int_{B_{1}}\left[  \frac{|\nabla u_{m}|^{2}}{(x_{1}^{\circ}+r_{m}x_{1})^{2}}+x_{2}\right]\,dx \leqslant Cr_{m},
    \end{equation*}
    where we have applied the growth assumption~\eqref{gastp} in the last inequality. This proves~\eqref{cvghm} as passing to the limit as $m\to\infty$. We now infer from the weak $L^{2}$ convergence of $\nabla u_{m}$ and the strong convergence of $H_{m}$ that 
    \begin{align*}
        \begin{split}
            0&=\int_{\mathbb{R} ^{2}}\frac{\nabla u_{m}\cdot\nabla \eta}{(x_{1}^{\circ}+r_{m}x_{1})H(\tfrac{r_{m}|\nabla u_{m}|^{2}}{(x_{1}^{\circ}+r_{m}x_{1})^{2}};r_{m}x_{2})}\,dx\\
            &\to \int_{\mathbb{R} ^{2}}\frac{\nabla u_{0}\cdot\nabla \eta}{x_{1}^{\circ}\bar{\rho}_{0}}\,dx\quad\text{ for each }\eta\in C_{0}^{\infty}(\mathbb{R} ^{2}).
        \end{split}
    \end{align*}
    Since $x^{\circ}\in S^{u}$, $x_{1}^{\circ}>0$ and we may deduce that $u_{0}$ is harmonic in the set $\{u_{0}>0\}$ in the weak sense. Additionally, since $u_{m}$ converges to $u_{0}$ strongly in $L_{\mathit{loc}}^{2}(\mathbb{R} ^{2})$, we have 
    \begin{equation*}
        \lim_{m\to\infty}\int_{\mathbb{R} ^{2}}\frac{u_{m}\nabla u_{m}\cdot\nabla \eta}{(x_{1}^{\circ}+r_{m}x_{1})H(\tfrac{r_{m}|\nabla u_{m}|^{2}}{(x_{1}^{\circ}+r_{m}x_{1})^{2}};r_{m}x_{2})}\,dx=\int_{\mathbb{R} ^{2}}\frac{u_{0}\nabla u_{0}\cdot\nabla \eta}{x_{1}^{\circ}\bar{\rho}_{0}}\,dx.
    \end{equation*}
    Consequently, for any test function $\eta\in C_{0}^{\infty}(\mathbb{R} ^{2})$,
    \begin{align*}
        &o(1)+\frac{1}{x_{1}^{\circ}\bar{\rho}_{0}}\int_{\mathbb{R} ^{2}}|\nabla u_{m}|^{2}\eta\,dx\\
        &=\int_{\mathbb{R} ^{2}}\frac{|\nabla u_{m}|^{2}}{(x_{1}^{\circ}+r_{m}x_{1})H(\tfrac{r_{m}|\nabla u_{m}|^{2}}{(x_{1}^{\circ}+r_{m}x_{1})^{2}};r_{m}x_{2})}\eta\,dx\\
        &=-\int_{\mathbb{R} ^{2}}\frac{u_{m}\nabla u_{m}\nabla \eta}{(x_{1}^{\circ}+r_{m}x_{1})H(\tfrac{r_{m}|\nabla u_{m}|^{2}}{(x_{1}^{\circ}+r_{m}x_{1})^{2}};r_{m}x_{2})}\,dx\\
        &\to-\int_{\mathbb{R} ^{2}}\frac{u_{0}\nabla u_{0}\cdot\nabla \eta}{x_{1}^{\circ}\bar{\rho}_{0}}\,dx\quad\text{ as }m\to\infty\\
        &=\frac{1}{x_{1}^{\circ}\bar{\rho}_{0}}\int_{\mathbb{R} ^{2}}|\nabla u_{0}|^{2}\eta\,dx.
    \end{align*}
    This proves the strong convergence of $u_{0}$.
\end{proof}
The following lemma establishes the explicit value of $M^{x_{1}}(0^{+})$ and characterizes the variational equation governing $u_{0}$.
\begin{lemma}[Weighted density at the stagnation points]\label{lemwd}
    Let $u$ be a subsonic variational solution of~\eqref{fb}, let $x^{\circ}\in S^{u}$ and suppose that $u$ satisfies the growth assumption~\eqref{gastp}. Then 
    \begin{enumerate}
        \item [(1).] $M^{x_{2}}(0^{+})$ takes the value 
        \begin{equation}\label{wden}
            M^{x_{2}}(0^{+})=\frac{x_{1}^{\circ}}{\bar{\rho}_{0}}\lim_{r\to 0^{+}}r^{-3}\int_{B_{r}(x^{\circ})}x_{2}^{+}\chi_{\left\{ u>0 \right\} }\,dx,
        \end{equation}
        and in particular $M^{x_{2}}(0^{+})=0$ implies that $u_{0}\equiv 0$ in $\mathbb{R} ^{2}$ for each blow-up limit $u_{0}$ of $u_{m}$ defined in equation~\eqref{blstp}.
        \item [(2).] Let $u_{m}$ be a blow-up sequence defined in equation~\eqref{blstp}, let $u_{0}$ be the $W_{\mathit{loc}}^{1,2}$ strong limit of $u_{m}$ and let $\chi_{0}$ be the $L_{\mathit{loc}}^{1}$ strong limit of $\chi_{\left\{ u_{m}>0 \right\} }$. Then $u_{0}$ is a homogeneous solution of
        \begin{align}\label{fvu0}
            \begin{split}
                0&=\frac{1}{x_{1}^{\circ}\bar{\rho}_{0}}\int_{\mathbb{R} ^{2}}\left( |\nabla u_{0}|^{2} \operatorname{div}\phi-2\nabla u_{0}D\phi\nabla u_{0} \right)\,dx\\
                &+\frac{x_{1}^{\circ}}{\bar{\rho}_{0}}\int_{\mathbb{R} ^{2}}(x_{2}\chi_{0} \operatorname{div}\phi+\chi_{0}\phi_{2})\,dx,
            \end{split}
        \end{align}
        for each $\phi=(\phi_{1},\phi_{2})\in C_{0}^{1}(\mathbb{R} ^{2};\mathbb{R} ^{2})$.
    \end{enumerate}
\end{lemma}
\begin{proof}
    (1). It follows from~\eqref{Mx2r} that for each vanishing sequence $r_{m}\to 0^{+}$ and each $r>0$
    \begin{align*}
        \begin{split}
            M^{x_{2}}(rr_{m})&=r^{-3}E_{H}(u_{m};B_{r})-\frac{3}{2}r^{-4}\int_{\partial B_{r}}\frac{1}{(x_{1}^{\circ}+rx_{1})\bar{\rho}_{0}}u_{m}^{2}\,d\mathcal{H}^{1}\\
            &-(rr_{m})^{-3}K_{1}^{x_{2}}(rr_{m})
        \end{split}
    \end{align*}
    where 
    \begin{align*}
        E_{H}(u_{m};B_{r})&:=\int_{B_{r}}(x_{1}^{\circ}+r_{m}x_{1})\Bigg[ \frac{|\nabla u_{m}|^{2}}{(x_{1}^{\circ}+r_{m}x_{1})^{2}H(\tfrac{r_{m}|\nabla u_{m}|^{2}}{(x_{1}^{\circ}+r_{m}x_{1})^{2}};r_{m}x_{2})}\\
        &\qquad\qquad\qquad\qquad\qquad+\frac{x_{2}}{\bar{\rho}_{0}}\chi_{\left\{ u_{m}>0 \right\} } \Bigg]\,dx.
    \end{align*}
    Thanks to~\eqref{esmstp1}, we have that $|K_{1}^{x_{2}}(rr_{m})| \leqslant C(rr_{m})^{4}$, and this implies that 
    \[
        \lim_{m\to \infty}r^{-3}r_{m}^{-3}K_{1}^{x_{2}}(rr_{m})=0\qquad\forall\,r>0.
    \]
    Moreover, it follows from~\eqref{cvghm}, $u_{m}$ converges strongly to $u_{0}$ in $W_{\mathit{loc}}^{1,2}$, and $u_{0}$ is harmonic in $\{u_{0}>0\}$, $u_{0}\equiv 0$ on $\{x_{2} \leqslant 0\}$ that 
    \begin{align*}
        0&=\lim_{m\to\infty}r^{-3}E_{H}(u_{m};B_{r})\\
        &-\lim_{m\to\infty}\left(\frac{3}{2}r^{-4}\int_{\partial B_{r}}\frac{u_{m}^{2}}{(x_{1}^{\circ}+rx_{1})\bar{\rho}_{0}}\,d\mathcal{H}^{1}-M^{x_{2}}(rr_{m}) \right)\\
        &=\frac{r^{-3}}{x_{1}^{\circ}\bar{\rho}_{0}}\int_{\partial B_{r}}u_{0}\left[ \nabla u_{0}\cdot\nu-\frac{3}{2}\frac{u_{0}}{r}\right]\,d\mathcal{H}^{1}\\
        &+\lim_{m\to \infty}\left( r^{-3}\frac{x_{1}^{\circ}}{\bar{\rho}_{0}}\int_{B_{r}}x_{2}\chi_{\left\{ u_{m}>0 \right\} }\right)\\
        &=\frac{x_{1}^{\circ}}{\bar{\rho}_{0}}\lim_{m\to \infty}r^{-3}\int_{B_{r}}x_{2}\chi_{\left\{ u_{m}>0 \right\} },
    \end{align*} 
    where we used the fact that $u_{0}\in W_{\mathit{loc}}^{1,2}(\mathbb{R} ^{2})$ is a homogeneous function of degree $3/2$ in the last equality. Then~\eqref{wden} follows immediately from the fact that $M^{x_{2}}(0^{+})$ exists and is finite (recall~\lemref{lemstp} (1)).

    (2). Let $\phi\in C_{0}^{1}(\mathbb{R} ^{2};\mathbb{R} ^{2})$ and set $\phi_{m}(x)=\phi(\tfrac{x-x^{\circ}}{r_{m}})$. It follows from~\eqref{fv} that 
    \begin{align*}
        0&=\int_{\Omega}x_{1}\left[ F( \tfrac{|\nabla u|^{2}}{x_{1}^{2}};x_{2}) + \lambda(x_{2})\chi_{\left\{ u>0 \right\} }  \right] \operatorname{div}\phi_{m}\,dx\\
        &-2\int_{\Omega}\frac{\nabla uD\phi_{m}\nabla u}{x_{1}H( \tfrac{|\nabla u|^{2}}{x_{1}^{2}};x_{2})}\,dx\\
        &+\int_{\Omega}\left[ F(\tfrac{|\nabla u|^{2}}{x_{1}^{2}};x_{2})-\tfrac{2|\nabla u|^{2}}{x_{1}^{2}H(\tfrac{|\nabla u|^{2}}{x_{1}^{2}};x_{2})}\right]\phi_{1}(\tfrac{x-x^{\circ}}{r_{m}})\,dx\\
        &+\int_{\Omega}+ \lambda(x_{2})\chi_{\left\{ u>0 \right\} }\phi_{1}\left( \frac{x-x^{\circ}}{r_{m}} \right)\,dx\\
        &+\int_{\Omega}x_{1}\left[\partial_{2}F(\tfrac{|\nabla u|^{2}}{x_{1}^{2}};x_{2})+\lambda'(x_{2})\chi_{\left\{ u>0 \right\} } \right]\phi_{2}(\tfrac{x-x^{\circ}}{r_{m}})\,dx.
    \end{align*}
    Changing variables gives 
    \begin{align*}
        \begin{split}
            0&=\int_{\Omega_{m}}\frac{|\nabla u_{m}|^{2}\operatorname{div}\phi}{(x_{1}^{\circ}+r_{m}x_{1})H(\tfrac{r_{m}|\nabla u_{m}|^{2}}{(x_{1}^{\circ}+r_{m}x_{1})^{2}};r_{m}x_{2})}\,dx\\
            &+\int_{\Omega_{m}}\frac{(x_{1}^{\circ}+rx_{1})x_{2}}{\bar{\rho}_{0}}\chi_{\left\{ u_{m}>0 \right\} }  \operatorname{div}\phi\,dx\\
            &-2\int_{\Omega}\frac{\nabla u_{m}D\phi \nabla u_{m}}{(x_{1}^{\circ}+r_{m}x_{1})H(\tfrac{r_{m}|\nabla u_{m}|^{2}}{(x_{1}^{\circ}+r_{m}x_{1})^{2}};r_{m}x_{2})}\,dx\\
            &+\sum_{i=1}^{3}I_{i}(r_{m}),
        \end{split}
    \end{align*}
    where 
    \begin{align*}
        I_{1}(r_{m})&=r_{m}^{-1}\int_{B_{r}}(x_{1}^{\circ}+r_{m}x_{1})\Bigg[\int_{0}^{r_{m}x_{2}}\tfrac{\partial}{\partial\tau}\left( \tfrac{1}{H(\tau;r_{m}x_{2})} \right)\tau\,d\tau\Bigg] \chi_{\left\{ u_{m}>0 \right\} }\operatorname{div}\phi\,dx\\
        &-r_{m}^{-1}\int_{B_{r}}(x_{1}^{\circ}+r_{m}x_{1})\left[\int_{0}^{\tfrac{r_{m}|\nabla u_{m}|^{2}}{(x_{1}^{\circ}+r_{m}x_{1})^{2}}}\tfrac{\partial}{\partial\tau}\left( \tfrac{1}{H(\tau;r_{m}x_{2})} \right)\tau\,d\tau \right]\operatorname{div}\phi\,dx,
    \end{align*}
    \begin{flalign*}
        &\ I_{2}(r_{m})=\int_{\Omega_{m}}F\Bigg(\tfrac{r_{m}|\nabla u_{m}|^{2}}{(x_{1}^{\circ}+r_{m}x_{1})^{2}};r_{m}x_{2}\Bigg)\phi_{1}\,dx+\int_{\Omega_{m}}\lambda(r_{m}x_{2})\phi_{1}\,dx\\
        &-2r_{m}\int_{\Omega_{m}}\frac{|\nabla u_{m}|^{2}}{(x_{1}^{\circ}+r_{m}x_{1})^{2}H(\tfrac{r_{m}|\nabla u_{m}|^{2}}{(x_{1}^{\circ}+r_{m}x_{1})^{2}};r_{m}x_{2})}\phi_{1}\,dx&
    \end{flalign*}
    and 
    \begin{flalign*}
        &\ I_{3}(r_{m})=\int_{\Omega_{m}}(x_{1}^{\circ}+r_{m}x_{1})\partial_{2}F(\tfrac{r_{m}|\nabla u_{m}|^{2}}{(x_{1}^{\circ}+r_{m}x_{1})^{2}};r_{m}x_{2})\phi_{2}\,dx\\
        &+\int_{\Omega_{m}}\left[\left( \frac{1}{\bar{\rho}_{0}}-\int_{0}^{r_{m}x_{2}}\tfrac{\partial}{\partial x_{2}}\left( \tfrac{1}{H(\tau;r_{m}x_{2})} \right) \,d\tau\right)\chi_{\left\{ u_{m}>0 \right\} } \right]\phi_{2}\,dx. &
    \end{flalign*}
    Then~\eqref{fvu0} follows from the strong convergence of $u_{m}$ in $W_{\mathit{loc}}^{1,2}(\mathbb{R} ^{2})$,~\eqref{cvghm}, and a similar calculation in~\eqref{esmstp1},~\eqref{esmstp2}. This finishes the proof of the lemma.
\end{proof}
As a direct corollary of~\lemref{lemwd}, we have 
\begin{corollary}[The blow-up limits for the stagnation points]\label{corbl}
    Let $u$ be a subsonic variational solution of~\eqref{fb}, let $x^{\circ}\in S^{u}$ and let $\delta$ be given as in~\eqref{d}. Suppose that $u$ satisfies~\eqref{gastp}. Then 
    \begin{enumerate}
        \item [(1).] The all possible value of weighted density defined in~\eqref{wden} are 
        \begin{equation*}
            M^{x_{2}}(0^{+})\in \left\{ \frac{\sqrt{3}x_{1}^{\circ}}{3\bar{\rho}_{0}},\frac{2x_{1}^{\circ}}{3\bar{\rho}_{0}},0\right\} 
        \end{equation*}
        \item [(2).] If $M^{x_{2}}(0^{+})=\frac{\sqrt{3}x_{1}^{\circ}}{3\bar{\rho}_{0}}$, then 
        \begin{align}\label{blstp1}
            \begin{split}
                &\lim_{r\to 0^{+}}\frac{u(x^{\circ}+rx)}{r^{3/2}}\\
                &=\frac{\sqrt{2}x_{1}^{\circ}\bar{\rho}_{0}}{3}(x_{1}^{2}+x_{2}^{2})^{3/4}\cos\left( \frac{3}{2}\arctan\left( \frac{x_{1}}{x_{2}} \right) \right)\chi_{\left\{ -\tfrac{\pi}{3}<\arctan(\tfrac{x_{1}}{x_{2}})<\tfrac{\pi}{3} \right\} }
            \end{split}
        \end{align} 
        strongly in $W_{\mathit{loc}}^{1,2}(\mathbb{R} ^{2})$ and locally uniformly in $\mathbb{R} ^{2}$. 
        \item [(3).] If $M^{x_{2}}(0^{+})=\frac{2x_{1}^{\circ}}{3\bar{\rho}_{0}}$, then 
        \begin{equation}\label{blstp2}
            \lim_{r\to 0^{+}}\frac{u(x^{\circ}+rx)}{r^{3/2}}= 0,
        \end{equation}
        strongly in $W_{\mathit{loc}}^{1,2}(\mathbb{R} ^{2})$ and locally uniformly in $\mathbb{R} ^{2}$.
        \item [(4).] If $M^{x_{2}}(0^{+})=0$, then 
        \begin{equation}\label{blstp3}
            \lim_{r\to 0^{+}}\frac{u(x^{\circ}+rx)}{r^{3/2}}= 0,
        \end{equation}
        strongly in $W_{\mathit{loc}}^{1,2}(\mathbb{R} ^{2})$ and locally uniformly in $\mathbb{R} ^{2}$.
    \end{enumerate}
\end{corollary}
\begin{proof}
    The idea of the proof is similar to that in~\cite[Theorem3.8 (ii)]{MR3225630}. We give a sketch here. If $u_{m}$ is a blow-up sequence defined in~\eqref{blstp} and $u_{m}$ converges to a blow-up limit $u_{0}$ strongly in $W_{\mathit{loc}}^{1,2}(\mathbb{R} ^{2})$. Then it follows from~\lemref{lemstp} and~\lemref{lemwd} that 
    \begin{align}\label{eqsu0}
        \left\{
            \begin{alignedat}{2}
                \Delta u_{0}&=0\quad&&\text{ in }B_{1}\cap\{u_{0}>0\},\\
                |\nabla u_{0}|^{2}&=(x_{1}^{\circ})^{2}\bar{\rho}_{0}^{2}x_{2}\quad&&\text{ on }B_{1}\cap\partial\{u_{0}>0\},\\
                u_{0}&=0\quad&&\text{ on }B_{1}\cap\partial\{u_{0}>0\},\\
                u_{0}&\geqslant 0\quad&&\text{ in }B_{1}\\
                u_{0}&=0\quad&&\text{ in }B_{1}\cap\{x_{2} \leqslant 0\},\\
                u_{0}(x)&=|x|^{3/2}u_{0}(\tfrac{x}{|x|})\quad&&\text{ in }B_{1}.
            \end{alignedat}
        \right.
    \end{align}
    The strategy of solving~\eqref{eqsu0} is as follows. Assume first that $u_{0}$ is non-trivial and we introduce the polar coordinates as $\varrho=\sqrt{x_{1}^{2}+x_{2}^{2}}$ and $\theta=\arctan(\tfrac{x_{1}}{x_{2}})$. Then $u_{0}(\varrho,\theta)=\varrho^{3/2}f(\theta)$ due to the last equation in~\eqref{eqsu0}. The Laplacian in polar coordinates gives
    \begin{equation*}
        f''(\theta)+\frac{9}{4}f(\theta)=0\quad\text{ in }\{f(\theta)>0\}.
    \end{equation*}
    The general solution is $f(\theta)=A\cos(\tfrac{3}{2}\theta+\theta_{1})$ for some $A$ and $\theta$ up to be determined. Since $u_{0} \geqslant 0$, we have that $\{f(\theta)>0\}$ is a cone of opening $120^{\circ}$. Moreover, the free boundary conditions on $\partial\{u_{0}>0\}$ give that 
    \begin{equation*}
        u_{0}(\varrho,\theta)=\frac{\sqrt{2}x_{1}^{\circ}\bar{\rho}_{0}}{3}\varrho^{3/2}\cos\left( \frac{3}{2}\theta \right)\chi_{\{-\pi/3<\theta<\pi/3\}}.
    \end{equation*}
    This together with the computation of the weighted density
    \begin{equation*}
        M^{x_{2}}(0^{+})=\frac{x_{1}^{\circ}}{\bar{\rho}_{0}}\int_{B_{1}}x_{2}\chi_{ \left\{  -\pi/3<\theta<\pi/3\right\} }\,dx=\frac{\sqrt{3}x_{1}^{\circ}}{3\bar{\rho}_{0}},
    \end{equation*}
    proves~\eqref{blstp1}. On the other hand, if $u_{0}\equiv 0$ is a trivial solution, then we have~\eqref{blstp2} and~\eqref{blstp3}. The difference is that when $\chi_{0}=1$, the weighted density is not degenerate and 
    \begin{equation*}
        M^{x_{2}}(0^{+})=\frac{x_{1}^{\circ}}{\bar{\rho}_{0}}\int_{B_{1}}x_{2}^{+}\,dx=\frac{2x_{1}^{\circ}}{3\bar{\rho}_{0}}.
    \end{equation*}
\end{proof}
\begin{remark}
    For any weak solution $u$ of~\eqref{fb} satisfies the growth condition~\eqref{gastp} for any $x^{\circ}\in S^{u}$, one has that $\chi_{\left\{ u>0 \right\} }$ is locally in $\Omega\cap\{x_{1}>0\}\cap\{x_{2}>0\}$ is a function of bounded variation, and the total variation measure $|\nabla\chi_{\left\{ u>0 \right\} }|$ satisfies
    \begin{equation*}
        \int_{B_{r}(x^{\circ})}\sqrt{x_{2}^{+}}d|\nabla\chi_{\left\{ u>0 \right\} }| \leqslant Cr^{3/2},
    \end{equation*} 
    for any $x^{\circ}\in S^{u}$ and any $r\in(0,r_{0})$ with $r_{0}$ sufficiently small. In fact, integrating by parts gives 
    \begin{align*}
        0&=\int_{B_{r}(x^{\circ})\cap\{u>0\}}\operatorname{div}\left( \frac{\nabla u}{x_{1}H(\tfrac{|\nabla u|^{2}}{x_{1}^{2}};x_{2})} \right)\,dx\\
        & \leqslant \int_{\partial B_{r}(x^{\circ})\cap\{u>0\}}\frac{|\nabla u|}{x_{1}H(\tfrac{|\nabla u|^{2}}{x_{1}^{2}};x_{2})}\,d\mathcal{H}^{1}\\
        &-\int_{B_{r}(x^{\circ})\cap \partial_{\mathit{red}}\{u>0\} }\frac{|\nabla u|}{x_{1}H(\tfrac{|\nabla u|^{2}}{x_{1}^{2}};x_{2})}\,d\mathcal{H}^{1} \\
        & \leqslant Cr^{3/2}-\frac{1}{\bar{\rho}_{0}}\int_{B_{r}^{+}(x^{\circ})\cap \partial_{\mathit{red}}\{u>0\} }\sqrt{x_{2}^{+}}\,d\mathcal{H}^{1},
    \end{align*}
    where we used $|\nabla u|=x_{1}\sqrt{x_{2}^{+}}$ on $\partial_{\mathit{red}}\{u>0\}$ and $H=\bar{\rho}_{0}$ on $\partial\{u>0\}$.
\end{remark}
\begin{figure}[!ht]
    \centering
    \tikzset{every picture/.style={line width=0.75pt}} 

    \begin{tikzpicture}[x=0.75pt,y=0.75pt,yscale=-1,xscale=1]
    
    \draw    (70.29,109.11) -- (249.29,109.6) ;
    \draw [shift={(251.29,109.61)}, rotate = 180.16] [color={rgb, 255:red, 0; green, 0; blue, 0 }  ][line width=0.75]    (10.93,-3.29) .. controls (6.95,-1.4) and (3.31,-0.3) .. (0,0) .. controls (3.31,0.3) and (6.95,1.4) .. (10.93,3.29)   ;
    \draw [color={rgb, 255:red, 0; green, 0; blue, 255 }  ,draw opacity=1 ]   (151.54,109.11) .. controls (169.43,91.57) and (201.12,85.95) .. (238.83,85.67) ;
    \draw [color={rgb, 255:red, 0; green, 0; blue, 255 }  ,draw opacity=1 ]   (151.54,109.11) .. controls (139.43,95.57) and (115.71,84.71) .. (70.29,83.86) ;
    \draw  [dash pattern={on 4.5pt off 4.5pt}] (191.43,69.86) -- (151.23,108.8) -- (111.71,69) ;
    \draw    (69.29,222.61) -- (248.29,223.1) ;
    \draw [shift={(250.29,223.11)}, rotate = 180.16] [color={rgb, 255:red, 0; green, 0; blue, 0 }  ][line width=0.75]    (10.93,-3.29) .. controls (6.95,-1.4) and (3.31,-0.3) .. (0,0) .. controls (3.31,0.3) and (6.95,1.4) .. (10.93,3.29)   ;
    \draw [color={rgb, 255:red, 0; green, 0; blue, 255 }  ,draw opacity=1 ]   (150.54,222.61) .. controls (189.57,222.86) and (230.31,197.93) .. (246.6,187) ;
    \draw [color={rgb, 255:red, 0; green, 0; blue, 255 }  ,draw opacity=1 ]   (150.54,222.61) .. controls (184.14,222.57) and (223.34,187.4) .. (232.2,177.4) ;
    \draw    (301.29,109.61) -- (480.29,110.1) ;
    \draw [shift={(482.29,110.11)}, rotate = 180.16] [color={rgb, 255:red, 0; green, 0; blue, 0 }  ][line width=0.75]    (10.93,-3.29) .. controls (6.95,-1.4) and (3.31,-0.3) .. (0,0) .. controls (3.31,0.3) and (6.95,1.4) .. (10.93,3.29)   ;
    \draw [color={rgb, 255:red, 0; green, 0; blue, 255 }  ,draw opacity=1 ]   (382.54,109.61) .. controls (347.86,109.86) and (308.66,75.74) .. (301.8,62.6) ;
    \draw [color={rgb, 255:red, 0; green, 0; blue, 255 }  ,draw opacity=1 ]   (382.54,109.61) .. controls (341.4,108.8) and (306.6,90.2) .. (292.2,75) ;
    \draw    (313.29,221.61) -- (492.29,222.1) ;
    \draw [shift={(494.29,222.11)}, rotate = 180.16] [color={rgb, 255:red, 0; green, 0; blue, 0 }  ][line width=0.75]    (10.93,-3.29) .. controls (6.95,-1.4) and (3.31,-0.3) .. (0,0) .. controls (3.31,0.3) and (6.95,1.4) .. (10.93,3.29)   ;
    \draw  [draw opacity=0] (468.1,180.36) .. controls (464.3,185.76) and (459.87,190.88) .. (454.82,195.59) .. controls (418.64,229.33) and (364.15,229.72) .. (333.13,196.45) .. controls (328.06,191.01) and (323.93,185.01) .. (320.72,178.62) -- (398.65,135.36) -- cycle ; \draw  [color={rgb, 255:red, 0; green, 0; blue, 255 }  ,draw opacity=1 ] (468.1,180.36) .. controls (464.3,185.76) and (459.87,190.88) .. (454.82,195.59) .. controls (418.64,229.33) and (364.15,229.72) .. (333.13,196.45) .. controls (328.06,191.01) and (323.93,185.01) .. (320.72,178.62) ;  
    
    \draw (255.5,105.9) node [anchor=north west][inner sep=0.75pt]  [font=\scriptsize]  {$x_{1}$};
    \draw (146.79,109.69) node [anchor=north west][inner sep=0.75pt]  [font=\scriptsize]  {$x_{1}^{\circ }$};
    \draw (141.37,90.83) node [anchor=north west][inner sep=0.75pt]  [font=\tiny]  {$120^{\circ }$};
    \draw (178.29,96.11) node [anchor=north west][inner sep=0.75pt]  [font=\scriptsize]  {$u=0$};
    \draw (100,95.54) node [anchor=north west][inner sep=0.75pt]  [font=\scriptsize]  {$u=0$};
    \draw (139.43,71.69) node [anchor=north west][inner sep=0.75pt]  [font=\scriptsize]  {$u >0$};
    \draw (254.5,219.4) node [anchor=north west][inner sep=0.75pt]  [font=\scriptsize]  {$x_{1}$};
    \draw (145.79,223.19) node [anchor=north west][inner sep=0.75pt]  [font=\scriptsize]  {$x_{1}^{\circ }$};
    \draw (197,212.47) node [anchor=north west][inner sep=0.75pt]  [font=\scriptsize]  {$u=0$};
    \draw (135,207.33) node [anchor=north west][inner sep=0.75pt]  [font=\scriptsize]  {$u=0$};
    \draw (215.92,191.33) node [anchor=north west][inner sep=0.75pt]  [font=\scriptsize,rotate=-322.91]  {$u >0$};
    \draw (486.5,106.4) node [anchor=north west][inner sep=0.75pt]  [font=\scriptsize]  {$x_{1}$};
    \draw (377.79,110.19) node [anchor=north west][inner sep=0.75pt]  [font=\scriptsize]  {$x_{1}^{\circ }$};
    \draw (371.29,91.76) node [anchor=north west][inner sep=0.75pt]  [font=\scriptsize]  {$u=0$};
    \draw (299.34,99.59) node [anchor=north west][inner sep=0.75pt]  [font=\scriptsize]  {$u=0$};
    \draw (292.5,59.91) node [anchor=north west][inner sep=0.75pt]  [font=\scriptsize,rotate=-41.76]  {$u >0$};
    \draw (498.5,218.4) node [anchor=north west][inner sep=0.75pt]  [font=\scriptsize]  {$x_{1}$};
    \draw (385.04,223.44) node [anchor=north west][inner sep=0.75pt]  [font=\scriptsize]  {$x_{1}^{\circ }$};
    \draw (434.54,211.11) node [anchor=north west][inner sep=0.75pt]  [font=\scriptsize]  {$u=0$};
    \draw (317.75,210.29) node [anchor=north west][inner sep=0.75pt]  [font=\scriptsize]  {$u=0$};
    \draw (376.35,196.19) node [anchor=north west][inner sep=0.75pt]  [font=\scriptsize]  {$u >0$};
    \draw (123.8,133.6) node [anchor=north west][inner sep=0.75pt]  [font=\scriptsize] [align=left] {Stokes corner};
    \draw (357.8,131.2) node [anchor=north west][inner sep=0.75pt]  [font=\scriptsize] [align=left] {Cusp (left)};
    \draw (126.37,242.74) node [anchor=north west][inner sep=0.75pt]  [font=\scriptsize] [align=left] {Cusp (right)};
    \draw (358.05,245.95) node [anchor=north west][inner sep=0.75pt]  [font=\scriptsize] [align=left] {Horizontal flatness};

    \end{tikzpicture}
    \caption{Possible asymptotics at stagnation points}
\end{figure}
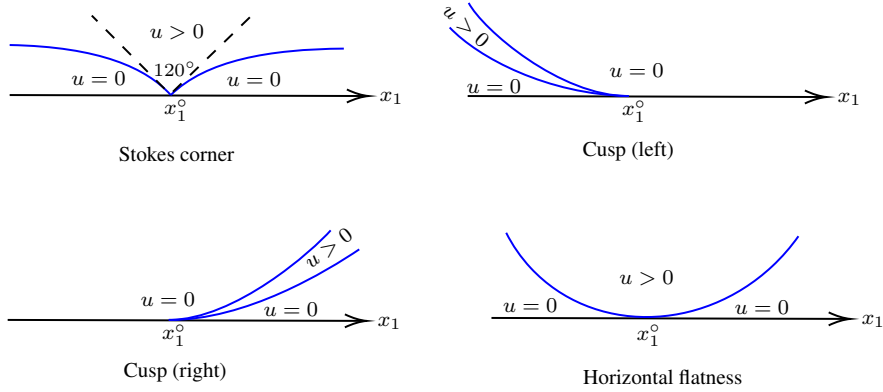
Finally, when assuming that the free surface is a continuous injective curve, we have
\begin{proposition}[Curve case]\label{propcc}
    Let $u$ be a subsonic weak solution of~\eqref{fb}, let $x^{\circ}\in S^{u}$ and suppose that $u$ satisfies~\eqref{gastp}. Assume further that $\partial\{u>0\}\cap B_{r}(x^{\circ})$ is in a neighborhood of $x^{\circ}$ a continuous injective curve $\sigma:I\to \mathbb{R} ^{2}$, where $0\in I$ and $\sigma(0)=(\sigma_{1}(0),\sigma_{2}(0))=x^{\circ}$. Then 
    \begin{enumerate}
        \item [(1).] If $M^{x_{2}}(0^{+})=\frac{\sqrt{3}x_{1}^{\circ}}{3\bar{\rho}_{0}}$, then we have that $\sigma_{1}(t)\neq x_{1}^{\circ}$ in $(-t_{1},t_{1})\setminus\{0\}$ and that either
        \begin{equation*}
            \lim_{t\to 0^{+}}\frac{\sigma_{2}(t)}{\sigma_{1}(t)-x_{1}^{\circ}}=\frac{1}{\sqrt{3}}\quad\text{ and }\quad \lim_{t\to 0^{-}}\frac{\sigma_{2}(t)}{\sigma_{1}(t)-x_{1}^{\circ}}=-\frac{1}{\sqrt{3}},
        \end{equation*}
        or 
        \begin{equation*}
            \lim_{t\to 0^{+}}\frac{\sigma_{2}(t)}{\sigma_{1}(t)-x_{1}^{\circ}}=-\frac{1}{\sqrt{3}}\quad\text{ and }\quad \lim_{t\to 0^{-}}\frac{\sigma_{2}(t)}{\sigma_{1}(t)-x_{1}^{\circ}}=\frac{1}{\sqrt{3}}
        \end{equation*}
        \item [(2).] If $M^{x_{2}}(0^{+})=\frac{2x_{1}^{\circ}}{3\bar{\rho}_{0}}$, then we have that $\sigma_{1}(t)\neq x_{1}^{\circ}$ in $(-t_{1},t_{1})\setminus\{0\}$, $\sigma_{1}(t)-x_{1}^{\circ}$ does change its sign at $t=0$, and 
        \begin{equation*}
            \lim_{t\to 0}\frac{\sigma_{2}(t)}{\sigma_{1}(t)-x_{1}^{\circ}}=0.
        \end{equation*}
        \item [(3).] If $M^{x_{2}}(0^{+})=0$, then $\sigma_{1}(t)\neq x_{1}^{\circ}$ in $(-t_{1},t_{1})\setminus\{0\}$, $\sigma_{1}(t)-x_{1}^{\circ}$ does not change its sign at $t=0$, and 
        \begin{equation*}
            \lim_{t\to 0}\frac{\sigma_{2}(t)}{\sigma_{1}(t)-x_{1}^{\circ}}=0.
        \end{equation*}
    \end{enumerate}
\end{proposition}
The proof is similar to~\cite[Theorem 3.10]{MR4595616} so we omit it. The next lemma claims that if $u$ is a subsonic weak solution which satisfies the strong Bernstein estimates, then under the curve assumption the cusp singularity does not exist.
\begin{lemma}\label{lem:cuspstp}
    Let $u$ be a weak solution of~\eqref{fb} and assume $u$ satisfies the strong Bernstein estimate 
    \begin{equation*}
        \frac{|\nabla u|^{2}}{x_{1}^{2}} \leqslant x_{2}^{+}\quad\text{ in }\Omega.
    \end{equation*}
    Let $x^{\circ}\in S^{u}$. Then $M^{x_{2}}(0^{+})=0$ implies that $u\equiv 0$ in $B_{\delta}(x^{\circ})$ for some $\delta>0$.
\end{lemma}
\begin{proof}
    Suppose towards a contradiction that $x^{\circ}\in \partial\{u>0\}$, and let us take a blowup sequence 
    \begin{equation*}
        u_{m}(x)=\frac{u(x^{\circ}+r_{m}x)}{r_{m}^{3/2}}
    \end{equation*}
    converging weakly in $W_{\mathit{loc}}^{1,2}(\mathbb{R}^{2})$ to a blow-up limit $u_{0}$. Then $M^{x_{2}}(0^{+})= 0$ implies that $u_{0}\equiv 0$ in $\mathbb{R} ^{2}$. Consequently, the nonnegative Radon measure $\operatorname{div}(\tfrac{1}{x_{1}H}\nabla u)$ satisfies 
    \begin{equation*}
        \left( \operatorname{div}\left( \frac{\nabla u_{m}}{(x_{1}^{\circ}+r_{m}x_{1})H(\tfrac{r_{m}|\nabla u_{m}|^{2}}{(x_{1}^{\circ}+r_{m}x_{1})^{2}};r_{m}x_{2})} \right) \right)(B_{2})\to 0,
    \end{equation*}
    and 
    \begin{align*}
        &\left( \operatorname{div}\left( \frac{\nabla u_{m}}{(x_{1}^{\circ}+r_{m}x_{1})H(\tfrac{r_{m}|\nabla u_{m}|^{2}}{(x_{1}^{\circ}+r_{m}x_{1})^{2}};r_{m}x_{2})} \right) \right)(B_{2}) \\
        &\geqslant \int_{B_{2}\cap\partial_{\mathit{red}}\{u_{m}>0\} }\sqrt{x_{2}^{+}}\,d\mathcal{H}^{1}.
    \end{align*}
    On the other hand, there is at least one connected component $V_{m}$ of $\{u_{m}>0\}$ touching the origin and containing, by the maximum principle, a point $x=(x_{1},x_{2})\in\partial A$, where $A=(-1,1)\times(0,1)$. If 
	\begin{align*}
		\max\{x_{2}:x\in V_{m}\cap\partial A\}\not\rightarrow 0\quad\text{ as }m\to+\infty,
	\end{align*}
	we immediately reach a contradiction. On the other hand, if 
	\begin{align*}
		\max\{x_{2}\colon x\in V_{m}\cap\partial A\}\rightarrow 0,
	\end{align*} 
	then 
	\begin{align*}
		0 &=\left( \operatorname{div}\left( \frac{\nabla u_{m}}{(x_{1}^{\circ}+r_{m}x_{1})H_{m}} \right) \right)(V_{m}\cap A)\\
        &
		=\int_{\partial_{\mathrm{red}}V_{m}\cap A}\frac{\nabla u_{m}\cdot\nu}{(x_{1}^{\circ}+r_{m}x_{1})H_{m}}\, d\mathcal{H}^{1}+\int_{V_{m}\cap\partial A}\frac{\nabla u_{m}\cdot\nu}{(x_{1}^{\circ}+r_{m}x_{1})H_{m}}\, d\mathcal{H}^{1}.
	\end{align*}
	Here $H_{m}:=H(\tfrac{r_{m}|\nabla u_{m}|^{2}}{(x_{1}^{\circ}+r_{m}x_{1})^{2}};r_{m}x_{2})$. Given that $u$ is a subsonic weak solution, we infer that $\frac{r_{m}|\nabla u_{m}|^{2}}{(x_{1}^{\circ}+r_{m}x_{2})^{2}}=r_{m}x_{2}$ on the reduced free boundary, and therefore 
    \[
        H_{m}=H(r_{m}x_{2};r_{m}x_{2})=\bar{\rho}_{0}\quad\text{ on }\quad \partial_{\mathrm{red}}V_{m}\cap A.
    \]
    It follows from the growth condition and $t\mapsto H(t;\cdot)$ is a strictly decreasing function that $H_{m}\geqslant \bar{\rho}_{0}$ on $V_{m}\cap\partial A$. Consequently,
	\begin{align*}
		0 &\leqslant-\int_{\partial_{\mathrm{red}}V_{m}\cap A}\frac{|\nabla u_{m}|}{(x_{1}^{\circ}+r_{m}x_{2})H_{m}}\,d\mathcal{H}^{1}+\int_{V_{m}\cap\partial A}\frac{|\nabla u_{m}|}{(x_{1}^{\circ}+r_{m}x_{1})H_{m}}\,d\mathcal{H}^{1}\\
		&\leqslant-\int_{\partial_{\mathrm{red}}V_{m}\cap A}\sqrt{x_{2}^{+}}\,d\mathcal{H}^{1}+\int_{V_{m}\cap\partial A}\sqrt{x_{2}^{+}}\,d\mathcal{H}^{1}.
	\end{align*}
	However, this contradicts to the fact that $\int_{V_{m}\cap\partial A}\sqrt{x_{2}^{+}}dS$ is the unique minimizer of $\int_{\partial D}\sqrt{x_{2}}\,dS$ with respect to all open sets $D$ with $D=V_{m}$ on $\partial A$. Thus $V_{m}$ cannot touch the origin, a contradiction.
\end{proof}
\begin{remark}
    The results established in~\lemref{lemstp},~\lemref{lemwd}, Corollary~\ref{corbl}, and~\propref{propcc} hold under the stronger growth condition
    \begin{equation*}
        \frac{|\nabla u|^{2}}{x_{1}^{2}} \leqslant C\left( |x_{2}|+|x_{1}-x_{1}^{\circ}| \right)\quad\text{ locally in }\Omega.
    \end{equation*}
    However, in our setting, the constraint that $\{u>0\}$ is contained in $\{x_{2} \geqslant 0\}$ (recall the definition of variational solutions (5))—corresponding to the physically motivated scenario where the water region lies below the air region—renders assumption~\eqref{gastp} more natural and better suited to the domain configuration under study.
\end{remark}
\section{The blow-up analysis at the non-stagnation axis  points}\label{sec:ad}
In this section, we aim to analyze the singular asymptotic near the non-stagnation axis  points of the problem \eqref{fb}. We define 
\begin{equation*}
    A^{u}= \left\{ (x_{1},x_{2})\in \Omega\cap\partial\{u>0\}: x_{1}=0, x_{2}>0 \right\}. 
\end{equation*}
For any $x^{\circ}\in A^{u}$, let us consider the rescaling $u_{r}(x):=\frac{u(x^{\circ}+rx)}{r^{2}}$. Similar to~\lemref{lemstp}, we have the following results concerning the limit $\lim_{r\to 0^{+}}u_{r}$.
\begin{lemma}\label{lemad}
    Let $u$ be a subsonic variational solution of~\eqref{fb}, and assume the growth assumption 
    \begin{equation}\label{gasad}
        \left|\frac{|\nabla u|^{2}}{x_{1}^{2}}-x_{2}\right| \leqslant C|x_{1}-x_{1}^{\circ}|\quad\text{ in }B_{r}^{+}(x^{\circ}).
    \end{equation}
    Then for each $x^{\circ}\in A^{u}$,
    \begin{enumerate}
        \item [(1).] The limit $M^{x_{1}}(0^{+}):=\lim_{r\to 0^{+}}M^{x_{1}}(r)$ exists and is finite (Recall equation~\eqref{Mx1r} for the definition of $M^{x_{1}}(r)$).
        \item [(2).] Let $r_{m}\to 0^{+}$ as $m\to\infty$ be a vanishing sequence so that the blow-up sequence 
        \begin{equation}\label{bladp}
            u_{m}(x):=\frac{u(x^{\circ}+r_{m}x)}{r_{m}^{2}}
        \end{equation}
        converges weakly in $W_{\mathit{w},\mathit{loc}}^{1,2}(\mathbb{R} _{+}^{2})$ to a blow-up limit $u_{0}$, then $u_{0}$ is a homogeneous function of degree $2$. Moreover, if $u_{m}$ converges to $u_{0}$ weakly, then the convergence is also strong.
    \end{enumerate}
\end{lemma}
\begin{remark}
    Let $u$ be a subsonic variational solution of~\eqref{fb}. Then the growth assumption~\eqref{gasad} implies that 
    \begin{equation*}
        |\nabla u(x)| \leqslant Cr\quad\text{ and }\quad u(x) \leqslant Cr^{2}\quad\text{ in }B_{r}^{+}(x^{\circ}).
    \end{equation*}
\end{remark}
\begin{proof}
    The idea of the proof follows exactly the same as in~\lemref{lemstp}. 
    We first prove the existence and finiteness of $M^{x_{1}}(0^{+})$. Recalling $K_{1}^{x_{1}}(r)$ defined in the equation~\eqref{K1x1r}, we may deduce from the uniform subsonic condition~\eqref{subs} that 
    \begin{equation}\label{esmad1}
        \left|\partial_{2}F(\tfrac{|\nabla u|^{2}}{x_{1}^{2}};x_{2})\right| \leqslant C\frac{|\nabla u|^{2}}{x_{1}^{2}} \leqslant C\qquad\text{ in }B_{r}^{+}(x^{\circ}).
    \end{equation}
    Similarly, it follows from~\eqref{lbd'} that 
    \begin{equation}\label{esmad2}
        |\lambda'(x_{2})| \leqslant C\quad\text{ in }B_{r}^{+}(x^{\circ}).
    \end{equation}
    Thus, we may deduce from~\eqref{esmad1} and~\eqref{esmad2} that 
    \begin{equation}\label{esmad3}
        |K_{1}^{x_{1}}(r)| \leqslant C\int_{B_{r}^{+}(x^{\circ})}|x_{1}||x_{2}-x_{2}^{\circ}|\,dx \leqslant Cr^{4}.
    \end{equation}
    Moreover, since $H(x_{2};x_{2})=\bar{\rho}_{0}$, we obtain that 
    \begin{align*}
        \begin{split}
            &\frac{1}{x_{1}H(\tfrac{|\nabla u|^{2}}{x_{1}^{2}};x_{2})}-\frac{1}{x_{1}\bar{\rho}_{0}}\\
            &=\frac{1}{x_{1}H(\tfrac{|\nabla u|^{2}}{x_{1}^{2}};x_{2})}-\frac{1}{x_{1}H(x_{2};x_{2})}\\
            &=\frac{1}{x_{1}}\left( \int_{0}^{1}\frac{-\partial_{1}H(\tfrac{\theta|\nabla u|^{2}}{x_{1}^{2}}+(1-\theta)x_{2};x_{2})}{H^{2}(\tfrac{\theta|\nabla u|^{2}}{x_{1}^{2}}+(1-\theta)x_{2};x_{2})}\,d\theta \right)\left( \frac{|\nabla u|^{2}}{x_{1}^{2}}-x_{2} \right).
        \end{split}
    \end{align*} 
    Thus, we have 
    \begin{equation}\label{esmad4}
        |K_{2}^{x_{1}}(r)| \leqslant C\int_{\partial B_{r}^{+}(x^{\circ})}|u||\nabla u|\,d\mathcal{H}^{1} \leqslant Cr^{4},
    \end{equation}
    and 
    \begin{equation}\label{esmad5}
        |K_{3}^{x_{1}}(r)| \leqslant C\int_{\partial B_{r}^{+}(x^{\circ})}u^{2}\,d\mathcal{H}^{1} \leqslant Cr^{4}.
    \end{equation}
    It follows from~\eqref{esmad3},~\eqref{esmad4} and~\eqref{esmad5} that 
    \begin{equation*}
        \left|r^{-4}\sum_{i=1}^{3}K_{i}^{x_{1}}(r)\right| \leqslant C.
    \end{equation*}
    Thus the limit $M^{x_{1}}(0^{+})$ exists and is finite. Now for any $0<\sigma<\tau<\infty$ and any vanishing sequence $r_{m}\to 0^{+}$ as $m\to\infty$, we integrate the equation~\eqref{Mx1r} from $r_{m}\sigma$ to $r_{m}\tau$ to get 
    \begin{align}\label{resclad}
        \begin{split}
            &2\int_{B_{\tau}^{+}\setminus B_{\sigma}^{+}}|x|^{-5}\frac{(\nabla u_{m}\cdot x-2u_{m})^{2}}{x_{1}H(\tfrac{|\nabla u_{m}|^{2}}{x_{1}^{2}};x_{2}^{\circ}+r_{m}x_{2})}\,dx \\
            &=M^{x_{1}}(r_{m}\tau)-M^{x_{1}}(r_{m}\sigma)\\
            &-\int_{r_{m}\sigma}^{r_{m}\tau}r^{-4}\sum_{i=1}^{3}K_{i}^{x_{1}}(r)\,dr.
        \end{split}
    \end{align}
    The existence of $M^{x_{1}}(0^{+})$ and the integrability of $r\mapsto\sum_{i=1}^{3}K_{i}^{x_{1}}(r)$ imply that the right hand side of the equation~\eqref{resclad} converges to $0$ as $m\to\infty$.  We now prove that 
    \begin{equation}\label{limadhm}
        \lim_{m\to\infty}\frac{1}{H\left(\tfrac{|\nabla u_{m}|^{2}}{x_{1}^{2}};x_{2}^{\circ}+r_{m}x_{2}\right)}=\frac{1}{\bar{\rho}_{0}}.
    \end{equation}
    Indeed, since $H(x_{2}^{\circ}+r_{m}x_{2};x_{2}^{\circ}+r_{m}x_{2})=\bar{\rho}_{0}$, we have by a direct calculation that 
    \begin{align}\label{hm-h0ad}
        \begin{split}
            &\frac{1}{H(\tfrac{|\nabla u_{m}|^{2}}{x_{1}^{2}};x_{2}^{\circ}+r_{m}x_{2})}-\frac{1}{H(x_{2}^{\circ}+r_{m}x_{2};x_{2}^{\circ}+r_{m}x_{2})}\\
            &=\int_{0}^{1}\frac{d}{d\theta}\left( \frac{1}{H\left(\tfrac{\theta|\nabla u_{m}|^{2}}{x_{1}^{2}}+(1-\theta)(x_{2}^{\circ}+r_{m}x_{2});x_{2}^{\circ}+r_{m}x_{2}\right)} \right)\,d\theta\\
            &=\int_{0}^{1}\frac{-\partial_{1} H}{H^{2}}\left( \frac{|\nabla u_{m}|^{2}}{x_{1}^{2}}-(x_{2}^{\circ}+r_{m}x_{2}) \right)\,d\theta,
        \end{split}
    \end{align}
    it now follows from~\eqref{gasad} that 
    \begin{equation*}
        \left|\frac{|\nabla u_{m}|^{2}}{x_{1}^{2}}-(x_{2}^{\circ}+r_{m}x_{2})\right| \leqslant Cr_{m}\quad\text{ in }B_{1}^{+}.
    \end{equation*}
    This implies that 
    \begin{equation*}
        \left|\frac{1}{H\left(\tfrac{|\nabla u_{m}|^{2}}{x_{1}^{2}};x_{2}^{\circ}+r_{m}x_{2}\right)}-\frac{1}{H(x_{2}^{\circ}+r_{m}x_{2};x_{2}^{\circ}+r_{m}x_{2})}\right| \leqslant Cr_{m},
    \end{equation*}
    which proves~\eqref{limadhm} by passing to the limit as $m\to\infty$. Therefore, passing to the limit as $m\to\infty$ in equation~\eqref{resclad}, we have that $u_{0}$ is a homogeneous function of degree $2$. It remains to prove that $u_{m}$ converges $u_{0}$ strongly in $W_{\mathit{loc}}^{1,2}$. To this end, we first show that 
    \begin{equation*}
        \frac{1}{H\left(\tfrac{|\nabla u_{m}|^{2}}{x_{1}^{2}};x_{2}^{\circ}+r_{m}x_{2}\right)}\to\frac{1}{\bar{\rho}_{0}}\quad\text{ strongly in }L_{\mathit{loc}}^{2}(\mathbb{R} _{+}^{2}).
    \end{equation*}
    Thanks to~\eqref{limadhm}, it suffices to prove that 
    \begin{equation}\label{cvgadhm1}
        \int_{B_{1}}\left|\frac{1}{H\left(\tfrac{|\nabla u_{m}|^{2}}{x_{1}^{2}};x_{2}^{\circ}+r_{m}x_{2}\right)}-\frac{1}{\bar{\rho}_{0}}\right|\,dx\to 0\quad\text{ as }m\to\infty.
    \end{equation}
    A same calculation as in~\eqref{hm-h0ad} gives that 
    \begin{equation*}
        \int_{B_{1}}\left|\frac{1}{H\left(\tfrac{|\nabla u_{m}|^{2}}{x_{1}^{2}};x_{2}^{\circ}+r_{m}x_{2}\right)}-\frac{1}{\bar{\rho}_{0}}\right|\,dx \leqslant Cr_{m}\omega_{2},
    \end{equation*}
    where $\omega_{2}=|B_{1}|$, and this gives~\eqref{cvgadhm1} by passing to the limit as $m\to\infty$. The strong $L^{2}$ convergence of $H(|\nabla u_{m}|^{2}/x_{1}^{2};x_{2}^{\circ}+r_{m}x_{2})$ to $\bar{\rho}_{0}$ and the weak $L_{\mathit{w}}^{2}$ convergence of $\nabla u_{m}$ give that
    \begin{equation*}
        \lim_{m\to\infty}\int_{\mathbb{R} _{+}^{2}}\frac{u_{m}\nabla u_{m}\cdot\nabla\eta}{x_{1}H\left(\tfrac{|\nabla u_{m}|^{2}}{x_{1}^{2}};x_{2}^{\circ}+r_{m}x_{2}\right)}\,dx=\int_{\mathbb{R} _{+}^{2}}\frac{u_{0}\nabla u_{0}\cdot\nabla\eta}{x_{1}\bar{\rho}_{0}}\,dx,
    \end{equation*}
    for any test function $\eta\in C_{0}^{\infty}(\mathbb{R} ^{2})$. Therefore,
    \begin{align*}
        &o(1)+\int_{\mathbb{R}_{+} ^{2}}\frac{1}{x_{1}\bar{\rho}_{0}}|\nabla u_{m}|^{2}\eta\,dx\\
        &=\int_{\mathbb{R}_{+} ^{2}}\frac{|\nabla u_{m}|^{2}}{x_{1}H(\tfrac{|\nabla u_{m}|^{2}}{x_{1}^{2}};x_{2}^{\circ}+r_{m}x_{2})}\eta\,dx\\
        &=-\int_{\mathbb{R}_{+} ^{2}}\frac{u_{m}\nabla u_{m}\nabla \eta}{x_{1}H(\tfrac{|\nabla u_{m}|^{2}}{x_{1}^{2}};x_{2}^{\circ}+r_{m}x_{2})}\,dx\\
        &=\to-\int_{\mathbb{R}_{+} ^{2}}\frac{u_{0}\nabla u_{0}\cdot\nabla \eta}{x_{1}\bar{\rho}_{0}}\,dx\\
        &=\int_{\mathbb{R}_{+} ^{2}}\frac{1}{x_{1}\bar{\rho}_{0}}|\nabla u_{0}|^{2}\eta\,dx.
    \end{align*}
    Here we used the fact that $\operatorname{div}(\frac{1}{x_{1}}\nabla u_{0})=0$ in $\{u_{0}>0\}$ in the distribution sense.
\end{proof}
We now turn to study the weighted density $M^{x_{1}}(0^{+})$ at the non-stagnation axis points.
\begin{lemma}[Weighted density at the non-stagnation axis points]
    Let $u$ be a subsonic variational solution of~\eqref{fb}, let $x^{\circ}\in A^{u}$ and assume that $u$ satisfies the growth assumption~\eqref{gasad}. Then 
    \begin{enumerate}
        \item [(1).] $M^{x_{1}}(0^{+})$ takes the value 
        \begin{equation}\label{wdenad}
            M^{x_{1}}(0^{+})=\frac{x_{2}^{\circ}}{\bar{\rho}_{0}}\lim_{r\to 0^{+}}r^{-3}\int_{B_{r}^{+}(x^{\circ})}x_{1}\chi_{\left\{ u>0 \right\} }\,dx,
        \end{equation}
        and $M^{x_{1}}(0^{+})=0$ implies that $u_{0}\equiv 0$ in $\mathbb{R} ^{2}$ for $u_{m}$ defined in equation~\eqref{bladp}.
        \item [(2).] Let $u_{m}$ be a blow-up sequence defined in the equation~\eqref{bladp}, let $u_{0}$ be the $W_{\mathit{w},\mathit{loc}}^{1,2}$ strong limit of $u_{m}$ and let $\chi_{0}$ be the $L_{\mathit{loc}}^{1}$ strong limit of $\chi_{\left\{ u_{m}>0 \right\} }$. Then $u_{0}$ is a homogeneous solution of 
        \begin{align}\label{fvadu0}
            \begin{split}
                0&=\frac{1}{\bar{\rho}_{0}}\int_{\mathbb{R} _{+}^{2}}\frac{1}{x_{1}}(|\nabla u_{0}|^{2}\operatorname{div}\phi-2\nabla u_{0}D\phi\nabla u_{0})-\frac{1}{x_{1}^{2}}|\nabla u_{0}|^{2}\phi_{1}\,dx\\
                &+\frac{x_{2}^{\circ}}{\bar{\rho}_{0}}\int_{\mathbb{R} _{+}^{2}}(x_{1}\chi_{0}\operatorname{div}\phi+\chi_{0}\phi_{1})\,dx
            \end{split}
        \end{align}
        for each $\phi=(\phi_{1},\phi_{2})\in C_{0}^{1}(\mathbb{R} ^{2};\mathbb{R} ^{2})$ so that $\phi_{1}=0$ on $\{x_{1}=0\}$.
    \end{enumerate}
\end{lemma}
\begin{proof}
    (1). It follows from the definition of $M^{x_{1}}(r)$ that
    \begin{align}\label{rescalad1}
        \begin{split}
            M^{x_{1}}(rr_{m})&=r^{-3}E_{H}(u_{m};B_{r}^{+})\\
            &-2r^{-4}\int_{\partial B_{r}^{+}}\frac{1}{x_{1}\bar{\rho}_{0}}u_{m}^{2}\,d\mathcal{H}^{1}\\
            &+r^{-3}\int_{B_{r}^{+}}x_{1}\int_{0}^{x_{2}^{\circ}+r_{m}x_{2}}\tfrac{\partial}{\partial\tau}\left( \tfrac{1}{H(\tau;x_{2}^{\circ}+r_{m}x_{2})} \right)\tau\,d\tau\chi_{\left\{ u_{m}>0 \right\} }\,dx\\
            &-r^{-3}\int_{B_{r}^{+}}x_{1}\int_{0}^{\tfrac{|\nabla u_{m}|^{2}}{x_{1}^{2}}}\tfrac{\partial}{\partial\tau}\left( \tfrac{1}{H(\tau;x_{2}^{\circ}+r_{m}x_{2})} \right)\tau\,d\tau dx,
        \end{split}
    \end{align}   
    where 
    \begin{align*}
        E_{H}(u_{m};B_{r}^{+})&=\int_{B_{r}^{+}}x_{1}\Bigg[ \frac{|\nabla u_{m}|^{2}}{x_{1}^{2}H(\tfrac{|\nabla u_{m}|^{2}}{x_{1}^{2}};x_{2}^{\circ}+r_{m}x_{2})}\\
        &\qquad\qquad\qquad+\frac{(x_{2}^{\circ}+r_{m}x_{2})}{\bar{\rho}_{0}}\chi_{\left\{ u_{m}>0 \right\} } \Bigg]\,dx.
    \end{align*}
    Since $\frac{\partial}{\partial\tau}\left( \frac{1}{H(\tau;x_{2})} \right) \geqslant 0$, we have that
    \begin{align*}
        &\int_{B_{r}^{+}}x_{1}\int_{0}^{x_{2}^{\circ}+r_{m}x_{2}}\tfrac{\partial}{\partial\tau}\left( \tfrac{1}{H(\tau;x_{2}^{\circ}+r_{m}x_{2})} \right)\tau\,d\tau\chi_{\left\{ u_{m}>0 \right\} }\,dx\\
        & \leqslant \int_{B_{r}^{+}}x_{1}\int_{0}^{x_{2}^{\circ}+r_{m}x_{2}}\tfrac{\partial}{\partial\tau}\left( \tfrac{1}{H(\tau;x_{2}^{\circ}+r_{m}x_{2})} \right)\tau\,d\tau\,dx.
    \end{align*} 
    It follows from~\eqref{gasad} that 
    \begin{equation*}
        \left|\frac{|\nabla u_{m}|^{2}}{x_{1}^{2}}-(x_{2}^{\circ}+r_{m}x_{2})\right| \leqslant Cr_{m}\qquad\text{ in }B_{r}^{+}.
    \end{equation*}
    Therefore, we obtain  
    \begin{equation*}
        \left|r^{-3}\int_{B_{r}^{+}}x_{1}\int_{0}^{x_{2}^{\circ}+r_{m}x_{2}-\tfrac{|\nabla u_{m}|^{2}}{x_{1}^{2}}}\tfrac{\partial}{\partial\tau}\left( \tfrac{1}{H(\tau;x_{2}^{\circ}+r_{m}x_{2})} \right)\tau\,d\tau\,dx\right| \leqslant Cr_{m}^{2}.
    \end{equation*} 
    Passing to the limit as $m\to\infty$ in equation~\eqref{rescalad1} gives
    \begin{align*}
        M^{x_{1}}(0^{+})&=\frac{1}{\bar{\rho}_{0}}\left(r^{-3}\int_{B_{r}^{+}}\frac{|\nabla u_{0}|^{2}}{x_{1}}\,dx-2r^{-4}\int_{\partial B_{r}^{+}}\frac{1}{x_{1}}u_{0}^{2}\,d\mathcal{H}^{1}\right)\\
        &+\frac{x_{2}^{\circ}}{\bar{\rho}_{0}}\lim_{m\to\infty}r^{-3}\int_{B_{r}^{+}}x_{1}\chi_{\left\{ u_{m}>0 \right\} }\,dx.
    \end{align*}
    Note that 
    \begin{align*}
        \int_{B_{r}^{+}}\frac{|\nabla u_{0}|^{2}}{x_{1}}\,dx&=-\int_{B_{r}^{+}}u_{0}\operatorname{div}\left( \frac{\nabla u_{0}}{x_{1}} \right)\,dx+\int_{\partial B_{r}^{+}}\frac{u_{0}\nabla u_{0}\cdot\nu}{x_{1}}\,d\mathcal{H}^{1}\\
        &=\int_{\partial B_{r}^{+}}\frac{u_{0}\nabla u_{0}\cdot\nu}{x_{1}}\,d\mathcal{H}^{1},
    \end{align*}
    we have 
    \begin{align*}
        M^{x_{1}}(0^{+})&=\frac{r^{-3}}{\bar{\rho}_{0}}\int_{\partial B_{r}^{+}}\frac{u_{0}}{x_{1}}\left[ \nabla u_{0}\cdot\nu-2\frac{u_{0}}{r} \right]\,d\mathcal{H}^{1}\\
        &+\frac{x_{2}^{\circ}}{\bar{\rho}_{0}}\lim_{m\to\infty}r^{-3}\int_{B_{r}^{+}}\frac{x_{1}(x_{2}^{\circ}+r_{m}x_{2})}{\bar{\rho}_{0}}\chi_{\left\{ u_{m}>0 \right\} }\,dx\\
        &=\frac{x_{2}^{\circ}}{\bar{\rho}_{0}}\lim_{m\to\infty}r^{-3}\int_{B_{r}^{+}}\frac{x_{1}(x_{2}^{\circ}+r_{m}x_{2})}{\bar{\rho}_{0}}\chi_{\left\{ u_{m}>0 \right\} }\,dx,
    \end{align*}
    where we used the homogeneity of the blow-up limit $u_{0}$. This proves~\eqref{wdenad}.

    (2). Let now $\phi=(\phi_{1},\phi_{2})\in C_{0}^{1}(\mathbb{R} ^{2};\mathbb{R} ^{2})$ with $\phi_{1}=0$ on $\{x_{1}=0\}$, and we set $\phi_{m}(x)=\phi(\tfrac{x-x^{\circ}}{r_{m}})$. Then it follows from the definition of subsonic variational solution that 
    \begin{align*}
        0&=\int_{\Omega}x_{1}\left[ F( \tfrac{|\nabla u|^{2}}{x_{1}^{2}};x_{2}) + \lambda(x_{2})\chi_{\left\{ u>0 \right\} }  \right] \operatorname{div}\phi_{m}\,dx\\
        &-2\int_{\Omega}\frac{\nabla uD\phi_{m}\nabla u}{x_{1}H( \tfrac{|\nabla u|^{2}}{x_{1}^{2}};x_{2})}\,dx\\
        &+\int_{\Omega}\left[ F(\tfrac{|\nabla u|^{2}}{x_{1}^{2}};x_{2})-\frac{2|\nabla u|^{2}}{x_{1}^{2}H(\tfrac{|\nabla u|^{2}}{x_{1}^{2}};x_{2})}\right]\phi_{1}(\tfrac{x-x^{\circ}}{r_{m}})\,dx\\
        &+\int_{\Omega}\lambda(x_{2})\chi_{\left\{ u>0 \right\} }\phi_{1}\left( \tfrac{x_{1}-x_{1}^{\circ}}{r_{m}} \right)\,dx\\
        &+\int_{\Omega}x_{1}\phi_{2}(\tfrac{x-x^{\circ}}{r_{m}})\left[\partial_{2}F(\tfrac{|\nabla u|^{2}}{x_{1}^{2}};x_{2})+\lambda'(x_{2})\chi_{\left\{ u>0 \right\} } \right]\,dx.
    \end{align*}
    Changing the variables, we have  
    \begin{align*}
        0&=\int_{\Omega_{m}}\frac{|\nabla u_{m}|^{2}}{x_{1}H(\tfrac{|\nabla u_{m}|^{2}}{x_{1}^{2}};x_{2}^{\circ}+r_{m}x_{2})}\operatorname{div}\phi\,dx\\
        &-2\int_{\Omega_{m}}\frac{\nabla u_{m}D\phi\nabla u_{m}}{x_{1}H(\tfrac{|\nabla u_{m}|^{2}}{x_{1}^{2}};x_{2}^{\circ}+r_{m}x_{2})}\,dx\\
        &-\int_{\Omega_{m}}\frac{|\nabla u_{m}|^{2}\phi_{1}}{x_{1}^{2}H(\tfrac{|\nabla u_{m}|^{2}}{x_{1}^{2}};x_{2}^{\circ}+r_{m}x_{2})}\,dx\\
        &+\frac{1}{\bar{\rho}_{0}}\int_{\Omega_{m}}x_{1}(x_{2}^{\circ}+r_{m}x_{2})\chi_{\left\{ u_{m}>0 \right\} }\operatorname{div}\phi\,dx\\
        &+\frac{1}{\bar{\rho}_{0}}\int_{\Omega_{m}}(x_{2}^{\circ}+r_{m}x_{2})\chi_{\left\{ u_{m}>0 \right\} }\phi_{1}\,dx\\
        &+\sum_{i=1}^{3}I_{i}(r_{m}),
    \end{align*}
    where 
    \begin{align*}
        I_{1}(r_{m})=&\int_{\Omega_{m}}x_{1}\int_{0}^{x_{2}^{\circ}+r_{m}x_{2}}\tfrac{\partial}{\partial\tau}\left( \tfrac{1}{H(\tau;x_{2}^{\circ}+r_{m}x_{2})} \right)\tau\,d\tau\chi_{\left\{ u_{m}>0 \right\} }\operatorname{div}\phi\,dx \\
        &-\int_{\Omega_{m}}x_{1}\int_{0}^{\tfrac{|\nabla u_{m}|^{2}}{x_{1}^{2}}}\tfrac{\partial}{\partial\tau}\left( \tfrac{1}{H(\tau;x_{2}^{\circ}+r_{m}x_{2})} \right)\tau\,d\tau \operatorname{div}\phi\,dx, 
    \end{align*}
    and 
    \begin{align*}
        I_{2}(r_{m})=&\int_{\Omega_{m}}x_{1}\int_{0}^{x_{2}^{\circ}+r_{m}x_{2}}\tfrac{\partial}{\partial\tau}\left( \tfrac{1}{H(\tau;x_{2}^{\circ}+r_{m}x_{2})} \right)\tau\,d\tau\chi_{\left\{ u_{m}>0 \right\} }\phi_{1}\,dx \\
        &-\int_{\Omega_{m}}x_{1}\int_{0}^{\tfrac{|\nabla u_{m}|^{2}}{x_{1}^{2}}}\tfrac{\partial}{\partial\tau}\left( \tfrac{1}{H(\tau;x_{2}^{\circ}+r_{m}x_{2})} \right)\tau\,d\tau \phi_{1}\,dx, 
    \end{align*}
    and 
    \begin{align*}
       I_{3}(r_{m})=&r_{m}\int_{\Omega_{m}}x_{1}\phi_{2}\partial_{2}F\left( \tfrac{|\nabla u_{m}|^{2}}{x_{1}^{2}};x_{2}^{\circ}+r_{m}x_{2} \right)\,dx\\
       &+r_{m}\int_{\Omega_{m}}x_{1}\phi_{2}\lambda'(x_{2}^{\circ}+r_{m}x_{2})\phi_{2}\,dx. 
    \end{align*}
    A similar argument as in the proof of (1) gives that $|I_{i}(r_{m})| \leqslant Cr_{m}$ for $i=1$, $2$ and $3$. Therefore, the equation in~\eqref{fvadu0} follows by passing to the limit in the above identity and using the strong convergence $u_{m}$ and $H(\tfrac{|\nabla u_{m}|^{2}}{x_{1}^{2}};x_{2}^{\circ}+r_{m}x_{2})$.
\end{proof}
As a direct corollary, one has 
\begin{corollary}[The blow-up limits for non-stagnation axis points]\label{corblad}
    Let $u$ be a subsonic variational solution of~\eqref{fb}, let $x^{\circ}\in A^{u}$ and let $\delta$ be given as in~\eqref{d}. Suppose that $u$ satisfies the growth condition~\eqref{gasad}. Then 
    \begin{enumerate}
        \item [(1).] The all possible values of weighted density defined in~\eqref{wdenad} are 
        \begin{equation*}
            M^{x_{1}}(0^{+})\in \left\{ \frac{2x_{2}^{\circ}}{3\bar{\rho}_{0}},0 \right\} 
        \end{equation*}
        \item [(2).] If $M^{x_{1}}(0^{+})=\frac{2x_{2}^{\circ}}{3\bar{\rho}_{0}}$, then either 
        \begin{equation*}
            \frac{u(x^{\circ}+rx)}{r^{2}}\to \alpha x_{1}^{2},\qquad \alpha>0,
        \end{equation*}
        as $r\to 0^{+}$, strongly in $W_{\mathit{w},\mathit{loc}}^{1,2}(\mathbb{R} _{+}^{2})$ and locally uniformly in $\mathbb{R} _{+}^{2}$, or 
        \begin{equation*}
            \frac{u(x^{\circ}+rx)}{r^{2}}\to 0\quad\text{ as }r\to 0^{+},
        \end{equation*}
        strongly in $W_{\mathit{w},\mathit{loc}}^{1,2}(\mathbb{R} _{+}^{2})$ and locally uniformly in $\mathbb{R} _{+}^{2}$.
        \item [(3).] If $M^{x_{1}}(0^{+})=0$, then 
        \begin{equation*}
            \frac{u(x^{\circ}+rx)}{r^{2}}\to 0\quad\text{ as }r\to 0^{+},
        \end{equation*}
        strongly in $W_{\mathit{w},\mathit{loc}}^{1,2}(\mathbb{R} _{+}^{2})$ and locally uniformly in $\mathbb{R} _{+}^{2}$.     
    \end{enumerate}
\end{corollary}
\begin{proof}
    The proof is similar to~\cite[Theorem 3.8, case $x_{1}^{\circ}=0$]{MR3225630}. We provide a sketch here. Note that for any $x_{0}\in A^{u}$ and any subsonic solution $u$, the blow-up limit $u_{0}$ satisfies
    \begin{equation*}
        \operatorname{div}\left( \frac{1}{x_{1}}\nabla u_{0} \right)=0\quad\text{ in }\{u_{0}>0\}\text{ in the distribution sense.}
    \end{equation*}
    One can define $w_{0}(x_{1},x_{2})$ (up to a constant) so that 
    \begin{equation*}
        \partial_{1}w_{0}=-\frac{1}{x_{1}}\partial_{2}u_{0}\quad\text{ and }\quad \partial_{2}w_{0}=\frac{1}{x_{1}}\partial_{1}u_{0}.
    \end{equation*}
    It follows that $w_{0}$ is harmonic and homogeneous of degree $1$. Thus, we write $w_{0}(\varrho\sin\theta,\varrho\cos\theta)=\varrho f_{0}(\cos\theta)$ and we obtain that $f_{0}(s)$ satisfies the Legendre differential equation
    \begin{equation*}
        (1-s^{2})f_{0}''(z)-2sf_{0}'(z)+2f(s)=0\quad\text{ in }\{u_{0}>0\},
    \end{equation*}
    where $s=\cos\theta$. It follows that $f_{0}(s)=\alpha P_{1}(s)+\beta Q_{1}(s)$ for some $\alpha$, $\beta\in \mathbb{R} $ and 
    \begin{equation*}
        P_{1}(s)=s,\qquad Q_{1}(s)=\frac{s}{2}\log\left( \frac{1+s}{1-s} \right)-1,\qquad s\in(-1,1).
    \end{equation*}
    Thanks to the requirement
    \begin{equation*}
        \int_{D}x_{1}|\nabla w_{0}|^{2}\,dx=\int_{D}\frac{1}{x_{1}}|\nabla u_{0}|^{2}<+\infty,
    \end{equation*}
    where $D$ is an arbitrary connected component of $\{u_{0}>0\}$. We may deduce that $\beta=0$. This implies that $w_{0}(x_{1},x_{2})=\alpha x_{2}$ and hence $u_{0}=\alpha x_{1}^{2}$. This together with the formula~\eqref{fvadu0} gives the desired result.
\end{proof}
When the free surface is assumed to be a continuous injective curve, the following holds. 
\begin{proposition}[Curve case]\label{propccad}
    Let $u$ be a subsonic weak solution of~\eqref{fb}, let $x^{\circ}\in A^{u}$ and suppose that $u$ satisfies~\eqref{gasad}. Assume further that $\partial\{u>0\}\cap B_{r}^{+}(x^{\circ})$ is in a neighborhood of $x^{\circ}$ a continuous injective curve $\sigma:I\to \mathbb{R}^{2}$, where $0\in I$ and $\sigma(0)=x^{\circ}$. Then if $M^{x_{1}}(0^{+})=\frac{2x_{1}^{\circ}}{3\bar{\rho}_{0}}$ (see Figure~\ref{Fig: cusp1}), then either $\sigma_{2}(t)\neq x_{2}^{\circ}$ in $(0,t_{1})$ and 
    \begin{equation*}
        \lim_{t\to 0^{+}}\frac{\sigma_{1}(t)}{\sigma_{2}(t)-x_{2}^{\circ}}=0,
    \end{equation*}
    or $\sigma_{2}(t)\neq x_{2}^{\circ}$ in $(-t_{1},t_{1})\setminus\{0\}$, $\sigma_{2}-x_{2}^{\circ}$ changes sign at $t=0$ and 
    \begin{equation*}
        \lim_{t\to 0}\frac{\sigma_{1}(t)}{\sigma_{2}(t)-x_{2}^{\circ}}=0.
    \end{equation*}
    If $M^{x_{1}}(0^{+})=0$ (see Figure~\ref{Fig: cusp1}), then $\sigma_{2}(t)\neq x_{2}^{\circ}$ in $(-t_{1},t_{1})\setminus 0$ and 
    \begin{equation*}
        \lim_{t\to 0}\frac{\sigma_{1}(t)}{\sigma_{2}(t)-x_{2}^{\circ}}=0.
    \end{equation*}
\end{proposition}
The proof is similar to~\cite[Theorem 3.8, case $x_{1}^{\circ}=0$]{MR3225630}, so we omit it.
\begin{figure}[!ht]
    \centering
    \tikzset{every picture/.style={line width=0.75pt}} 

    \begin{tikzpicture}[x=0.75pt,y=0.75pt,yscale=-1,xscale=1]
    
    \draw    (187.84,190.32) -- (188.09,39.45) ;
    \draw [shift={(188.09,37.45)}, rotate = 90.09] [color={rgb, 255:red, 0; green, 0; blue, 0 }  ][line width=0.75]    (10.93,-3.29) .. controls (6.95,-1.4) and (3.31,-0.3) .. (0,0) .. controls (3.31,0.3) and (6.95,1.4) .. (10.93,3.29)   ;
    \draw [color={rgb, 255:red, 0; green, 0; blue, 255 }  ,draw opacity=1 ]   (187.82,113.67) .. controls (188.17,75.92) and (198.17,65.67) .. (217.62,56.67) ;
    \draw [color={rgb, 255:red, 0; green, 0; blue, 255 }  ,draw opacity=1 ]   (158.42,56.67) .. controls (177.92,65.67) and (187.42,76.42) .. (187.82,113.67) ;
    \draw    (158.42,56.67) .. controls (182.17,68.42) and (202.17,65.17) .. (217.62,56.67) ;
    \draw  [dash pattern={on 0.84pt off 2.51pt}]  (158.42,56.67) .. controls (171.42,51.17) and (203.17,51.42) .. (217.62,56.67) ;

    \draw [color={rgb, 255:red, 0; green, 0; blue, 255 }  ,draw opacity=1 ]   (187.97,113.88) .. controls (187.62,151.63) and (177.62,161.88) .. (158.17,170.88) ;
    \draw [color={rgb, 255:red, 0; green, 0; blue, 255 }  ,draw opacity=1 ]   (217.37,170.88) .. controls (197.87,161.88) and (188.37,151.13) .. (187.97,113.88) ;
    \draw  [dash pattern={on 0.84pt off 2.51pt}]  (217.37,170.88) .. controls (202.67,163.92) and (172.67,163.42) .. (158.17,170.88) ;
    \draw    (217.37,170.88) .. controls (200.17,180.92) and (176.17,179.42) .. (158.17,170.88) ;
    \draw    (175.67,159.17) .. controls (180.92,162.29) and (193.42,163.92) .. (199.42,159.17) ;
    \draw  [dash pattern={on 0.84pt off 2.51pt}]  (175.67,159.17) .. controls (180.42,156.92) and (192.42,156.17) .. (199.42,159.17) ;
    
    \draw [color={rgb, 255:red, 0; green, 0; blue, 255 }  ,draw opacity=1 ]   (345.15,112) .. controls (345.5,74.25) and (355.5,64) .. (374.95,55) ;
    \draw [color={rgb, 255:red, 0; green, 0; blue, 255 }  ,draw opacity=1 ]   (315.75,55) .. controls (335.25,64) and (344.75,74.75) .. (345.15,112) ;
    \draw    (315.75,55) .. controls (339.5,66.75) and (359.5,63.5) .. (374.95,55) ;
    \draw  [dash pattern={on 0.84pt off 2.51pt}]  (315.75,55) .. controls (328.75,49.5) and (360.5,49.75) .. (374.95,55) ;

    \draw    (345.03,188.44) -- (345.27,37.56) ;
    \draw [shift={(345.28,35.56)}, rotate = 90.09] [color={rgb, 255:red, 0; green, 0; blue, 0 }  ][line width=0.75]    (10.93,-3.29) .. controls (6.95,-1.4) and (3.31,-0.3) .. (0,0) .. controls (3.31,0.3) and (6.95,1.4) .. (10.93,3.29)   ;
    \draw [color={rgb, 255:red, 0; green, 0; blue, 255 }  ,draw opacity=1 ]   (482.3,102.88) .. controls (481.95,140.63) and (471.95,150.88) .. (452.5,159.88) ;
    \draw [color={rgb, 255:red, 0; green, 0; blue, 255 }  ,draw opacity=1 ]   (511.7,159.88) .. controls (492.2,150.88) and (482.7,140.13) .. (482.3,102.88) ;
    \draw  [dash pattern={on 0.84pt off 2.51pt}]  (511.7,159.88) .. controls (497,152.92) and (467,152.42) .. (452.5,159.88) ;
    \draw    (511.7,159.88) .. controls (494.5,169.92) and (470.5,168.42) .. (452.5,159.88) ;
    \draw    (470,148.17) .. controls (475.25,151.29) and (487.75,152.92) .. (493.75,148.17) ;
    \draw  [dash pattern={on 0.84pt off 2.51pt}]  (470,148.17) .. controls (474.75,145.92) and (486.75,145.17) .. (493.75,148.17) ;
    
    \draw    (482.17,180) -- (482.6,36.56) ;
    \draw [shift={(482.61,34.56)}, rotate = 90.17] [color={rgb, 255:red, 0; green, 0; blue, 0 }  ][line width=0.75]    (10.93,-3.29) .. controls (6.95,-1.4) and (3.31,-0.3) .. (0,0) .. controls (3.31,0.3) and (6.95,1.4) .. (10.93,3.29)   ;
    \draw    (263.95,391.8) -- (264.78,228.5) ;
    \draw [shift={(264.79,226.5)}, rotate = 90.29] [color={rgb, 255:red, 0; green, 0; blue, 0 }  ][line width=0.75]    (10.93,-3.29) .. controls (6.95,-1.4) and (3.31,-0.3) .. (0,0) .. controls (3.31,0.3) and (6.95,1.4) .. (10.93,3.29)   ;
    \draw [color={rgb, 255:red, 0; green, 0; blue, 255 }  ,draw opacity=1 ]   (264.21,346.14) .. controls (264.66,283.16) and (279.33,274.6) .. (307.77,256.27) ;
    \draw [color={rgb, 255:red, 0; green, 0; blue, 255 }  ,draw opacity=1 ]   (220.88,254.94) .. controls (239.77,266.16) and (265.1,280.71) .. (264.44,344.05) ;
    \draw [color={rgb, 255:red, 0; green, 0; blue, 255 }  ,draw opacity=1 ]   (264.36,316.86) .. controls (265.07,284.14) and (274.21,263.29) .. (283.64,255.86) ;
    \draw [color={rgb, 255:red, 0; green, 0; blue, 255 }  ,draw opacity=1 ]   (242.79,255.57) .. controls (255.07,266.72) and (264.5,282.72) .. (264.36,316.86) ;
    \draw    (242.79,255.57) .. controls (255.93,262.43) and (273.36,262.14) .. (283.64,255.86) ;
    \draw    (220.88,254.94) .. controls (247.55,270.71) and (285.1,270.05) .. (307.77,256.27) ;
    \draw  [dash pattern={on 0.84pt off 2.51pt}]  (242.79,255.57) .. controls (255.07,251.57) and (274.79,251.57) .. (283.64,255.86) ;
    \draw  [dash pattern={on 0.84pt off 2.51pt}]  (220.88,254.94) .. controls (244.88,244.05) and (283.23,244.11) .. (306.66,255.83) ;
    \draw [color={rgb, 255:red, 0; green, 0; blue, 255 }  ,draw opacity=1 ]   (406.34,260) .. controls (405.9,322.98) and (391.23,331.54) .. (362.79,349.87) ;
    \draw [color={rgb, 255:red, 0; green, 0; blue, 255 }  ,draw opacity=1 ]   (449.67,351.2) .. controls (430.79,339.98) and (405.45,325.43) .. (406.12,262.09) ;
    \draw [color={rgb, 255:red, 0; green, 0; blue, 255 }  ,draw opacity=1 ]   (406.2,289.28) .. controls (405.48,322) and (396.34,342.85) .. (386.91,350.28) ;
    \draw [color={rgb, 255:red, 0; green, 0; blue, 255 }  ,draw opacity=1 ]   (427.77,350.57) .. controls (415.48,339.42) and (406.06,323.42) .. (406.2,289.28) ;
    \draw  [dash pattern={on 0.84pt off 2.51pt}]  (427.77,350.57) .. controls (414.63,343.71) and (397.2,344) .. (386.91,350.28) ;
    \draw  [dash pattern={on 0.84pt off 2.51pt}]  (449.67,351.2) .. controls (423.01,335.43) and (385.76,337.24) .. (362.79,349.87) ;
    \draw    (427.77,350.57) .. controls (412.9,355.81) and (400.9,355.52) .. (386.91,350.28) ;
    \draw    (449.67,351.2) .. controls (430.62,364.1) and (386.62,365.81) .. (363.9,350.31) ;
    \draw    (405.64,391.12) -- (406.47,227.82) ;
    \draw [shift={(406.48,225.82)}, rotate = 90.29] [color={rgb, 255:red, 0; green, 0; blue, 0 }  ][line width=0.75]    (10.93,-3.29) .. controls (6.95,-1.4) and (3.31,-0.3) .. (0,0) .. controls (3.31,0.3) and (6.95,1.4) .. (10.93,3.29)   ;
    
    \draw (184.94,23.32) node [anchor=north west][inner sep=0.75pt]  [font=\scriptsize]  {$x_{2}$};
    \draw (172.42,105.32) node [anchor=north west][inner sep=0.75pt]  [font=\scriptsize]  {$x_{2}^{\circ }$};
    \draw (191.17,105.07) node [anchor=north west][inner sep=0.75pt]  [font=\scriptsize]  {$u >0$};
    \draw (178.98,53.62) node [anchor=north west][inner sep=0.75pt]  [font=\tiny,rotate=-359.94]  {$u=0$};
    \draw (346.83,23.98) node [anchor=north west][inner sep=0.75pt]  [font=\scriptsize]  {$x_{2}$};
    \draw (330.08,105.32) node [anchor=north west][inner sep=0.75pt]  [font=\scriptsize]  {$x_{2}^{\circ }$};
    \draw (350.17,105.07) node [anchor=north west][inner sep=0.75pt]  [font=\scriptsize]  {$u >0$};
    \draw (335.98,52.96) node [anchor=north west][inner sep=0.75pt]  [font=\tiny,rotate=-359.94]  {$u=0$};
    \draw (181.16,167.97) node [anchor=north west][inner sep=0.75pt]  [font=\tiny,rotate=-0.31]  {$u=0$};
    \draw (473.49,156.97) node [anchor=north west][inner sep=0.75pt]  [font=\tiny,rotate=-0.31]  {$u=0$};
    \draw (488.17,105.07) node [anchor=north west][inner sep=0.75pt]  [font=\scriptsize]  {$u >0$};
    \draw (465.75,105.32) node [anchor=north west][inner sep=0.75pt]  [font=\scriptsize]  {$x_{2}^{\circ }$};
    \draw (485.17,21.98) node [anchor=north west][inner sep=0.75pt]  [font=\scriptsize]  {$x_{2}$};
    \draw (268.67,217.39) node [anchor=north west][inner sep=0.75pt]  [font=\scriptsize]  {$x_{2}$};
    \draw (250.03,338.65) node [anchor=north west][inner sep=0.75pt]  [font=\scriptsize]  {$x_{2}^{\circ }$};
    \draw (268.71,334.32) node [anchor=north west][inner sep=0.75pt]  [font=\scriptsize]  {$u=0$};
    \draw (411.01,217.2) node [anchor=north west][inner sep=0.75pt]  [font=\scriptsize]  {$x_{2}$};
    \draw (408.08,257.21) node [anchor=north west][inner sep=0.75pt]  [font=\scriptsize]  {$x_{2}^{\circ }$};
    \draw (267.45,274.58) node [anchor=north west][inner sep=0.75pt]  [font=\tiny,rotate=-316.81]  {$u >0$};
    \draw (253.91,251.02) node [anchor=north west][inner sep=0.75pt]  [font=\tiny,rotate=-359.3]  {$u=0$};
    \draw (411.62,283.84) node [anchor=north west][inner sep=0.75pt]  [font=\scriptsize]  {$u=0$};
    \draw (423.31,331.08) node [anchor=north west][inner sep=0.75pt]  [font=\tiny,rotate=-43.99]  {$u >0$};
    \draw (395.05,344.5) node [anchor=north west][inner sep=0.75pt]  [font=\tiny]  {$u=0$};
    \draw (460.33,179.67) node [anchor=north west][inner sep=0.75pt]  [font=\scriptsize] [align=left] {Down-ward\\cusp};
    \draw (326.67,189) node [anchor=north west][inner sep=0.75pt]  [font=\scriptsize] [align=left] {Up-ward\\cusp};
    \draw (174.33,191.67) node [anchor=north west][inner sep=0.75pt]  [font=\scriptsize] [align=left] {Double\\cusp};
    \draw (249.33,394.67) node [anchor=north west][inner sep=0.75pt]  [font=\scriptsize] [align=left] {Double\\cusp};
    \draw (391,391.66) node [anchor=north west][inner sep=0.75pt]  [font=\scriptsize] [align=left] {Double\\cusp};

    \end{tikzpicture}
    \caption{Cusp}
    \label{Fig: cusp1}
\end{figure}
\section{The blow-up analysis at the origin}\label{sec:0}
In this section, we analyze the final case outlined in the introduction, corresponding to the point $x^{\circ}=(0,0)$. We begin by characterizing the blow-up limit $u_{0}$, establishing the following result.
\begin{lemma}\label{lemo}
    Let $u$ be a subsonic variational solution of~\eqref{fb}, let $0\in\Omega$ and assume the growth assumption~\eqref{gastp}. Then 
    \begin{enumerate}
        \item [(1).] The limit $M^{x_{1}x_{2}}(0^{+}):=\lim_{r\to 0^{+}}M^{x_{1}x_{2}}(r)$ exists and is finite (recall~\eqref{Mx12r} for the definition of $M^{x_{1}x_{2}}(r)$).
        \item [(2).] Let $r_{m}\to 0^{+}$ as $m\to\infty$ be a vanishing sequence so that the blow-up sequence 
        \begin{equation}\label{bl0}
            u_{m}(x):=\frac{u(r_{m}x)}{r_{m}^{5/2}}
        \end{equation}
        converges weakly in $W_{\mathit{w},\mathit{loc}}^{1,2}(\mathbb{R} _{+}^{2})$ to a blow-up limit $u_{0}$, then $u_{m}$ converges strongly to $u_{0}$ in $W_{\mathit{w},\mathit{loc}}^{1,2}(\mathbb{R} _{+}^{2})$ and $u_{0}$ is a homogeneous function of degree $5/2$. 
    \end{enumerate}
\end{lemma}
\begin{proof}
    We first prove the existence of the limit $M^{x_{1}x_{2}}(0^{+})$. It follows from \eqref{subs} that for subsonic variational solution $u$, 
    \begin{equation*}
        \left|\pd{!}{\tau}\left( \frac{1}{H(\tau;x_{2})} \right)\right| \leqslant C\quad\text{ and }\quad\left|\pd{!}{x_{2}}\left( \frac{1}{H(\tau;x_{2})} \right)\right| \leqslant C\quad\text{ in }B_{r}^{+}.
    \end{equation*}
    This, together with the growth assumption gives that 
    \begin{equation}\label{esm02}
        |K_{1}^{x_{1}}(r)|\leqslant C\int_{B_{r}^{+}}|x_{1}|\left(  \int_{0}^{\tfrac{|\nabla u|^{2}}{x_{1}^{2}}}\tau\,d\tau\right)+|x_{1}|\left( \int_{0}^{x_{2}}\tau\,d\tau \right)\,dx \leqslant Cr^{5},
    \end{equation}
    and 
    \begin{equation}\label{esm03}
        |K_{1}^{x_{1}x_{2}}(r)| \leqslant C\int_{B_{r}^{+}}|x_{1}||x_{2}|\left( \frac{|\nabla u|^{2}}{x_{1}^{2}}+|x_{2}| \right)\,dx \leqslant Cr^{5}.
    \end{equation}
    It follows by a direct calculation that~\eqref{esmH} and~\eqref{esmH1} holds in $B_{r}^{+}$, Moreover, the growth assumption~\eqref{gastp} implies that 
    \begin{equation*}
        |\nabla u| \leqslant Cr^{3/2}\quad\text{ and }\quad |u| \leqslant Cr^{5/2}\quad\text{ in }B_{r}^{+}.
    \end{equation*}
    Thus, we can deduce from the equation~\eqref{K2x12r} and equation~\eqref{K3x12r} that 
    \begin{equation}\label{esm04}
        |K_{2}^{x_{1}x_{2}}(r)| \leqslant Cr^{5},
    \end{equation}
    and 
    \begin{equation}\label{esm05}
        |K_{3}^{x_{1}x_{2}}(r)| \leqslant Cr^{5}.
    \end{equation}
    Then the existence of $M^{x_{1}x_{2}}(0^{+})$ follows directly from the estimates \eqref{esm02}, \eqref{esm03},~\eqref{esm04} and~\eqref{esm05}. For any $0<\sigma<\tau<\infty$ and any vanishing $r_{m}\to 0^{+}$, we integrate the equation~\eqref{Mx12r1} from $r_{m}\sigma$ to $r_{m}\tau$ and we get 
    \begin{align*}
        \begin{split}
            &2\int_{B_{\tau}^{+}\setminus B_{\sigma}^{+}}|x|^{-6}\frac{1}{x_{1}H(\tfrac{r_{m}|\nabla u_{m}|^{2}}{x_{1}^{2}};r_{m}x_{2})}\left(\nabla u_{m}\cdot x-\frac{5}{2}u_{m} \right)^{2}\,dx\\
            &=M^{x_{1}x_{2}}(r_{m}\tau)-M^{x_{1}x_{2}}(r_{m}\sigma)\\
            &-\int_{r_{m}\sigma}^{r_{m}\tau}r^{-5}K_{1}^{x_{2}}(r)\,dr\\
            &-\int_{r_{m}\sigma}^{r_{m}\tau}r^{-5}\sum_{i=1}^{3}K_{i}^{x_{1}x_{2}}(r)\,dr.
        \end{split}
    \end{align*}
    Passing to the limit as $m\to\infty$, since
    \begin{equation*}
        \lim_{m\to\infty}H\left(\tfrac{r_{m}|\nabla u_{m}|^{2}}{x_{1}^{2}};r_{m}x_{2}\right)=\bar{\rho}_{0}\quad\text{ as }m\to\infty,
    \end{equation*}
    we obtain that $u_{0}$ is a homogeneous function of degree $5/2$. It follows from
    \begin{align*}
        &\frac{1}{H(\tfrac{r_{m}|\nabla u_{m}|^{2}}{x_{1}^{2}};r_{m}x_{2})}-\frac{1}{\bar{\rho}_{0}}\\
        &=\int_{0}^{1}\frac{d}{d\theta}\left( \frac{1}{H(\tfrac{r_{m}\theta|\nabla u_{m}|^{2}}{x_{1}^{2}};r_{m}x_{2})} \right)\,d\theta\\
        &+\int_{0}^{1}\frac{d}{ds}\left( \frac{1}{H(0;r_{m}sx_{2})} \right)\,ds\\
        &=\int_{0}^{1}\frac{-\partial_{1}H(\tfrac{r_{m}\theta|\nabla u_{m}|^{2}}{x_{1}^{2}};r_{m}x_{2})}{H^{2}(\tfrac{r_{m}\theta|\nabla u_{m}|^{2}}{x_{1}^{2}};r_{m}x_{2})}\,d\theta\left( \frac{r_{m}|\nabla u_{m}|^{2}}{x_{1}^{2}} \right)\\
        &+\int_{0}^{1}\frac{-\partial_{2}H(0;r_{m}sx_{2})}{H^{2}(0;r_{m}sx_{2})}\,ds(r_{m}x_{2})
    \end{align*}
    and a similar argument as in~\lemref{lemstp} (3) that 
    \begin{equation*}
        \frac{1}{H(\tfrac{r_{m}|\nabla u_{m}|^{2}}{x_{1}^{2}};r_{m}x_{2})}\to\frac{1}{\bar{\rho}_{0}}\quad\text{ strongly in }L_{\mathit{loc}}^{2}(\mathbb{R} _{+}^{2}).
    \end{equation*}
    This together with the weak $L_{\mathit{w},\mathit{loc}}^{2}$ convergence of $\nabla u_{m}$ gives 
    \begin{equation*}
        \lim_{m\to\infty}\int_{\mathbb{R} _{+}^{2}}\frac{u_{m}\nabla u_{m}\cdot\nabla\eta}{x_{1}H(\tfrac{r_{m}|\nabla u_{m}|^{2}}{x_{1}^{2}};r_{m}x_{2})}\,dx=\int_{\mathbb{R} _{+}^{2}}\frac{u_{0}\nabla u_{0}\cdot\nabla\eta}{x_{1}\bar{\rho}_{0}}\,dx,
    \end{equation*}
    for any test function $\eta\in C_{0}^{\infty}(\mathbb{R} _{+}^{2})$. Therefore,
    \begin{align*}
        &o(1)+\int_{\mathbb{R}_{+} ^{2}}\frac{1}{x_{1}\bar{\rho}_{0}}|\nabla u_{m}|^{2}\eta\,dx\\
        &=\int_{\mathbb{R}_{+} ^{2}}\frac{|\nabla u_{m}|^{2}}{x_{1}H(\tfrac{r_{m}|\nabla u_{m}|^{2}}{x_{1}^{2}};r_{m}x_{2})}\eta\,dx\\
        &=-\int_{\mathbb{R}_{+} ^{2}}\frac{u_{m}\nabla u_{m}\nabla \eta}{x_{1}H(\tfrac{r_{m}|\nabla u_{m}|^{2}}{x_{1}^{2}};r_{m}x_{2})}\,dx\\
        &\to-\int_{\mathbb{R}_{+} ^{2}}\frac{u_{0}\nabla u_{0}\cdot\nabla \eta}{x_{1}\bar{\rho}_{0}}\,dx\\
        &=\int_{\mathbb{R}_{+} ^{2}}\frac{1}{x_{1}\bar{\rho}_{0}}|\nabla u_{0}|^{2}\eta\,dx.
    \end{align*}
    This gives the strong convergence of $u_{m}$.
\end{proof}
The existence and finiteness of $M^{x_{1}x_{2}}(0^{+})$ and the homogeneity of $u_{0}$ allow us to characterize the weighted density.
\begin{lemma}[Weighted density at the origin]
    Let $u$ be a subsonic variational solution of~\eqref{fb}, assume $u$ satisfies the growth assumption~\eqref{gastp}. Then 
    \begin{enumerate}
        \item [(1).] $M^{x_{1}}(0^{+})$ takes the value 
        \begin{equation}\label{wden0}
            M^{x_{1}x_{2}}(0^{+})=\frac{1}{\bar{\rho}_{0}}\lim_{r\to 0^{+}}r^{-4}\int_{B_{r}^{+}}x_{1}x_{2}\chi_{\left\{ u>0 \right\} }\,dx.
        \end{equation}
        \item [(2).] Let $u_{m}$ be a blow-up sequence defined in equation~\eqref{bl0}, let $u_{0}$ be the $W_{\mathit{w},\mathit{loc}}^{1,2}$ strong limit of $u_{m}$ and $\chi_{0}$ be the $L_{\mathit{loc}}^{1}$ strong limit of $\chi_{\left\{ u_{m}>0 \right\} }$. Then $u_{0}$ is a homogeneous solution of 
        \begin{align}\label{fv0u0}
            \begin{split}
                0&=\frac{1}{\bar{\rho}_{0}}\int_{\mathbb{R} _{+}^{2}}\frac{1}{x_{1}}(|\nabla u_{0}|^{2}\operatorname{div}\phi-2\nabla u_{0}D\phi\nabla u_{0})-\frac{1}{x_{1}^{2}}|\nabla u_{0}|^{2}\phi_{1}\,dx\\
                &+\frac{1}{\bar{\rho}_{0}}\int_{\mathbb{R} _{+}^{2}}(x_{1}x_{2}\chi_{0}\operatorname{div}\phi+x_{2}\chi_{0}\phi_{1}+x_{1}\chi_{0}\phi_{2})\,dx
            \end{split}
        \end{align}
        for each $\phi=(\phi_{1},\phi_{2})\in C_{0}^{1}(\mathbb{R} ^{2};\mathbb{R} ^{2})$ such that $\phi_{1}=0$ on $\{x_{1}=0\}$.
    \end{enumerate}
\end{lemma}
\begin{proof}
    (1). It follows from the definition of $M^{x_{1}x_{2}}(r)$ that for any $r>0$ and any vanishing sequence $r_{m}\to 0^{+}$, 
    \begin{align}\label{rescal01}
        \begin{split}
            M^{x_{1}x_{2}}(rr_{m})&=r^{-4}E_{H}(u_{m};B_{r}^{+})-\frac{5}{2}r^{-5}\int_{\partial B_{r}^{+}}\frac{u_{m}^{2}}{x_{1}\bar{\rho}_{0}}\,d\mathcal{H}^{1}\\
            &-(rr_{m})^{-4}K_{1}^{x_{2}}(rr_{m}),
        \end{split}
    \end{align}
    where 
    \begin{equation*}
        E_{H}(u_{m};B_{r}^{+}):=\int_{B_{r}^{+}}x_{1}\left[ \frac{|\nabla u_{m}|^{2}}{x_{1}^{2}H(\tfrac{r_{m}|\nabla u_{m}|^{2}}{x_{1}^{2}};r_{m}x_{2})}+\frac{x_{2}}{\bar{\rho}_{0}}\chi_{\left\{ u_{m}>0 \right\} } \right]\,dx,
    \end{equation*}
    and 
    \begin{align*}
        K_{1}^{x_{2}}(rr_{m})&=\int_{B_{rr_{m}}^{+}}x_{1}\int_{0}^{\tfrac{r_{m}|\nabla u_{m}|^{2}}{x_{1}^{2}}}\tfrac{\partial}{\partial\tau}\left( \tfrac{1}{H(\tau;r_{m}x_{2})} \right)\,d\tau\,dx\\
        &-\int_{B_{rr_{m}}^{+}}\left( \int_{0}^{r_{m}x_{2}}\tfrac{\partial}{\partial\tau}\left( \tfrac{1}{H(\tau;r_{m}x_{2})} \right) \tau\,d\tau\right)\chi_{\left\{ u_{m}>0 \right\} } dx.
    \end{align*}
    It follows from estimates~\eqref{esm02} that $|K_{1}^{x_{2}}(rr_{m})| \leqslant C(rr_{m})^{5}$ and this implies that 
    \begin{equation*}
        \lim_{m\to \infty}(rr_{m})^{-4}K_{1}^{x_{2}}(rr_{m})=0.
    \end{equation*}
    Passing to the limit in the equation~\eqref{rescal01} gives 
    \begin{align*}
        M^{x_{1}x_{2}}(0^{+})&=\frac{r^{-4}}{\bar{\rho}_{0}}\int_{\partial B_{r}^{+}}\frac{u_{0}}{x_{1}}\left[ \nabla u_{0}\cdot\nu-\frac{5}{2}\frac{u_{0}}{r} \right]\,d\mathcal{H}^{1}\\
        &+\lim_{m\to\infty}\frac{r^{-4}}{\bar{\rho}_{0}}\int_{B_{r}^{+}}x_{1}x_{2}\chi_{\left\{ u_{m}>0 \right\} }\,dx\\
        &=\lim_{m\to\infty}\frac{r^{-4}}{\bar{\rho}_{0}}\int_{B_{r}^{+}}x_{1}x_{2}\chi_{\left\{ u_{m}>0 \right\} }\,dx,
    \end{align*}
    where we used $\operatorname{div}(\tfrac{1}{x_{1}}\nabla u_{0})=0$ in the weak sense and the fact that $u_{0}$ is a homogeneous function of degree $3/2$. This proves~\eqref{wden0}.

    (2). Let $\phi\in C_{0}^{1}(\mathbb{R} ^{2};\mathbb{R} ^{2})$ with $\phi_{1}=0$ on $\{x_{1}=0\}$. Set $\phi_{m}(x)=\phi(\tfrac{x}{r_{m}})$ and we obtain 
    \begin{align*}
        0&=\int_{\Omega_{m}}\frac{|\nabla u_{m}|^{2}}{x_{1}H(\tfrac{r_{m}|\nabla u_{m}|^{2}}{x_{1}^{2}};r_{m}x_{2})}\operatorname{div}\phi\,dx\\
        &-2\int_{\Omega_{m}}\frac{\nabla u_{m}D\phi\nabla u_{m}}{x_{1}H(\tfrac{r_{m}|\nabla u_{m}|^{2}}{x_{1}^{2}};r_{m}x_{2})}\,dx\\
        &-\int_{\Omega_{m}}\frac{|\nabla u_{m}|^{2}\phi_{1}}{x_{1}H(\tfrac{r_{m}|\nabla u_{m}|^{2}}{x_{1}^{2}};r_{m}x_{2})}\,dx\\
        &+\int_{\Omega_{m}}\frac{x_{1}x_{2}}{\bar{\rho}_{0}}\chi_{\left\{ u_{m}>0 \right\} }\operatorname{div}\phi+\frac{x_{2}}{\bar{\rho}_{0}}\chi_{\left\{ u_{m}>0 \right\} }\phi_{1}+\frac{x_{1}}{\bar{\rho}_{0}}\chi_{\left\{ u_{m}>0 \right\} }\phi_{2}\,dx\\
        &+r_{m}^{-1}\int_{\Omega_{m}}x_{1}\left( \int_{0}^{r_{m}x_{2}}\tfrac{\partial}{\partial\tau}\left( \tfrac{1}{H(\tau;r_{m}x_{2})} \right)\tau\,d\tau \right) \chi_{\left\{ u_{m}>0 \right\} } \operatorname{div}\phi\,dx\\
        &-r_{m}^{-1}\int_{\Omega_{m}}\int_{0}^{\tfrac{r_{m}|\nabla u_{m}|^{2}}{x_{1}^{2}}}\tfrac{\partial}{\partial\tau}\left( \tfrac{1}{H(\tau;r_{m}x_{2})} \right)\tau\,d\tau\operatorname{div}\phi\,dx\\
        &+r_{m}^{-1}\int_{\Omega_{m}}\phi_{1}\left( \int_{0}^{r_{m}x_{2}}\tfrac{\partial}{\partial\tau}\left( \tfrac{1}{H(\tau;r_{m}x_{2})} \right)\tau\,d\tau \right)\chi_{\left\{ u_{m}>0 \right\} }\,dx \times\\
        &-r_{m}^{-1}\int_{\Omega_{m}}\int_{0}^{\tfrac{r_{m}|\nabla u_{m}|^{2}}{x_{1}^{2}}}\tfrac{\partial}{\partial\tau}\left( \tfrac{1}{H(\tau;r_{m}x_{2})} \right)\tau\,d\tau\,dx.
    \end{align*}
    Using a similar argument as in~\eqref{esm02} and~\eqref{esm03}, we obtain~\eqref{fv0u0} by passing to the limit as $m\to\infty$ in the above identity.
\end{proof}
Finally, for the blow-up limits of $u_{0}$ near the origin, one has 
\begin{corollary}[The blow-up limits at the origin]
    Let $u$ be a subsonic variational solution of~\eqref{fb}, and let $\delta$ be defined as in~\eqref{d}. Suppose that $u$ satisfies the growth assumption~\eqref{gastp}. Then 
    \begin{enumerate}
        \item [(1).] The only possible values of weighted density defined in~\eqref{wden0} are 
        \begin{equation*}
            M^{x_{1}x_{2}}(0^{+})\in \left\{ \frac{m_{0}}{\bar{\rho}_{0}},\frac{1}{8\bar{\rho}_{0}},0 \right\} 
        \end{equation*}
        where $m_{0}:=\int_{B_{1}^{+}\cap\{\pi-\theta^{*}<\arctan(x_{1}/x_{2})<\pi\} }x_{1}x_{2}\,dx$ and $\theta^{*}=\arccos s^{*}$, $s^{*}\in(-1,0)$ is the unique solution $s\in(-1,1)$ of $P_{3/2}'(s)=0$ ($P_{3/2}$ is the Legendre function of the first kind).
        \item [(2).] If $M^{x_{1}x_{2}}(0^{+})=\frac{m_{0}}{\bar{\rho}_{0}}$, then 
        \begin{equation*}
            \frac{u(rx)}{r^{5/2}}\to \beta_{0}(x_{1}^{2}+x_{2}^{2})^{1/4}P_{3/2}'\left(-\tfrac{x_{2}}{\sqrt{x_{1}^{2}+x_{2}^{2}}}\right)\chi_{\left\{\pi-\theta^{*}<\arctan\left(\tfrac{x_{1}}{x_{2}}\right)<\pi \right\} },
        \end{equation*}
        as $r\to 0^{+}$, strongly in $W_{\mathit{w},\mathit{loc}}^{1,2}(\mathbb{R} _{+}^{2})$ and locally uniformly in $\mathbb{R} _{+}^{2}$.
        \item [(3).] If $M^{x_{1}x_{2}}(0^{+})=\frac{1}{8\bar{\rho}_{0}}$, then 
        \begin{equation*}
            \frac{u(rx)}{r^{5/2}}\to 0\quad\text{ as }r\to 0^{+},
        \end{equation*}
        strongly in $W_{\mathit{w},\mathit{loc}}^{1,2}(\mathbb{R} _{+}^{2})$ and locally uniformly in $\mathbb{R} _{+}^{2}$.
        \item [(4).] If $M^{x_{1}x_{2}}(0^{+})=0$, then 
        \begin{equation*}
            \frac{u(rx)}{r^{5/2}}\to 0\quad\text{ as }r\to 0^{+},
        \end{equation*}
        strongly in $W_{\mathit{w},\mathit{loc}}^{1,2}(\mathbb{R} _{+}^{2})$ and locally uniformly in $\mathbb{R} _{+}^{2}$.
    \end{enumerate}
\end{corollary}
\begin{figure}[!ht]
    \centering
    \tikzset{every picture/.style={line width=0.75pt}} 

    \begin{tikzpicture}[x=0.75pt,y=0.75pt,yscale=-1,xscale=1]
    
    \draw    (218.71,268.14) -- (218.71,73.76) ;
    \draw [shift={(218.71,71.76)}, rotate = 90] [color={rgb, 255:red, 0; green, 0; blue, 0 }  ][line width=0.75]    (10.93,-3.29) .. controls (6.95,-1.4) and (3.31,-0.3) .. (0,0) .. controls (3.31,0.3) and (6.95,1.4) .. (10.93,3.29)   ;
    \draw [color={rgb, 255:red, 144; green, 19; blue, 254 }  ,draw opacity=1 ]   (305.6,120.92) .. controls (277,135.45) and (245.5,145.95) .. (218.42,211.69) ;
    \draw [color={rgb, 255:red, 144; green, 19; blue, 254 }  ,draw opacity=1 ]   (218.42,211.69) .. controls (192.5,138.45) and (145,133.45) .. (124,120.44) ;
    \draw  [dash pattern={on 0.84pt off 2.51pt}]  (305.6,120.92) .. controls (299,102.95) and (153,100.95) .. (124,120.44) ;
    \draw    (305.6,120.92) .. controls (261,140.95) and (163,138.95) .. (124,120.44) ;
    \draw  [dash pattern={on 0.84pt off 2.51pt}]  (246,163.45) .. controls (236,156.95) and (195.35,157.58) .. (189.5,162.45) ;
    \draw    (246,163.45) .. controls (239.5,169.45) and (201,173.45) .. (189.5,162.45) ;
    \draw  [dash pattern={on 4.5pt off 4.5pt}]  (285,67.45) -- (219,209.85) ;
    \draw  [draw opacity=0][fill={rgb, 255:red, 80; green, 227; blue, 194 }  ,fill opacity=1 ] (221.49,201.11) .. controls (224.05,202.37) and (226.03,205.27) .. (226.48,208.85) .. controls (227.11,214.02) and (224.28,218.67) .. (220.16,219.22) .. controls (219.72,219.28) and (219.29,219.29) .. (218.86,219.25) -- (219,209.85) -- cycle ; \draw   (221.49,201.11) .. controls (224.05,202.37) and (226.03,205.27) .. (226.48,208.85) .. controls (227.11,214.02) and (224.28,218.67) .. (220.16,219.22) .. controls (219.72,219.28) and (219.29,219.29) .. (218.86,219.25) ;  
    \draw    (428.97,270.06) -- (428.97,75.68) ;
    \draw [shift={(428.97,73.68)}, rotate = 90] [color={rgb, 255:red, 0; green, 0; blue, 0 }  ][line width=0.75]    (10.93,-3.29) .. controls (6.95,-1.4) and (3.31,-0.3) .. (0,0) .. controls (3.31,0.3) and (6.95,1.4) .. (10.93,3.29)   ;
    \draw [color={rgb, 255:red, 144; green, 19; blue, 254 }  ,draw opacity=1 ]   (518.92,121.56) .. controls (503.6,155.52) and (474.79,213.71) .. (431.74,212.33) ;
    \draw [color={rgb, 255:red, 144; green, 19; blue, 254 }  ,draw opacity=1 ]   (431.74,212.33) .. controls (386.52,216.91) and (349.74,156.16) .. (337.32,121.08) ;
    \draw  [dash pattern={on 0.84pt off 2.51pt}]  (518.92,121.56) .. controls (496.24,110.69) and (385.14,107.65) .. (337.32,121.08) ;
    \draw    (518.92,121.56) .. controls (474.5,141.95) and (393.5,143.45) .. (337.32,121.08) ;
    \draw  [dash pattern={on 0.84pt off 2.51pt}]  (491.34,171.51) .. controls (477,165.45) and (373,164.95) .. (365.5,173.45) ;
    \draw    (491.34,171.51) .. controls (463,194.95) and (381.5,184.45) .. (365.5,173.45) ;
    \draw    (327.52,487.3) -- (327.15,278.3) ;
    \draw [shift={(327.15,276.3)}, rotate = 89.9] [color={rgb, 255:red, 0; green, 0; blue, 0 }  ][line width=0.75]    (10.93,-3.29) .. controls (6.95,-1.4) and (3.31,-0.3) .. (0,0) .. controls (3.31,0.3) and (6.95,1.4) .. (10.93,3.29)   ;
    \draw [color={rgb, 255:red, 144; green, 19; blue, 254 }  ,draw opacity=1 ]   (415.67,327.76) .. controls (383,348.95) and (325.68,396.12) .. (327.52,467.73) ;
    \draw [color={rgb, 255:red, 144; green, 19; blue, 254 }  ,draw opacity=1 ]   (328.25,462.36) .. controls (324.71,398.53) and (276.5,354.95) .. (234.07,327.28) ;
    \draw  [dash pattern={on 0.84pt off 2.51pt}]  (415.67,327.76) .. controls (396.6,307.7) and (263,306.5) .. (234.07,327.28) ;
    \draw    (415.67,327.76) .. controls (342,377.45) and (251.5,339.45) .. (234.07,327.28) ;
    \draw [color={rgb, 255:red, 144; green, 19; blue, 254 }  ,draw opacity=1 ] [dash pattern={on 4.5pt off 4.5pt}]  (287.55,329.75) .. controls (308,350.45) and (327.89,387.94) .. (327.52,467.73) ;
    \draw [color={rgb, 255:red, 144; green, 19; blue, 254 }  ,draw opacity=1 ] [dash pattern={on 4.5pt off 4.5pt}]  (328.25,462.36) .. controls (328.25,374.13) and (346,353.95) .. (370.92,330.39) ;
    \draw    (370.92,330.39) .. controls (352,341.95) and (300,341.45) .. (287.55,329.75) ;
    \draw  [dash pattern={on 0.84pt off 2.51pt}]  (368.34,329.63) .. controls (342.5,319.95) and (314.2,320.9) .. (287.55,329.75) ;
    \draw  [dash pattern={on 4.5pt off 4.5pt}]  (151,77.95) -- (219,209.85) ;
    
    \draw (211.85,50.38) node [anchor=north west][inner sep=0.75pt]    {$x_{2}$};
    \draw (223.28,134.94) node [anchor=north west][inner sep=0.75pt]    {$u >0$};
    \draw (202.98,202.21) node [anchor=north west][inner sep=0.75pt]    {$0$};
    \draw (228.39,204.88) node [anchor=north west][inner sep=0.75pt]    {$\theta ^{*} \approx 114.799^{\circ }$};
    \draw (236.02,176.87) node [anchor=north west][inner sep=0.75pt]    {$u=0$};
    \draw (425.18,51.02) node [anchor=north west][inner sep=0.75pt]    {$x_{2}$};
    \draw (440.23,141.7) node [anchor=north west][inner sep=0.75pt]    {$u >0$};
    \draw (415.08,212.88) node [anchor=north west][inner sep=0.75pt]    {$0$};
    \draw (470.69,202.3) node [anchor=north west][inner sep=0.75pt]    {$u=0$};
    \draw (321.92,257.22) node [anchor=north west][inner sep=0.75pt]    {$x_{2}$};
    \draw (334.36,379.76) node [anchor=north west][inner sep=0.75pt]  [rotate=-309.51]  {$u >0$};
    \draw (311.22,448.74) node [anchor=north west][inner sep=0.75pt]    {$0$};
    \draw (356.64,382.77) node [anchor=north west][inner sep=0.75pt]    {$u=0$};
    \draw (164,272) node [anchor=north west][inner sep=0.75pt]  [font=\small] [align=left] {Garabedian pointed\\bubble};
    \draw (371,277) node [anchor=north west][inner sep=0.75pt]  [font=\small] [align=left] {Horizontal flatness};
    \draw (312,490) node [anchor=north west][inner sep=0.75pt]  [font=\small] [align=left] {Cusp};

    \end{tikzpicture}
\end{figure}
When the free boundary is assumed to be a curve. Then
\begin{proposition}[Curve case]
    Let $u$ be a subsonic weak solution of~\eqref{fb} with~\eqref{gastp}. Assume further that $\partial\{u>0\}\cap B_{r}^{+}$ is in a neighborhood of $0$ a continuous injective curve $\sigma:I\to \mathbb{R} ^{2}$, where $0\in I$ and $\sigma(0)=0$. Then,
    \begin{enumerate}
        \item [(1).] If $M^{x_{1}}(0^{+})=\frac{m_{0}}{\bar{\rho}_{0}}$, then $\sigma_{1}(t)\neq 0$ in $(0,t_{1})$ and 
        \begin{equation*}
            \lim_{t\to 0^{+}}\frac{\sigma_{2}(t)}{\sigma_{1}(t)}=\cot(\pi-\theta^{*}),
        \end{equation*}
        where $\theta^{*}:=\arccos s^{*}$ where $s^{*}\in(-1,0)$ such that $P'_{3/2}(s^{*})=0$.
        \item [(2).] If $M^{x_{1}x_{2}}(0^{+})=\frac{1}{8\bar{\rho}_{0}}$, then $\sigma_{1}(t)\neq 0$ in $(0,t_{1})$ and 
        \begin{equation*}
            \lim_{t\to 0^{+}}\frac{\sigma_{2}(t)}{\sigma_{1}(t)}=0.
        \end{equation*}
        \item [(3).] If $M^{x_{1}x_{2}}(0^{+})=0$, then $\sigma_{1}(t)\neq 0$ in $(-t_{1},t_{1})\setminus\{0\}$ and 
        \begin{equation*}
            \lim_{t\to 0}\frac{\sigma_{2}(t)}{\sigma_{1}(t)}=0.
        \end{equation*}
    \end{enumerate}
\end{proposition}
\section{A nonlinear frequency formula at the origin}\label{Sec:fre}
In this section, we examine the degenerate point at the origin. Specifically, we consider the case where $u_{0}=0$, 
\begin{equation*}
    M^{x_{1}x_{2}}(0^{+})=\frac{1}{\bar{\rho}_{0}}\int_{B_{1}^{+}}x_{1}x_{2}^{+}\,dx=\frac{1}{8\bar{\rho}_{0}},
\end{equation*}
and $u=0$ on $\{x_{2} \leqslant 0\}$. We refer to such a point as a \textit{trivial original with non-zero density}. The remainder of this paper is dedicated to investigating the qualitative properties of such points, such as finding a nontrivial solution of homogeneous higher than $5/2$.

We begin by noting the following key observation. Define for any $r\in(0,\delta)$ the function
\begin{align}\label{Mx12rp}
    \begin{split}
        \widetilde{M}^{x_{1}x_{2}}(r)&:=M^{x_{1}x_{2}}(r)\\
        &-\int_{0}^{r}t^{-5}K_{1}^{x_{2}}(t)\,dt-\int_{0}^{r}t^{-5}\sum_{i=1}^{3}K_{i}^{x_{1}x_{2}}(t)\,dt,
    \end{split}
\end{align}
where $M^{x_{1}x_{2}}(r)$, $K_{1}^{x_{2}}(r)$ and $K_{i}^{x_{1}x_{2}}(r)$ for $i=1$, $2$, $3$ are defined in~\propref{prop0}. It follows from equation~\eqref{Mx12r1} that   
\begin{align*}
    (\widetilde{M}^{x_{1}x_{2}}(r))'&=(M^{x_{1}x_{2}}(r))'-r^{-5}K_{1}^{x_{2}}(r)-r^{-5}\sum_{i=1}^{3}K_{i}^{x_{1}x_{2}}(r)\\
    &=2r^{-4}\int_{\partial B_{r}^{+}}\frac{\left( \nabla u\cdot\nu-\frac{5}{2}\frac{u}{r} \right)^{2}}{x_{1}H(\tfrac{|\nabla u|^{2}}{x_{1}^{2}};x_{2})}\,d\mathcal{H}^{1} \geqslant 0.
\end{align*}
Here we used the integrability of the functions $r\mapsto r^{-5}K_{1}^{x_{2}}(r)$ and $r\mapsto r^{-5}\sum_{i=1}^{3}K_{i}^{x_{1}x_{2}}(r)$, which is proved in~\lemref{lemo} (1). Therefore, $r\mapsto\widetilde{M}^{x_{1}x_{2}}(r)$ is a nondecreasing function defined for $r\in(0,\delta)$. Moreover, it is easy to deduce that 
\begin{equation}\label{Mx12rp1}
    \widetilde{M}^{x_{1}x_{2}}(r) \geqslant \widetilde{M}^{x_{1}x_{2}}(0^{+})=M^{x_{1}x_{2}}(0^{+})=\frac{1}{\bar{\rho}_{0}}r^{-4}\int_{B_{r}^{+}}x_{1}x_{2}^{+}\,dx.
\end{equation}
The inequality in~\eqref{Mx12rp1} and the definition of $\widetilde{M}^{x_{1}x_{2}}(r)$ in the equation~\eqref{Mx12rp} imply 
\begin{align}\label{freqf1}
    \begin{split}
        &r^{-4}\int_{B_{r}^{+}}\frac{|\nabla u|^{2}}{x_{1}H(\tfrac{|\nabla u|^{2}}{x_{1}^{2}};x_{2})}\,dx-\frac{5}{2}r^{-5}\int_{\partial B_{r}^{+}}\frac{1}{x_{1}\bar{\rho}_{0}}u^{2}\,d\mathcal{H}^{1}\\
        & \geqslant \frac{1}{\bar{\rho}_{0}}r^{-4}\int_{B_{r}^{+}}x_{1}x_{2}^{+}\left( 1-\chi_{\left\{ u>0 \right\} } \right)\,dx\\
        &+r^{-4}\int_{B_{r}^{+}}x_{1}\left[ \frac{|\nabla u|^{2}}{x_{1}^{2}H(\tfrac{|\nabla u|^{2}}{x_{1}^{2}};x_{2})}-F(\tfrac{|\nabla u|^{2}}{x_{1}^{2}};x_{2}) \right]\,dx\\
        &+r^{-4}\int_{B_{r}^{+}}x_{1}\left[ \frac{x_{2}^{+}}{\bar{\rho}_{0}} -\lambda(x_{2})\right]\chi_{\left\{ u>0 \right\} }\,dx\\
        &+\int_{0}^{r}t^{-5}K_{1}^{x_{2}}(t)\,dt+\int_{0}^{r}t^{-5}\sum_{i=1}^{3}K_{i}^{x_{1}x_{2}}(t)\,dt.
    \end{split}
\end{align}
Inspired by this inequality, we have 
\begin{lemma}\label{lemma:freq}
    Let $u$ be a subsonic variational solution of~\eqref{fb}, and let $\delta$ be defined as in~\eqref{d}. Assume that $0$ is a trivial original with non-zero density and that $u$ satisfies the growth assumption 
    \begin{equation}\label{Bern}
        \frac{|\nabla u|^{2}}{x_{1}^{2}} \leqslant x_{2}^{+}\qquad\text{ locally in }\Omega.
    \end{equation}
    Define for a.e. $r\in(0,\delta)$ the functions 
    \begin{equation*}
        D(r)=\frac{\displaystyle r\int_{B_{r}^{+}}\frac{|\nabla u|^{2}}{x_{1}H(\tfrac{|\nabla u|^{2}}{x_{1}^{2}};x_{2})}\,dx}{\displaystyle\int_{\partial B_{r}^{+}}\frac{u^{2}}{x_{1}\bar{\rho}_{0}}\,d\mathcal{H}^{1}},
    \end{equation*}
    and 
    \begin{equation*}
        V(r)=\frac{\displaystyle r\frac{1}{\bar{\rho}_{0}}\int_{B_{r}^{+}}x_{1}x_{2}^{+}\left( 1-\chi_{\left\{ u>0 \right\} } \right)\,dx+e(r)}{\displaystyle\int_{\partial B_{r}^{+}}\frac{u^{2}}{x_{1}\bar{\rho}_{0}}\,d\mathcal{H}^{1}},
    \end{equation*}
    where 
    \begin{equation*}
        e(r):=rK_{1}^{x_{2}}(r)+r^{5}\int_{0}^{r}t^{-5}K_{1}^{x_{2}}(t)\,dt+r^{5}\int_{0}^{r}\sum_{i=1}^{3}t^{-5}K_{i}^{x_{1}x_{2}}(t)\,dt.
    \end{equation*}
    Then we have 
    \begin{equation}\label{Dr-Vr}
        D(r)-V(r) \geqslant \frac{5}{2}\quad\text{ for all }r\in(0,\delta_{0}).
    \end{equation}
    Moreover, there exists $r_{0}\in(0,\delta)$ sufficiently small so that 
    \begin{equation}\label{Vrge0}
        V(r) \geqslant 0\qquad\text{ for all }r\in(0,r_{0}).
    \end{equation}
\end{lemma}
\begin{proof}
    The inequality in~\eqref{Dr-Vr} can be obtained directly from the inequality in~\eqref{freqf1}. Indeed, dividing the inequality~\eqref{freqf1} on both sides by the nonnegative term $\frac{1}{\bar{\rho}_{0}}r^{-4}\int_{B_{r}^{+}}x_{1}x_{2}^{+}\chi_{\left\{ u>0 \right\} }\,dx$ yields the inequality~\eqref{Dr-Vr}. It remains to prove that $V(r)$ is nonnegative for all $r\in(0,r_{0})$ with $r_{0}>0$ sufficiently small. It follows from $\tfrac{\partial}{\partial\tau}(\tfrac{1}{H(\tau;x_{2})}) \geqslant 0$ and the growth assumption~\eqref{Bern} that 
    \begin{equation*}
        \int_{0}^{\tfrac{|\nabla u|^{2}}{x_{1}^{2}}}\tfrac{\partial}{\partial\tau}\left( \tfrac{1}{H(\tau;x_{2})} \right)\tau\,d\tau \leqslant \int_{0}^{x_{2}}\tfrac{\partial}{\partial\tau}\left( \tfrac{1}{H(\tau;x_{2})} \right)\tau\,d\tau.
    \end{equation*}
    Recall~\eqref{EH-EF}, we have that 
    \begin{align*}
        |K_{1}^{x_{2}}(r)|& \leqslant \int_{B_{r}^{+}}x_{1}\int_{0}^{x_{2}}\left|\tfrac{\partial}{\partial\tau}\left( \tfrac{1}{H(\tau;x_{2})} \right)\tau\,d\tau\right|\left( 1-\chi_{\left\{ u>0 \right\} } \right)\,dx.
    \end{align*}
    Thus, one has 
    \begin{equation}\label{0}
        0 \leqslant  r^{-4}|K_{1}^{x_{2}}(r)| \leqslant Cr^{-3}\int_{B_{r}^{+}}x_{1}x_{2}^{+}\left( 1-\chi_{\left\{ u>0 \right\} } \right)\,dx.
    \end{equation} 
    A similar argument for the term $\int_{0}^{r}t^{-5}K_{1}^{x_{2}}(t)\,dt$ gives that 
    \begin{equation}\label{1}
        0 \leqslant \int_{0}^{r}t^{-5}|K_{1}^{x_{2}}(t)|\,dt \leqslant C\int_{0}^{r}t^{-4}\int_{B_{t}^{+}}x_{1}x_{2}^{+}\left( 1-\chi_{\left\{ u>0 \right\} } \right)\,dx.
    \end{equation}
    Moreover, it follows from~\eqref{K1x12r}  that  
    \begin{align*}
        &\int_{0}^{r}t^{-5}K_{1}^{x_{1}x_{2}}(t)\,dt\\
        &=\int_{0}^{r}t^{-5}\int_{B_{t}^{+}}x_{1}x_{2}^{+}\int_{0}^{\tfrac{|\nabla u|^{2}}{x_{1}^{2}}}\tfrac{\partial}{\partial x_{2}}\left( \tfrac{1}{H(\tau;x_{2})} \right)\,d\tau(1-\chi_{\left\{ u>0 \right\} })\,dx\\
        &+\int_{0}^{r}t^{-5}\int_{B_{t}^{+}}x_{1}x_{2}^{+}\Bigg[ \int_{0}^{\tfrac{|\nabla u|^{2}}{x_{1}^{2}}}\tfrac{\partial}{\partial x_{2}}(\tfrac{1}{H(\tau;x_{2})})\,d\tau\\
        &\qquad\qquad\qquad\qquad-\int_{0}^{x_{2}}\tfrac{\partial}{\partial x_{2}}(\tfrac{1}{H(\tau;x_{2})})\,d\tau \Bigg]\chi_{\left\{ u>0 \right\} }\,dx.
    \end{align*}
    Therefore, we may deduce from $-C \leqslant \tfrac{\partial}{\partial x_{2}}(\tfrac{1}{H(\tau;x_{2})}) \leqslant 0$ that
    \begin{equation*}
        \int_{0}^{\tfrac{|\nabla u|^{2}}{x_{1}^{2}}}\tfrac{\partial}{\partial x_{2}}(\tfrac{1}{H(\tau;x_{2})})\,d\tau-\int_{0}^{x_{2}}\tfrac{\partial}{\partial x_{2}}(\tfrac{1}{H(\tau;x_{2})})\,d\tau \geqslant 0. 
    \end{equation*}
    Thus, one has 
    \begin{align}\label{2}
        \begin{split}
            &\int_{0}^{r}t^{-5}K_{1}^{x_{2}x_{2}}(t)\,dt \\
            &\geqslant \int_{0}^{r}t^{-5}\int_{B_{t}^{+}}x_{1}x_{2}^{+}\int_{0}^{\tfrac{|\nabla u|^{2}}{x_{1}^{2}}}\tfrac{\partial}{\partial x_{2}}\left( \tfrac{1}{H(\tau;x_{2})} \right)\,d\tau(1-\chi_{\left\{ u>0 \right\} })\,dx \\
            &\geqslant -C\int_{0}^{r}t^{-4}\int_{B_{t}^{+}}x_{1}x_{2}^{+}(1-\chi_{\left\{ u>0 \right\} })\,dx dt.
        \end{split} 
     \end{align}
     Let us now consider 
     \begin{equation*}
         \Pi(r)=\int_{0}^{r}t^{-4}\int_{B_{t}^{+}}x_{1}x_{2}^{+}\left( 1-\chi_{\left\{ u>0 \right\} } \right)
         \,dx dt.
     \end{equation*}
     It follows that $\Pi(r) \geqslant 0$ and $\Pi(0)=0$. Moreover, a direct calculation gives that 
     \begin{equation*}
         \Pi'(r)=r^{-4}\int_{B_{r}^{+}}x_{1}x_{2}^{+}(1-\chi_{\left\{ u>0 \right\} })\,dx.
    \end{equation*}
    Since $\chi_{\left\{ u_{r}>0 \right\} }\to 1$ as $r\to 0^{+}$, we have that 
    \begin{equation*}
        \lim_{r\to 0^{+}}\Pi'(r)=\lim_{r\to 0^{+}}\int_{B_{1}^{+}}x_{1}x_{2}^{+}(1-\chi_{\left\{ u_{r}>0 \right\} })\,dx=0.
    \end{equation*}
    A direct calculation gives that 
    \begin{align*}
        \Pi''(r)&=\pd{!}{r}\left( r^{-4}\int_{B_{r}^{+}}x_{1}x_{2}^{+}\chi_{\left\{ u>0 \right\} }\,dx \right)\\
        &=4r^{-5}\int_{B_{r}^{+}}x_{1}x_{2}^{+}\chi_{\left\{ u>0 \right\} }\,dx-r^{-4}\int_{\partial B_{r}^{+}}x_{1}x_{2}^{+}\chi_{\left\{ u>0 \right\} }\,d\mathcal{H}^{1}\\
        &=r^{-1}\left[ 4\int_{B_{1}^{+}}x_{1}x_{2}^{+}\chi_{\left\{ u_{r}>0 \right\} }\,dx-\int_{\partial B_{1}^{+}}x_{1}x_{2}^{+}\chi_{\left\{ u_{r}>0 \right\} }\,d\mathcal{H}^{1} \right].
    \end{align*}
    Let us consider the $\tfrac{5}{2}$ homogeneous replacement of $u_{r}$ defined by 
    \begin{equation*}
        z_{r}(x):=|x|^{5/2}u_{r}(\tfrac{x}{|x|})=\frac{|x|^{5/2}}{r^{5/2}}u\left( r\frac{x}{|x|} \right).
    \end{equation*}
    In terms of the polar coordinates, we have by a direct calculation that 
    \begin{align*}
        \int_{B_{1}\cap\{z_{r}>0\}}x_{1}x_{2}^{+}\,dx&=\int_{0}^{1}\int_{\mathbb{S}^{1}\cap\{z_{r}>0\} }\rho\sin\theta\rho(\cos\theta)^{+}\rho\,d\rho d\theta\\
        &=\int_{0}^{1}\rho^{3}\,d\rho\int_{\mathbb{S}^{1}\cap\{z_{r}>0\} }\sin\theta(\cos\theta)^{+} d\theta\\
        &=\frac{1}{4}\int_{\partial B_{1}^{+}\cap\{u_{r}>0\}}x_{1}x_{2}^{+}\,d\mathcal{H}^{1}.
    \end{align*}
    This implies 
    \begin{equation*}
        \Pi''(r)=4r^{-1}\left[ \int_{B_{1}^{+}}x_{1}x_{2}^{+}\chi_{\left\{ u_{r}>0 \right\} } -\int_{B_{1}^{+}}x_{1}x_{2}^{+}\chi_{\left\{ z_{r}>0 \right\} }\,dx\right].
    \end{equation*}
    Note that if $x\in \{u_{r}<z_{r}\}\cap B_{1}^{+}$, then $\Pi''(r)<0$ and this implies that $\Pi(r)<r\Pi'(0)=0$ for $r\in(0,\delta)$. However this contradicts to the fact that $\Pi(r) \geqslant 0$. Thus, $u_{r} \geqslant z_{r}$ in $B_{1}^{+}$ and $\Pi''(r) \geqslant 0$. This gives that 
    \begin{equation}\label{Pir1}
        0 \leqslant \Pi(r)=\Pi(r)-\Pi(0)=\Pi'(\tilde{r})r \leqslant \Pi'(r)r=r^{-3}\int_{B_{r}^{+}}x_{1}x_{2}^{+}\left( 1-\chi_{\left\{ u>0 \right\} } \right)\,dx.
    \end{equation}
    This, combined with~\eqref{1} and~\eqref{2} gives that 
    \begin{equation}\label{4}
        \int_{0}^{r}t^{-5}(K_{1}^{x_{2}}(t)+K_{1}^{x_{1}x_{2}}(t))\,dt \geqslant -Cr^{-3}\int_{B_{r}^{+}}x_{1}x_{2}^{+}\left( 1-\chi_{\left\{ u>0 \right\} } \right)\,dx.
    \end{equation}
    We now consider the $\int_{0}^{r}t^{-5}(K_{2}^{x_{1}x_{2}}(t)+K_{3}^{x_{1}x_{2}}(t))\,dt$. Under the growth  assumption $\tfrac{|\nabla u|^{2}}{x_{1}^{2}} \leqslant x_{2}^{+}$, we deduce from $\partial_{1}H<0$ that $H(\tfrac{|\nabla u|^{2}}{x_{1}^{2}};x_{2}) \geqslant  H(x_{2};x_{2})=\bar{\rho}_{0}$. This implies 
    \begin{equation}\label{Bern1}
        \frac{1}{\bar{\rho}_{0}}-\frac{1}{H(\tfrac{|\nabla u|^{2}}{x_{1}^{2}};x_{2})} \geqslant 0\quad\text{ in }B_{r}^{+}.
    \end{equation}
    Therefore, it follows from the definition of $K_{2}^{x_{1}x_{2}}(r)$ and $K_{3}^{x_{1}x_{2}}(r)$ in equations~\eqref{K2x12r} and~\eqref{K3x12r} that $K_{3}^{x_{1}x_{2}}(r) \geqslant 0$ for all $r\in(0,\delta)$ and $K_{2}^{x_{1}x_{2}}(r) \geqslant 0$ in $\partial B_{r}^{+}\cap\{u_{\nu}\leqslant 0\}$. It suffices to consider the case when $(K_{2}^{x_{1}x_{2}}(r)+K_{3}^{x_{1}x_{2}}(r))$ in $\partial B_{r}^{+}\cap\{u_{\nu} \geqslant 0\}$. A direct calculation gives that 
    \begin{align*}
        \begin{split}
            &K_{2}^{x_{1}x_{2}}(r)+K_{3}^{x_{1}x_{2}}(r)\\
            &\geqslant 5\int_{\partial B_{r}^{+}\cap\{u_{\nu} \geqslant 0\}}\left( \frac{1}{x_{1}H(\tfrac{|\nabla u|^{2}}{x_{1}^{2}};x_{2})}-\frac{1}{x_{1}\bar{\rho}_{0}} \right)u\left( \nabla u\cdot\nu-\frac{5}{2}\frac{u}{r} \right)d\mathcal{H}^{1}\\
            &=\frac{25}{2}\int_{\partial B_{r}^{+}\cap\{u_{\nu} \geqslant 0\}}\left(\frac{1}{x_{1}\bar{\rho}_{0}}-\frac{1}{x_{1}H(\tfrac{|\nabla u|^{2}}{x_{1}^{2}};x_{2})} \right)\frac{u^{2}}{r}\,d\mathcal{H}^{1}\\
            &-5\int_{\partial B_{r}^{+}\cap\{u_{\nu} \geqslant 0\}}\left(\frac{1}{x_{1}\bar{\rho}_{0}}-\frac{1}{x_{1}H(\tfrac{|\nabla u|^{2}}{x_{1}^{2}};x_{2})} \right)u\nabla u\cdot\nu\, d\mathcal{H}^{1}\\
            &\geqslant\frac{25}{2}\int_{\partial B_{r}^{+}\cap\{u_{\nu} \geqslant 0\}}\left(\frac{1}{x_{1}\bar{\rho}_{0}}-\frac{1}{x_{1}H(\tfrac{|\nabla u|^{2}}{x_{1}^{2}};x_{2})}\right)\frac{u^{2}}{r}\chi_{\left\{ u>0 \right\} }\,d\mathcal{H}^{1}\\
            &-5\int_{\partial B_{r}^{+}\cap\{u_{\nu} \geqslant 0\}}\left(\frac{1}{x_{1}\bar{\rho}_{0}}-\frac{1}{x_{1}H(\tfrac{|\nabla u|^{2}}{x_{1}^{2}};x_{2})}\right)u\nabla u\cdot\nu\, d\mathcal{H}^{1}\\
            &=\frac{25}{2}\int_{\partial B_{r}^{+}\cap\{u_{\nu} \geqslant 0\}}\left(\frac{1}{x_{1}\bar{\rho}_{0}}-\frac{1}{x_{1}H(\tfrac{|\nabla u|^{2}}{x_{1}^{2}};x_{2})}\right)\frac{u^{2}}{r}\left( \chi_{\left\{ u>0 \right\} } -1\right)\,d\mathcal{H}^{1}\\
            &+5\int_{\partial B_{r}^{+}\cap\{u_{\nu} \geqslant 0\}}\left(\frac{1}{x_{1}\bar{\rho}_{0}}-\frac{1}{x_{1}H(\tfrac{|\nabla u|^{2}}{x_{1}^{2}};x_{2})}\right)u\left( \frac{5}{2}\frac{u}{r}-\nabla u\cdot\nu \right)\,d\mathcal{H}^{1}\\
            & \geqslant \frac{25}{2}\int_{\partial B_{r}^{+}\cap\{u_{\nu} \geqslant 0\}}\left(\frac{1}{x_{1}\bar{\rho}_{0}}-\frac{1}{x_{1}H(\tfrac{|\nabla u|^{2}}{x_{1}^{2}};x_{2})}\right)\frac{u^{2}}{r}\left( \chi_{\left\{ u>0 \right\} } -1\right)\,d\mathcal{H}^{1}\\
            &+5\int_{\partial B_{r}^{+}\cap\{u_{\nu} \geqslant 0\}}\left(\frac{1}{x_{1}\bar{\rho}_{0}}-\frac{1}{x_{1}H(\tfrac{|\nabla u|^{2}}{x_{1}^{2}};x_{2})}\right)u\nabla u\cdot\nu\left(\chi_{\left\{ u>0 \right\} }-1 \right)\,d\mathcal{H}^{1}.
        \end{split}
    \end{align*}
    where in the last inequality, we used the fact that $\tfrac{u}{r} \geqslant \frac{u}{r}\chi_{\left\{ u>0 \right\}}=\tfrac{2}{5}u_{\nu}\chi_{\left\{ u>0 \right\} }$ for any $r\in(0,r_{0})$ with $r_{0}$ sufficiently small, and this can be deduced from the monotonicity formula~\eqref{Mx12r1}.
    Thanks to~\eqref{esmH}, we have 
    \begin{equation*}
        \int_{0}^{r}t^{-5}(K_{2}^{x_{1}x_{2}}(t)+K_{3}^{x_{1}x_{2}}(t))\,dt \geqslant -C\int_{0}^{r}t^{-3}\int_{\partial B_{t}^{+}}x_{1}x_{2}^{+}(1-\chi_{\left\{ u>0 \right\} })\,d\mathcal{H}^{1}dt.
    \end{equation*}
    Moreover, a direct calculation gives 
    \begin{align*}
        &\int_{0}^{r}t^{-3}\int_{\partial B_{t}^{+}}x_{1}x_{2}^{+}\left( 1-\chi_{\left\{ u>0 \right\} } \right)\,d\mathcal{H}^{1}dt\\
        &=\int_{0}^{r}t^{-3}\frac{d}{dt}\int_{B_{t}^{+}}x_{1}x_{2}^{+}(1-\chi_{\left\{ u>0 \right\} })\,dx dt\\
        &=r^{-3}\int_{B_{r}^{+}}x_{1}x_{2}^{+}\left( 1-\chi_{\left\{ u>0 \right\} } \right)\,dx+3\Pi(r).
    \end{align*}
    This implies that 
    \begin{equation}\label{5}
        \int_{0}^{r}t^{-3}(K_{2}^{x_{1}}(t)+K_{3}^{x_{1}x_{2}}(t))\,dt \geqslant -Cr^{-3}\int_{B_{r}^{+}}x_{1}x_{2}^{+}(1-\chi_{\left\{ u>0 \right\} })\,dx.
    \end{equation}
    We now obtain $V(r) \geqslant 0$ for all $r\in(0,r_{0})$, as a direct application of the estimates \eqref{0}, \eqref{Pir1},~\eqref{4} and~\eqref{5}.
\end{proof}
Define the function $N(r):=D(r)-V(r)$, the preceding lemma establishes that $N(r) \geqslant \frac{5}{2}$. We aim to prove the existence of the limit $\lim_{r\to 0^{+}}N(r)$. To this end, we analyze the derivative of $N(r)$, which is derived in the following proposition.
\begin{proposition}\label{prop:freq}
    Let $u$ be a subsonic variational solution of~\eqref{fb}, let $\delta$ be defined as in~\eqref{d} and let $0$ be a trivial original with non-zero density. Assume that $u$ satisfies the growth assumption~\eqref{Bern}. Let $D(r)$, $V(r)$ and $e(r)$ defined as in~\lemref{lemma:freq}. Define for a.e. $r\in(0,\delta)$ the function
    \begin{equation*}
        N(r):=D(r)-V(r)
    \end{equation*}
    Then
    \begin{align}\label{N1r1}
        \begin{split}
            N'(r)&=\frac{2}{rJ(r)}\displaystyle\int_{\partial B_{r}^{+}}\frac{[r(\nabla u\cdot\nu)-D(r)u]^{2}}{x_{1}H(\tfrac{|\nabla u|^{2}}{x_{1}^{2}};x_{2})}\,d\mathcal{H}^{1}\\
            &+\frac{2}{r}V^{2}(r)+\frac{2}{r}V(r)\left( N(r)-\frac{5}{2} \right)\\
            &+\frac{\frac{2}{5}K_{2}^{x_{1}x_{2}}(r)}{J(r)}\left( N(r)-\frac{5}{2} \right)-\frac{K_{3}^{x_{1}x_{2}}(r)}{J(r)},
        \end{split}
    \end{align}
    and 
    \begin{align}\label{N1r2}
        \begin{split}
            N'(r)&=\frac{2}{rJ(r)}\int_{\partial B_{r}^{+}}\frac{[r(\nabla u\cdot\nu)-N(r)u]^{2}}{x_{1}H(\tfrac{|\nabla u|^{2}}{x_{1}^{2}};x_{2})}\,d\mathcal{H}^{1}\\
            &+\frac{2}{r}V(r)\left( N(r)-\frac{5}{2} \right)\\
            &+\frac{\frac{2}{5}K_{2}^{x_{1}x_{2}}(r)}{J(r)}\left( N(r)-\frac{5}{2} \right)-\frac{K_{3}^{x_{1}x_{2}}(r)}{J(r)}\\
            &+\frac{4V(r)}{r}\left( N(r)-\frac{5}{2} \right)\frac{\mathcal{J}(r)}{J(r)}\\
            &+\frac{10V(r)}{r}\frac{\mathcal{J}(r)}{J(r)}+\frac{2}{r}V^{2}(r)\frac{\mathcal{J}(r)}{J(r)},
        \end{split}
    \end{align}
    where 
    \begin{equation*}
        \mathcal{J}(r):=\int_{\partial B_{r}^{+}}\left( \frac{1}{x_{1}H(\tfrac{|\nabla u|^{2}}{x_{1}^{2}};x_{2})}-\frac{1}{x_{1}\bar{\rho}_{0}} \right)u^{2}\,d\mathcal{H}^{1}.
    \end{equation*}
\end{proposition}
\begin{proof}
    Let us define 
    \begin{equation*}
        \tilde{I}(r):=r^{-4}I(r)-\int_{0}^{r}t^{-5}K_{1}^{x_{2}}(t)\,dt-\int_{0}^{r}\sum_{i}t^{-5}K_{i}^{x_{1}x_{2}}(t)\,dt.
    \end{equation*} 
    Then a direct calculation gives that 
    \begin{align*}
        &(r^{-5}J(r))\cdot(D(r)-V(r))\\
        &=r^{-4}\int_{B_{r}^{+}}\frac{|\nabla u|^{2}}{x_{1}H(\tfrac{|\nabla u|^{2}}{x_{1}^{2}};x_{2})}\,dx\\
        &-r^{-4}(E_{H}(u;B_{r}^{+})-E_{F}(u;B_{r}^{+}))\\
        &-r^{-4}\frac{1}{\bar{\rho}_{0}}\int_{B_{r}^{+}}x_{1}x_{2}^{+}\left( 1-\chi_{\left\{ u>0 \right\} } \right)\,dx\\
        &-\int_{0}^{r}t^{-5}K_{1}^{x_{2}}(t)\,dt-\int_{0}^{r}\sum_{i=1}^{3}t^{-5}K_{i}^{x_{1}x_{2}}(t)\,dt\\
        &=r^{-4}\int_{B_{r}^{+}}x_{1}\left[  F(\tfrac{|\nabla u|^{2}}{x_{1}^{2}};x_{2})+\lambda(x_{2})\chi_{\left\{ u>0 \right\} }\right]\,dx-\int_{0}^{r}t^{-5}K_{1}^{x_{2}}(t)\,dt\\
        &-\int_{0}^{r}\sum_{i=1}^{3}t^{-5}K_{i}^{x_{1}x_{2}}(t)\,dt-r^{-4}\frac{1}{\bar{\rho}_{0}}\int_{B_{r}^{+}}x_{1}x_{2}^{+}\,dx\\
        &=\tilde{I}(r)-r^{-4}\frac{1}{\bar{\rho}_{0}}\int_{B_{r}^{+}}x_{1}x_{2}^{+}\,dx.
    \end{align*}
    This implies that 
    \begin{equation*}
        N(r)=D(r)-V(r)=\frac{\tilde{I}(r)-r^{-4}\frac{1}{\bar{\rho}_{0}}\int_{B_{r}^{+}}x_{1}x_{2}^{+}\,dx}{r^{-5}J(r)}=\frac{\tilde{I}(r)-\frac{1}{8\bar{\rho}_{0}}}{r^{-5}J(r)}.
    \end{equation*}
    Therefore, 
    \begin{align*}
        N'(r)&=\frac{\tilde{I}'(r)\cdot(r^{-5}J(r))-(\tilde{I}(r)-\tfrac{1}{8\bar{\rho}_{0}})(r^{-5}J(r))'}{(r^{-5}J(r))^{2}}\\
        &=\frac{\tilde{I}'(r)}{r^{-5}J(r)}-\frac{(\tilde{I}(r)-\tfrac{1}{8\bar{\rho}_{0}})}{r^{-5}J(r)}\frac{(r^{-5}J(r))}{(r^{-5}J(r))'}.
    \end{align*}
    It follows from~\eqref{Ir10} and~\eqref{Jr10} that 
    \begin{align*}
        N'(r)&=\frac{\displaystyle\left(2r\int_{\partial B_{r}^{+}}\frac{(\nabla u\cdot\nu)^{2}}{x_{1}H(\tfrac{|\nabla u|^{2}}{x_{1}^{2}};x_{2})}\,d\mathcal{H}^{1}-5\int_{\partial B_{r}^{+}}\frac{u\nabla u\cdot\nu}{x_{1}H(\tfrac{|\nabla u|^{2}}{x_{1}^{2}};x_{2})}\,d\mathcal{H}^{1}\right)}{\displaystyle\int_{\partial B_{r}^{+}}\frac{u^{2}}{x_{1}\bar{\rho}_{0}}\,d\mathcal{H}^{1}}\\
        &-\frac{K_{2}^{x_{1}x_{2}}(r)}{\displaystyle\int_{\partial B_{r}^{+}}\frac{u^{2}}{x_{1}\bar{\rho}_{0}}\,d\mathcal{H}^{1}}-\frac{K_{3}^{x_{1}x_{2}}(r)}{\displaystyle\int_{\partial B_{r}^{+}}\frac{u^{2}}{x_{1}\bar{\rho}_{0}}\,d\mathcal{H}^{1}}\\
        &-(D(r)-V(r))\frac{1}{r}\cdot\frac{\displaystyle \left( 2r\int_{\partial B_{r}^{+}}\frac{u\nabla u\cdot\nu}{x_{1}\bar{\rho}_{0}}\,d\mathcal{H}^{1}-5\int_{\partial B_{r}^{+}}\frac{u^{2}}{x_{1}\bar{\rho}_{0}}\,d\mathcal{H}^{1} \right)}{\displaystyle\int_{\partial B_{r}^{+}}\frac{u^{2}}{x_{1}\bar{\rho}_{0}}\,d\mathcal{H}^{1}}.
    \end{align*}
    Since
    \begin{align*}
        &(D(r)-V(r))\frac{1}{r}\cdot\frac{\displaystyle 2r\int_{\partial B_{r}^{+}}\frac{u\nabla u\cdot\nu}{x_{1}\bar{\rho}_{0}}\,d\mathcal{H}^{1}}{\displaystyle\int_{\partial B_{r}^{+}}\frac{u^{2}}{x_{1}\bar{\rho}_{0}}\,d\mathcal{H}^{1}}\\
        &=(D(r)-V(r))\frac{1}{r}\cdot\frac{\displaystyle2r\int_{\partial B_{r}^{+}}\frac{u\nabla u\cdot\nu}{x_{1}H(\tfrac{|\nabla u|^{2}}{x_{1}^{2}};x_{2})}\,d\mathcal{H}^{1}}{\displaystyle\int_{\partial B_{r}^{+}}\frac{u^{2}}{x_{1}\bar{\rho}_{0}}\,d\mathcal{H}^{1}}\\
        &-(D(r)-V(r))\frac{\frac{2}{5}K_{2}^{x_{1}x_{2}}(r)}{\displaystyle\int_{\partial B_{r}^{+}}\frac{u^{2}}{x_{1}\bar{\rho}_{0}}\,d\mathcal{H}^{1}},
    \end{align*} 
    we have that 
    \begin{align*}
        N'(r)&=\frac{2}{r}\left( \frac{\displaystyle r^{2}\int_{\partial B_{r}^{+}}\frac{(\nabla u\cdot\nu)^{2}}{x_{1}H(\tfrac{|\nabla u|^{2}}{x_{1}^{2}};x_{2})}\,d\mathcal{H}^{1}}{\displaystyle\int_{\partial B_{r}^{+}}\frac{u^{2}}{x_{1}\bar{\rho}_{0}}\,d\mathcal{H}^{1}}-D^{2}(r) \right)\\
        &+\frac{2}{r}V^{2}(r)+\frac{2}{r}V(r)\left( N(r)-\frac{5}{2} \right)\\
        &+\frac{\frac{2}{5}K_{2}^{x_{1}x_{2}}(r)}{\displaystyle\int_{\partial B_{r}^{+}}\frac{u^{2}}{x_{1}\bar{\rho}_{0}}\,d\mathcal{H}^{1}}\left( N(r)-\frac{5}{2} \right)-\frac{K_{3}^{x_{1}x_{2}}(r)}{\displaystyle\int_{\partial B_{r}^{+}}\frac{u^{2}}{x_{1}\bar{\rho}_{0}}\,d\mathcal{H}^{1}}.
    \end{align*}
    It follows from the definition of $D(r)$ that 
    \begin{equation*}
        D(r)=\frac{\displaystyle r\int_{\partial B_{r}^{+}}\frac{u\nabla u\cdot\nu}{x_{1}H(\tfrac{|\nabla u|^{2}}{x_{1}^{2}};x_{2})}\,d\mathcal{H}^{1}}{\displaystyle\int_{\partial B_{r}^{+}}\frac{u^{2}}{x_{1}\bar{\rho}_{0}}\,d\mathcal{H}^{1}}.
    \end{equation*}
    This immediately implies that 
    \begin{align*}
        &\frac{\displaystyle r^{2}\int_{\partial B_{r}^{+}}\frac{(\nabla u\cdot\nu)^{2}}{x_{1}H(\tfrac{|\nabla u|^{2}}{x_{1}^{2}};x_{2})}\,d\mathcal{H}^{1}}{\displaystyle\int_{\partial B_{r}^{+}}\frac{u^{2}}{x_{1}\bar{\rho}_{0}}\,d\mathcal{H}^{1}}-D^{2}(r)\\
        &=\frac{\displaystyle r^{2}\int_{\partial B_{r}^{+}}\frac{(\nabla u\cdot\nu)^{2}}{x_{1}H(\tfrac{|\nabla u|^{2}}{x_{1}^{2}};x_{2})}\,d\mathcal{H}^{1}}{\displaystyle\int_{\partial B_{r}^{+}}\frac{u^{2}}{x_{1}\bar{\rho}_{0}}\,d\mathcal{H}^{1}}\\
        &-2\frac{\displaystyle r\int_{\partial B_{r}^{+}}\frac{u\nabla u\cdot\nu}{x_{1}H(\tfrac{|\nabla u|^{2}}{x_{1}^{2}};x_{2})}\,d\mathcal{H}^{1}}{\displaystyle\int_{\partial B_{r}^{+}}\frac{u^{2}}{x_{1}\bar{\rho}_{0}}\,d\mathcal{H}^{1}}\cdot D(r)+D^{2}(r)\\
        &=\frac{\displaystyle\int_{\partial B_{r}^{+}}\frac{[r(\nabla u\cdot\nu)-D(r)u]^{2}}{x_{1}H(\tfrac{|\nabla u|^{2}}{x_{1}^{2}};x_{2})}\,d\mathcal{H}^{1}}{\displaystyle\int_{\partial B_{r}^{+}}\frac{u^{2}}{x_{1}\bar{\rho}_{0}}\,d\mathcal{H}^{1}}.
    \end{align*}
    This implies 
    \begin{align*}
        N'(r)&=\frac{2}{rJ(r)}\displaystyle\int_{\partial B_{r}^{+}}\frac{[r(\nabla u\cdot\nu)-D(r)u]^{2}}{x_{1}H(\tfrac{|\nabla u|^{2}}{x_{1}^{2}};x_{2})}\,d\mathcal{H}^{1}\\
        &+\frac{2}{r}V^{2}(r)+\frac{2}{r}V(r)\left( N(r)-\frac{5}{2} \right)\\
        &+\frac{\frac{2}{5}K_{2}^{x_{1}x_{2}}(r)}{J(r)}\left( N(r)-\frac{5}{2} \right)-\frac{K_{3}^{x_{1}x_{2}}(r)}{J(r)},
    \end{align*}
    and the equality in~\eqref{N1r1} is obtained. Since $D(r)=N(r)+V(r)$, we obtain 
    \begin{align*}
        &\int_{\partial B_{r}^{+}}\frac{[r(\nabla u\cdot\nu)-D(r)u]^{2}}{x_{1}H(\tfrac{|\nabla u|^{2}}{x_{1}^{2}};x_{2})}\,d\mathcal{H}^{1}\\
        &=\int_{\partial B_{r}^{+}}\frac{[r(\nabla u\cdot\nu)-N(r)u-V(r)u]^{2}}{x_{1}H(\tfrac{|\nabla u|^{2}}{x_{1}^{2}};x_{2})}\,d\mathcal{H}^{1}\\
        &=\int_{\partial B_{r}^{+}}\frac{[r(\nabla u\cdot\nu)-N(r)u]^{2}}{x_{1}H(\tfrac{|\nabla u|^{2}}{x_{1}^{2}};x_{2})}\,d\mathcal{H}^{1}\\
        &-2V(r)r\int_{\partial B_{r}^{+}}\frac{u\nabla u\cdot\nu}{x_{1}H(\tfrac{|\nabla u|^{2}}{x_{1}^{2}};x_{2})}\,d\mathcal{H}^{1}\\
        &+2N(r)V(r)\int_{\partial B_{r}^{+}}\frac{u^{2}}{x_{1}H(\tfrac{|\nabla u|^{2}}{x_{1}^{2}};x_{2})}\,d\mathcal{H}^{1}\\
        &+V^{2}(r)\int_{\partial B_{r}^{+}}\frac{u^{2}}{x_{1}H(\tfrac{|\nabla u|^{2}}{x_{1}^{2}};x_{2})}\,d\mathcal{H}^{1}.
    \end{align*}
    Note that 
    \begin{align*}
        &2N(r)V(r)\int_{\partial B_{r}^{+}}\frac{u^{2}}{x_{1}H(\tfrac{|\nabla u|^{2}}{x_{1}^{2}};x_{2})}\,d\mathcal{H}^{1}\\
        &=2N(r)V(r)\int_{\partial B_{r}^{+}}\frac{1}{x_{1}\bar{\rho}_{0}}u^{2}\,d\mathcal{H}^{1}\\
        &+2N(r)V(r)\int_{\partial B_{r}^{+}}\left( \frac{1}{x_{1}H(\tfrac{|\nabla u|^{2}}{x_{1}^{2}};x_{2})}-\frac{1}{x_{1}\bar{\rho}_{0}} \right)u^{2}\,d\mathcal{H}^{1}.
    \end{align*}
    Therefore, we have that 
    \begin{align*}
        &\int_{\partial B_{r}^{+}}\frac{[r(\nabla u\cdot\nu)-D(r)u]^{2}}{x_{1}H(\tfrac{|\nabla u|^{2}}{x_{1}^{2}};x_{2})}\,d\mathcal{H}^{1}\\
        &=\int_{\partial B_{r}^{+}}\frac{[r(\nabla u\cdot\nu)-N(r)u]^{2}}{x_{1}H(\tfrac{|\nabla u|^{2}}{x_{1}^{2}};x_{2})}\,d\mathcal{H}^{1}-V^{2}(r)\int_{\partial B_{r}^{+}}\frac{1}{x_{1}\bar{\rho}_{0}}u^{2}\,d\mathcal{H}^{1}\\
        &+2N(r)V(r)\int_{\partial B_{r}^{+}}\left( \frac{1}{x_{1}H(\tfrac{|\nabla u|^{2}}{x_{1}^{2}};x_{2})}-\frac{1}{x_{1}\bar{\rho}_{0}} \right)u^{2}\,d\mathcal{H}^{1}\\
        &+V^{2}(r)\int_{\partial B_{r}^{+}}\left( \frac{1}{x_{1}H(\tfrac{|\nabla u|^{2}}{x_{1}^{2}};x_{2})}-\frac{1}{x_{1}\bar{\rho}_{0}} \right)u^{2}\,d\mathcal{H}^{1}.
    \end{align*}
    Note that at this stage, the equality in equation~\eqref{N1r2} is obtained by dividing the above identity on both sides by $\frac{2r}{J(r)}$.
\end{proof}
We then collect some properties of $D(r)$, $V(r)$ and $N(r)$ in the following proposition. 
\begin{proposition}\label{prop:propfre}
    Let $u$ be a subsonic variational solution of~\eqref{fb}, and let $\delta$ be defined as in~\eqref{d}. Assume that $0$ is a trivial original with non-zero density and that $u$ satisfies the growth assumption~\eqref{Bern}. Then there exists $r_{0}>0$ sufficiently small so that 
    \begin{enumerate}
        \item [(1).] $N(r) \geqslant \frac{5}{2}$ for all $r\in(0,r_{0})$.
        \item [(2).] $r\mapsto\frac{1}{r}V^{2}(r)\in L^{1}(0,r_{0})$.
        \item [(3).] $r\mapsto J(r)$ is monotone nondecreasing for all $r\in(0,r_{0})$.
        \item [(4).] The function $r\mapsto N(r)$ has a right limit $N(0^{+})$.
    \end{enumerate}
\end{proposition}
\begin{proof}
    (1). $N(r)\geqslant \frac{5}{2}$ follows directly from the inequality~\eqref{Dr-Vr} in~\lemref{lemma:freq}.

    (2). It follows from the monotonicity formula~\eqref{Mx12r1} that $u_{\nu}=\frac{5}{2}\frac{u}{r}$ for any $r\in(0,r_{0})$. This together with the definition of $K_{2}^{x_{1}x_{2}}(r)$ and $K_{3}^{x_{1}x_{2}}(r)$ gives that 
    \begin{equation*}
        \frac{|K_{2}^{x_{1}x_{2}}(r)|}{\int_{\partial B_{r}^{+}}\frac{1}{x_{1}\bar{\rho}_{0}}u^{2}\,d\mathcal{H}^{1}}\leqslant C_{1}\qquad\text{ and }\qquad\frac{|K_{3}^{x_{1}x_{2}}(r)|}{\int_{\partial B_{r}^{+}}\frac{1}{x_{1}\bar{\rho}_{0}}u^{2}\,d\mathcal{H}^{1}} \leqslant C_{2}.
    \end{equation*}
    Let us now consider two cases:
    \textit{Case I:} If $N(r) \geqslant \frac{7}{2}$, which implies that $N(r)-\frac{5}{2} \geqslant 1$. Since we only now that $N(r)$ is bounded from below by $\frac{5}{2}$, we do not know the case if $N(r)$ is large. It follows from~\eqref{Bern1} that $K_{3}^{x_{1}x_{2}}(r) \geqslant 0$, which implies that 
    \begin{equation*}
        \frac{K_{3}^{x_{1}x_{2}}(r)}{J(r)} \leqslant \frac{K_{3}^{x_{1}x_{2}}(r)}{J(r)} \left( N(r)-\frac{5}{2} \right).
    \end{equation*}
    The frequency formula in~\eqref{N1r1} gives
    \begin{align*}
        \left( N(r)-\frac{5}{2} \right)' &\geqslant -C_{1}\left( N(r)-\frac{5}{2} \right)-C_{2}\left( N(r)-\frac{5}{2} \right)\\
        &=-(C_{1}+C_{2})\left( N(r)-\frac{5}{2} \right):=-\beta\left( N(r)-\frac{5}{2} \right)\qquad \beta>0.
    \end{align*}
    It follows that 
    \begin{equation*}
        r\mapsto e^{\beta r}\left( N(r)-\frac{5}{2} \right)\quad\text{ is nondecreasing on }(0,r_{0}).
    \end{equation*}
    Therefore, if we set $g(r):=e^{\beta r}\left( N(r)-\frac{5}{2} \right)$, we have $g(r) \leqslant g(r_{0}):=c_{0}$. Thus, we have that 
    \begin{align*}
        \frac{\tfrac{2}{5}K_{2}^{x_{1}x_{2}}(r)}{J(r)}\left( N(r)-\frac{5}{2} \right)& \geqslant -\frac{\tfrac{2}{5}|K_{2}^{x_{1}x_{2}}(r)|}{J(r)}\left( N(r)-\frac{5}{2} \right)\\
        &=-\frac{\tfrac{2}{5}|K_{2}^{x_{1}x_{2}}(r)|}{e^{\beta r}J(r)}e^{\beta r}\left( N(r)-\frac{5}{2} \right)\\
        &=-\frac{\tfrac{2}{5}|K_{2}^{x_{1}x_{2}}(r)|}{e^{\beta r}J(r)}g(r)\\
        & \geqslant -\frac{\tfrac{2}{5}|K_{2}^{x_{1}x_{2}}(r)|}{e^{\beta r}J(r)}c_{0}\\
        &\geqslant -C_{1}e^{-\beta r} \geqslant -C_{1}.
    \end{align*}
    Therefore, in this case, we obtain 
    \begin{equation*}
        N'(r) \geqslant \frac{1}{r}V^{2}(r)-C.
    \end{equation*}
    Since $N(r)$ is bounded below as $r\to 0^{+}$, we obtain the integrability of $r\mapsto\frac{1}{r}V^{2}(r)$ in this case. 

    If \textit{Case II}: $N(r) \leqslant \frac{7}{2}$, we rewrite the equality in~\eqref{N1r1} as 
    \begin{align*}
        N'(r)&=\frac{2}{rJ(r)}\displaystyle\int_{\partial B_{r}^{+}}\frac{[r(\nabla u\cdot\nu)-D(r)u]^{2}}{x_{1}H(\tfrac{|\nabla u|^{2}}{x_{1}^{2}};x_{2})}\,d\mathcal{H}^{1}\\
        &+\frac{2}{r}V^{2}(r)+\frac{1}{r}V(r)\left( N(r)-\frac{5}{2} \right)\\
        &+\frac{1}{r}V(r)\left( N(r)-\frac{5}{2} \right)\\
        &+\frac{\tfrac{2}{5}(K_{2}^{x_{1}x_{2}}(r)+K_{3}^{x_{1}x_{2}}(r))}{J(r)}\left( N(r)-\frac{5}{2} \right)\\
        &-\frac{\tfrac{2}{5}K_{3}^{x_{1}x_{2}}(r)}{J(r)}N(r).
    \end{align*}
    It follows from a similar argument as in~\lemref{lemma:freq} that
    \begin{equation*}
        \frac{1}{r}V(r)+\frac{\tfrac{2}{5}K_{2}^{x_{1}x_{2}}(r)+K_{3}^{x_{1}x_{2}}(r)}{J(r)} \geqslant 0,
    \end{equation*}
    which gives 
    \begin{equation*}
        N'(r) \geqslant \frac{1}{r}V^{2}(r)-C.
    \end{equation*}
    The integrability of $r\mapsto\frac{1}{r}V^{2}(r)$ also follows.

    (3). Recalling the derivative of $(r^{-5}J(r))$ derived in~\eqref{Jr10}, one has 
    \begin{align*}
        &(r^{-5}J(r))'\\
        &=r^{-1}\left( r^{-4}\int_{\partial B_{r}^{+}}\frac{u\nabla u\cdot\nu}{x_{1}H(\tfrac{|\nabla u|^{2}}{x_{1}^{2}};x_{2})}\,d\mathcal{H}^{1}-\frac{5}{2}r^{-5}\int_{\partial B_{r}^{+}}\frac{u^{2}}{x_{1}\bar{\rho}_{0}}\,d\mathcal{H}^{1} \right.\\
        &\left.+r^{-4}\int_{\partial B_{r}^{+}}\left( \frac{1}{x_{1}\bar{\rho}_{0}}-\frac{1}{x_{1}H(\tfrac{|\nabla u|^{2}}{x_{1}^{2}};x_{2})} \right)u\nabla u\cdot\nu\,d\mathcal{H}^{1}\right)\\
        &=r^{-1}\left( r^{-4}\int_{\partial B_{r}^{+}}\frac{|\nabla u|^{2}}{x_{1}H(\tfrac{|\nabla u|^{2}}{x_{1}^{2}};x_{2})}\,d\mathcal{H}^{1}-\frac{5}{2}r^{-5}\int_{\partial B_{r}^{+}}\frac{u^{2}}{x_{1}\bar{\rho}_{0}}\,d\mathcal{H}^{1} \right.\\
        &\left.+r^{-4}\int_{\partial B_{r}^{+}}\left( \frac{1}{x_{1}\bar{\rho}_{0}}-\frac{1}{x_{1}H(\tfrac{|\nabla u|^{2}}{x_{1}^{2}};x_{2})} \right)u\nabla u\cdot\nu\,d\mathcal{H}^{1}\right).
    \end{align*}
    We deduce from the definition of $K_{3}^{x_{1}x_{2}}(r)$ and \eqref{Bern1} that $K_{3}^{x_{1}x_{2}}(r) \geqslant 0$ for any $r \geqslant 0$. Therefore, since $\lim_{r\to 0^{+}}(K_{2}^{x_{1}x_{2}}(r)+K_{3}^{x_{1}x_{2}}(r))=0$. We have that $K_{2}^{x_{1}x_{2}}(r) \leqslant 0$ for $r\in(0,r_{0})$ sufficiently small. Moreover, we infer from the definition of $K_{2}^{x_{1}x_{2}}(r)$ in the equation~\eqref{K2x12r} that $u_{\nu}=\frac{5}{2}\frac{u}{r} \geqslant 0$ for all $r\in(0,r_{0})$. This gives $(r^{-5}J(r))' \geqslant 0$ for all $r\in(0,r_{0})$.

    (4). It follows from the definition of $\mathcal{J}(r)$ that 
    \begin{equation*}
        \frac{|\mathcal{J}(r)|}{J(r)} \leqslant Cr\quad\text{ for any }r\in(0,\delta).
    \end{equation*}
    Then the claim follows by the integrability of $r\mapsto\frac{1}{r}V^{2}(r)$.
\end{proof}
\section{Analysis of the frequency sequence}
Based on the construction of our newly developed frequency formula for quasilinear free boundary problem, we can study the compactness of the frequency sequence. Furthermore, the frequency formula allows us passing to the limit to blow-up limits. Consider 
\begin{equation}\label{vm}
    v_{m}(x):=\frac{u(r_{m}x)}{\|u\|_{L_{\mathit{w}}^{2}(\partial B_{r_{m}}^{+})}}.
\end{equation}
Let us define $u_{m}(x)=\frac{u(r_{m}x)}{r_{m}^{5/2}}$, then it is easy to check that 
\begin{equation*}
    v_{m}(x)=\frac{u_{m}(x)}{\|u_{m}\|_{L_{\mathit{w}}^{2}(\partial B_{1}^{+})}},
\end{equation*}
which gives that 
\begin{equation}\label{vm1}
    \|v_{m}\|_{L_{\mathit{w}}^{2}(\partial B_{1}^{+})}\equiv 1.
\end{equation}
Moreover, a direct calculation gives that 
\begin{align}\label{vm2}
    \|\nabla v_{m}\|_{L_{\mathit{w}}^{2}(B_{1}^{+})}\leqslant \mathscr{C} D(r_{m}).
\end{align}
Inspired by~\eqref{vm1} and~\eqref{vm2}, we can prove 
\begin{proposition}\label{propv0}
    Let $u$ be a subsonic variational solution of~\eqref{fb}, and let $0$ be a trivial original with non-zero density. Assume that $u$ satisfies the growth assumption~\eqref{Bern}. Then 
    \begin{enumerate}
        \item [(1).] $\lim_{r\to 0^{+}}V(r)=0$ and $\lim_{r\to 0^{+}}D(r)=N(0^{+})$.
        \item [(2).] For any blow-up sequence $v_{m}$ defined in~\eqref{vm}, $v_{m}$ is bounded in $W_{\mathit{w}}^{1,2}(B_{1})$.
    \end{enumerate}
\end{proposition}
\begin{proof}
    It follows from the estimates~\eqref{vm1} and~\eqref{vm2} that the boundedness of $\|\nabla v_{m}\|_{L_{\mathit{w}}^{2}(B_{1}^{+})}$ is equivalent to the boundedness of $D(r_{m})$. Therefore once we proved (1), (2) is immediately obtained. In order to prove (1), we begin by using the frequency formula~\eqref{N1r2} derived in~\propref{prop:freq}. Let us define 
    \begin{equation*}
        \varphi(r;u)=\frac{1}{rJ(r;u)}\int_{\partial B_{r}^{+}}\frac{[(r\nabla u\cdot\nu)-N(r)u]^{2}}{x_{1}H(\tfrac{|\nabla u|^{2}}{x_{1}^{2}};x_{2})}\,d\mathcal{H}^{1}.
    \end{equation*}
    It follows by rescaling that 
    \begin{align*}
        \int_{r_{m}\varrho}^{r_{m}\sigma}\varphi(r)\,dr=\int_{\varrho}^{\sigma}\frac{2\bar{\rho}_{0}}{r\|v_{m}\|_{L_{\mathit{w}}^{2}(\partial B_{r}^{+})}}\int_{\partial B_{r}^{+}}\frac{[(r\nabla v_{m}\cdot\nu)-N(rr_{m})v_{m}]^{2}}{x_{1}H(\tfrac{r_{m}|\nabla u_{m}|^{2}}{x_{1}^{2}};r_{m}x_{2})}\,d\mathcal{H}^{1}.
    \end{align*}
    The statement (4) in~\propref{prop:freq} implies that 
    \begin{align*}
        \int_{\varrho}^{\sigma}\frac{2\bar{\rho}_{0}}{r\|v_{m}\|_{L_{\mathit{w}}^{2}(\partial B_{r}^{+})}}\int_{\partial B_{r}^{+}}\frac{[(r\nabla v_{m}\cdot\nu)-N(0^{+})v_{m}]^{2}}{x_{1}H(\tfrac{r_{m}|\nabla u_{m}|^{2}}{x_{1}^{2}};r_{m}x_{2})}\,d\mathcal{H}^{1}\to 0\quad\text{ as }m\to\infty.
    \end{align*}
    Therefore, we may deduce from (3) in~\propref{prop:freq} that for any $0<\varrho<\sigma<1$,
    \begin{equation}\label{f1}
        \int_{B_{\sigma}^{+}\setminus B_{\varrho}^{+}}\frac{1}{x_{1}}|x|^{-6}(\nabla v_{m}\cdot x-N(0^{+})v_{m})^{2}\,dx\to 0\quad\text{ as }m\to\infty.
    \end{equation}
    Let us define 
    \begin{equation*}
        V^{+}(r)=\frac{\displaystyle r\frac{1}{\bar{\rho}_{0}}\int_{B_{r}^{+}}x_{1}x_{2}^{+}\left( 1-\chi_{\left\{ u>0 \right\} } \right)\,dx}{\displaystyle\int_{\partial B_{r}^{+}}\frac{u^{2}}{x_{1}\bar{\rho}_{0}}\,d\mathcal{H}^{1}},
    \end{equation*}
    and 
    \begin{equation*}
        \widetilde{V}(r)=\frac{\displaystyle e(r)}{\displaystyle\int_{\partial B_{r}^{+}}\frac{u^{2}}{x_{1}\bar{\rho}_{0}}\,d\mathcal{H}^{1}}.
    \end{equation*}
    It follows that $V(r)=V^{+}(r)+\widetilde{V}(r)$. Moreover, we infer from~\eqref{Vrge0} that $\lim_{r\to 0^{+}}\frac{\widetilde{V}(r)}{V^{+}(r)}=0$. Since 
    \begin{align*}
        \frac{1}{r}V^{2}(r)&=\frac{1}{r}(V^{+}(r))^{2}+\frac{1}{r}\widetilde{V}^{2}(r)+\frac{2}{r}V^{+}(r)\widetilde{V}(r)\\
        &\geqslant \frac{1}{r}(V^{+}(r))^{2}+\frac{2}{r}V^{+}(r)\widetilde{V}(r).
    \end{align*}
    It follows from $\lim_{r\to 0^{+}}\frac{V^{+}(r)\widetilde{V}(r)}{(V^{+}(r))^{2}}=0$ that there exists $r_{0}$ sufficiently small so that 
    \begin{equation*}
        \frac{1}{r}V^{2}(r) \geqslant \frac{1}{r}(V^{+}(r))^{2}+\frac{2}{r}V^{+}(r)\widetilde{V}(r) \geqslant \frac{c_{0}}{r}(V^{+}(r))^{2}.
    \end{equation*}
    Then the integrability of $r\mapsto\frac{1}{r}V^{2}(r)$ in $L^{1}(0,r_{0})$ implies the integrability of $r\mapsto\frac{1}{r}(V^{+}(r))^{2}$. We can now prove (1) of the proposition. 
    Suppose a contradiction that (1) is not true. Let $s_{m}\to 0^{+}$ be a sequence such $V^{+}(s_{m})$ is bounded away from zero. Due to its integrability, we have that 
    \begin{equation*}
        \min_{r\in[s_{m},s_{2m}]}V^{+}(r)\to 0\quad\text{ as }m\to\infty.
    \end{equation*}
    Let $t_{m}\in[s_{m},s_{2m}]$ be such that $V^{+}(t_{m})\to 0$ as $m\to\infty$. For the choice $r_{m}:=t_{m}$ for each $m$, the sequence $v_{m}$. The fact that $V^{+}(r_{m})\to 0$ implies $\widetilde{V}(r_{m})\to 0$ and thus, $V(r_{m})\to 0$. This further implies that $D(r_{m})$ is bounded since $N(0^{+})$ has a right limit. Hence $v_{m}$ is bounded in $W_{\mathit{w}}^{1,2}(B_{1}^{+})$. It follows that $v_{0}$ is a homogeneous function of degree $N(0^{+})$. Note that also, 
    \begin{align*}
        V^{+}(s_{m})&=\frac{s_{m}^{-4}\int_{B_{s_{m}}^{+}}x_{1}x_{2}^{+}\left( 1-\chi_{\left\{ u>0 \right\} } \right)\,dx}{s_{m}^{-5}\int_{\partial B_{s_{m}}^{+}}\tfrac{1}{x_{1}}u^{2}\,d\mathcal{H}^{1}}\\
        & \leqslant \frac{s_{m}^{-4}\int_{B_{s_{m}}^{+}}x_{1}x_{2}^{+}\left( 1-\chi_{\left\{ u>0 \right\} } \right)\,dx}{(r_{m}/2)^{-5}\int_{\partial B_{r_{m}/2}^{+}}\tfrac{1}{x_{1}}u^{2}\,d\mathcal{H}^{1}}\\
        & \leqslant \frac{1}{2}\frac{\int_{\partial B_{r_{m}}^{+}}\frac{1}{x_{1}}u^{2}\,d\mathcal{H}^{1}}{\int_{\partial B_{r_{m}/2}^{+}}\frac{1}{x_{1}}u_{0}^{2}\,d\mathcal{H}^{1}}\widetilde{V}^{+}(r_{m})\\
        &=\frac{1}{2}\frac{V^{+}(r_{m})}{\|v_{m}\|_{L_{\mathit{w}}^{2}( \partial B_{1/2}^{+} )}^{2}}.
    \end{align*}
    Since, at least along a subsequence,
    \begin{equation*}
        \|v_{m}\|_{L_{\mathit{w}}^{2}(\partial B_{1/2}^{+})}^{2}\to \|v_{0}\|_{L_{\mathit{w}}^{2}(\partial B_{1/2}^{+})}^{2}>0
    \end{equation*}
    This leads a contradiction. It follows that $V^{+}(r)\to 0$ as $r\to 0^{+}$. Thus, $\widetilde{V}(r)\to 0^{+}$ and $D(r)\to N(0^{+})$.
\end{proof}
As a direct corollary of~\propref{propv0} and equation~\eqref{f1}, one has 
\begin{corollary}
    Let $v_{m}$ be the sequence defined in~\eqref{vm} that converges weakly in $W_{\mathit{w}}^{1,2}(B_{1}^{+})$ to a blow-up limit $v_{0}$. Then $v_{0}$ is a continuous homogeneous function of degree $N(0^{+})$ in $B_{1}^{+}$, and satisfies 
    \begin{equation*}
        v_{0} \geqslant 0\quad\text{ in }B_{1}^{+}\qquad v_{0}\equiv 0\quad\text{ in }B_{1}^{+}\cap\{x_{2} \leqslant 0\}\qquad \|v_{0}\|_{L_{\mathit{w}}^{2}(\partial B_{1}^{+})}=1.
    \end{equation*}
\end{corollary}
\section{Compensated compactness and the strong convergence}
In order to passing to the limit in domain variation solutions, we have to prove that the strong convergence of $v_{m}$. 
\begin{proposition}\label{prop:cc}
    Let $u$ be a subsonic variational solution of~\eqref{fb}, and let $0$ be a trivial original with non-zero density. Assume that $u$ satisfies the growth assumption~\eqref{Bern}. Then $v_{m}$ converges to $v_{0}$ strongly in $W_{\mathit{w},\mathit{loc}}^{1,2}(B_{1}^{+}\setminus\{0\})$ and $v_{0}\operatorname{div}\left( \frac{1}{x_{1}}\nabla v_{0} \right)=0$ in the sense of distributions on $B_{1}^{+}$.
\end{proposition}
\begin{proof}
    Given that $u_{m}$ satisfies $\operatorname{div}(\tfrac{1}{x_{1}H_{m}}\nabla u_{m})=0$ in $\{u_{m}>0\}$ where $H_{m}=H(\tfrac{r_{m}|\nabla u_{m}|^{2}}{x_{1}^{2}};r_{m}x_{2})$, we have that $v_{m}$ satisfies 
    \begin{equation*}
        \operatorname{div}\left( \frac{1}{x_{1}H_{m}}\nabla v_{m} \right)=0\quad\text{ in }\{v_{m}>0\}.
    \end{equation*}
    Let us now state the compensated compactness framework. Assume that $(u_{1,m},u_{2,m},\rho_{m})$ is a sequence of approximate solutions to the Euler system 
    \begin{align*}
        \left\{
            \begin{alignedat}{2}
                &\partial_{x_{1}}(x_{1}\rho_{m}u_{1,m})+\partial_{x_{2}}(x_{1}\rho_{m}u_{2,m})=0,\\
                &\partial_{x_{1}}(x_{1}\rho_{m}u_{1,m}^{2})+\partial_{x_{2}}(x_{1}\rho u_{1,m}u_{2,m})+x_{1}\partial_{x_{1}}p(\rho_{m})=0,\\
                &\partial_{x_{1}}(x_{1}\rho u_{1,m}u_{2,m})+\partial_{x_{2}}(x_{1}\rho u_{2,m}^{2})+\partial_{x_{2}}p(\rho_{m})+\rho_{m}g=0.
            \end{alignedat}
        \right.
    \end{align*} 
    Assume further that $(\rho_{m},u_{1,m},u_{2,m})$ satisfies $u_{1,m}^{2}+u_{2,m}^{2} \leqslant \sqrt{p'(\rho_{m})}$ a.e. in $\Omega$, that $B_{m}:=\frac{u_{1,m}^{2}+u_{2,m}^{2}}{2}+\int_{\bar{\rho}_{0}}^{\rho_{m}}\frac{p'(s)}{s}\,ds$ is uniformly bounded and that $\curl(u_{1,m},u_{2,m})$ is a locally bounded measure, then there exists a subsequence that converges a.e. to a weak solution $(\bar{\rho},\bar{u}_{1},\bar{u}_{2})$ as $m\to\infty$ of the above Euler system. In other words, for any test function $\varphi\in C_{0}^{1}(\Omega)$,
    \begin{align*}
        \left\{
            \begin{alignedat}{2}
                &\int_{\Omega}x_{1}(\bar{\rho}\bar{u}_{1}\varphi_{x_{1}}+\bar{\rho}\bar{u}_{2}\varphi_{x_{2}})\,dx=0,\\
                &\int_{\Omega}x_{1}(\bar{\rho}\bar{u}_{1}^{2}+p(\bar{\rho})+\bar{\rho}\bar{u}_{1}\bar{u}_{2})\varphi_{x_{1}}\,dx=0,\\
                &\int_{\Omega}x_{1}\left( (\bar{\rho}\bar{u}_{1}\bar{u}_{2})\varphi_{x_{1}}+(\bar{\rho}\bar{u}_{2}^{2}+p)\varphi_{x_{2}}+\rho g\varphi\right)\,dx=0.
            \end{alignedat}
        \right.
    \end{align*}
    This implies that 
    \begin{equation}\label{cc1}
        \int_{\Omega}x_{1}\rho_{m}u_{1,m}u_{2,m}\eta\,dx\to \int_{\Omega}x_{1}\bar{\rho}\bar{u}_{1}\bar{u}_{2}\eta\,dx\quad\text{ for any test function }\eta\in C_{0}^{0}(\Omega).
    \end{equation}
    We modify each $v_{m}$ and $H_{m}$ to 
    \begin{equation*}
        \tilde{v}_{m}:=v_{m}*\phi_{m}\in C^{\infty}(B_{1}),\qquad\tilde{H}_{m}=H_{m}*\phi_{m}\in C^{\infty}(B_{1}),
    \end{equation*}
    where $\phi_{m}$ is a standard mollifier such that $\operatorname{div}\left( \tfrac{1}{x_{1}H_{m}}\nabla\tilde{v}_{m} \right) \geqslant 0$ and 
    \begin{equation*}
        \left( \operatorname{div}\left( \frac{\nabla\tilde{v}_{m}}{x_{1}H_{m}} \right) \right)(B_{\sigma}) \leqslant \mathscr{C}_{2}\quad\text{ for all }m.
    \end{equation*}
    Moreover, 
    \begin{equation*}
        \|v_{m}-\tilde{v}_{m}\|_{W_{\mathit{w}}^{1,2}(B_{\sigma})}\to 0\qquad\|\tilde{H}_{m}-H_{m}\|_{L^{2}(B_{\sigma})}\to 0\quad\text{ as }m\to \infty.
    \end{equation*}
    Set $\rho_{m}=H_{m}$, and consider the associated velocity field which is given by $(u_{1,m},u_{2,m}):=(\tfrac{\partial_{2}\tilde{v}_{m}}{x_{1}H_{m}},-\tfrac{\partial_{1}\tilde{v}_{m}}{x_{1}H_{m}})$.  For any $0<\varrho<\sigma<1$, it is easy to verify that i) $u_{1,m}^{2}+u_{2,m}^{2} \leqslant \sqrt{p'(\rho_{m})}$ in $B_{\sigma}^{+}\setminus B_{\varrho}^{+}$, ii) $\frac{u_{1,m}^{2}+u_{2,m}^{2}}{2}+\int_{\bar{\rho}_{0}}^{\rho_{m}}\frac{p'(s)}{s}\,ds$ is bounded in $B_{\sigma}^{+}\setminus B_{\varrho}^{+}$, and iii) $\curl(u_{1,m},u_{2,m})=\operatorname{div}\left( \tfrac{\nabla\tilde{v}_{m}}{x_{1}H_{m}} \right)$ is locally a bounded measure. The framework of compensated compactness can be applied and we deduce from~\eqref{cc1} that
    \begin{equation*}
        \int_{B_{\sigma}^{+}\setminus B_{\varrho}^{+}}\frac{1}{x_{1}\tilde{H}_{m}}\partial_{2}\tilde{v}_{m}\partial_{1}\tilde{v}_{m}\eta\,dx\to\int_{B_{\sigma}^{+}\setminus B_{\varrho}^{+}}\frac{1}{x_{1}\bar{\rho}_{0}}\partial_{2}v_{0}\partial_{1}v_{0}\eta\,dx,
    \end{equation*}
    for each $\eta\in C_{0}^{0}(B_{\sigma}^{+}\setminus\bar{B}_{\varrho}^{+})$. It follows from the boundedness of $\nabla v_{m}$ and the strong $W^{1,2}$ convergence of $\tilde{v}_{m}$ to $v_{m}$ and the following two identities 
    \begin{align*}
        &\int_{B_{\sigma}^{+}\setminus B_{\varrho}^{+}}\frac{1}{x_{1}H_{m}}\left( \partial_{1}v_{m}\partial_{2}v_{m}-\partial_{1}v_{0}\partial_{2}v_{0} \right)\eta\,dx\\
        &=-\int_{B_{\sigma}^{+}\setminus B_{\varrho}^{+}}\frac{1}{x_{1}H_{m}}\left( \partial_{1}\tilde{v}_{m}\partial_{2}\tilde{v}_{m}-\partial_{1}v_{m}\partial_{2}v_{m} \right)\eta\,dx\\
        &+\int_{B_{\sigma}^{+}\setminus B_{\varrho}^{+}}\frac{1}{x_{1}H_{m}}\left( \partial_{1}\tilde{v}_{m}\partial_{2}\tilde{v}_{m}-\partial_{1}v_{0}\partial_{2}v_{0} \right)\eta\,dx,
    \end{align*}
    and 
    \begin{align*}
        &\int_{B_{\sigma}^{+}\setminus B_{\varrho}^{+}}\frac{1}{x_{1}H_{m}}\left( \partial_{1}\tilde{v}_{m}\partial_{2}\tilde{v}_{m}-\partial_{1}v_{m}\partial_{2}v_{m} \right)\eta\,dx\\
        &=\int_{B_{\sigma}^{+}\setminus B_{\varrho}^{+}}\frac{1}{x_{1}H_{m}}\left( (\partial_{1}\tilde{v}_{m}-\partial_{1}v_{m})\partial_{2}\tilde{v}_{m}+\partial_{1}v_{m}\left( \partial_{2}\tilde{v}_{m}-\partial_{2}v_{m} \right) \right)\eta \,dx
    \end{align*}
    that 
    \begin{equation}\label{cc2}
        \int_{B_{\sigma}^{+}\setminus B_{\varrho}^{+}}\frac{1}{x_{1}H_{m}}\partial_{2}v_{m}\partial_{1}v_{m}\eta\,dx\to\int_{B_{\sigma}^{+}\setminus B_{\varrho}^{+}}\frac{1}{x_{1}\bar{\rho}_{0}}\partial_{2}v_{0}\partial_{1}v_{0}\eta\,dx,
    \end{equation}
    for each $\eta\in C_{0}^{0}(B_{\sigma}^{+}\setminus\bar{B}_{\varrho}^{+})$. On the other hand, for each $0<\varrho<\sigma$, the equation~\eqref{f1} implies that 
    \begin{equation*}
        \nabla v_{m}\cdot x-N(0^{+})v_{m}\to 0\quad\text{ strongly in }L^{2}(B_{\sigma}^{+}\setminus B_{\varrho}^{+}),
    \end{equation*}
    Since $v_{0}$ is a homogeneous function of degree $N(0^{+})$, we have $\nabla v_{0}\cdot x=N(0^{+})v_{0}$ and this implies 
    \begin{align}\label{f2}
        \begin{split}
            \partial_{1}v_{m}\cdot x_{1}+\partial_{2}v_{m}\cdot x_{2}&=\nabla v_{m}\cdot x\stackrel{L^{2}(B_{\sigma}^{+}\setminus B_{\varrho}^{+})}{\longrightarrow} N(0^{+})v_{m}\stackrel{L^{2}(B_{\sigma}^{+}\setminus B_{\varrho}^{+})}{\longrightarrow } N(0^{+})v_{0}\\
            &=\nabla v_{0}\cdot x=\partial_{1}v_{0}\cdot x_{1}+\partial_{2}v_{0}\cdot x_{2}.
        \end{split}
    \end{align}
    We now claim that 
    \begin{align}\label{cc3}
        \begin{split}
            &\int_{B_{\sigma}^{+}\setminus B_{\varrho}^{+}}\frac{1}{x_{1}H_{m}}\left( \partial_{1}v_{m}\partial_{1}v_{m}x_{1}+\partial_{1}v_{m}\partial_{2}v_{m}x_{2} \right)\eta\,dx\\
            &\to\int_{B_{\sigma}^{+}\setminus B_{\varrho}^{+}}\frac{1}{x_{1}\bar{\rho}_{0}}\left( \partial_{1}v_{0}\partial_{1}v_{0}+\partial_{1}v_{0}\partial_{2}v_{0} \right)\eta\,dx.
        \end{split}
    \end{align}
    Since $H_{m}$ converges to $\bar{\rho}_{0}$ strongly in $L^{2}(B_{\sigma}^{+})$. It suffices to prove that 
    \begin{align*}
        &\int_{B_{\sigma}^{+}\setminus B_{\varrho}^{+}}\frac{1}{x_{1}H_{m}}\Bigg[  \partial_{1}v_{m}\left( \partial_{1}v_{m}x_{1}+\partial_{2}v_{m}x_{2} \right)\eta\\
        &\qquad\qquad\qquad\qquad-\partial_{1}v_{0}\left( \partial_{1}v_{0}x_{1}+\partial_{2}v_{0}x_{2} \right)\eta\Bigg]\,dx \to 0.
    \end{align*}
    However, this is a direct application of~\eqref{f2} and the following computation
    \begin{align*}
        &\int_{B_{\sigma}^{+}\setminus B_{\varrho}^{+}}\frac{1}{x_{1}H_{m}}\left[  \partial_{1}v_{m}\left( \partial_{1}v_{m}x_{1}+\partial_{2}v_{m}x_{2} \right)-\partial_{1}v_{0}\left( \partial_{1}v_{0}x_{1}+\partial_{2}v_{0}x_{2} \right)\right]\eta\,dx\\
        &=\int_{B_{\sigma}^{+}\setminus B_{\varrho}^{+}}\frac{1}{x_{1}H_{m}}\left[ \partial_{1}v_{m}\left( \partial_{1}v_{m}x_{1}+\partial_{2}v_{m}x_{2}-(\partial_{1}v_{0}x_{1}+\partial_{2}v_{0}x_{2}) \right) \right]\eta\,dx\\
        &+\int_{B_{\sigma}^{+}\setminus B_{\varrho}^{+}}\frac{1}{x_{1}H_{m}}(\partial_{1}v_{0}x_{1}+\partial_{2}v_{0}x_{2})(\partial_{1}v_{m}-\partial_{1}v_{0})\eta\,dx,
    \end{align*}
    where we used the weak convergence of $\nabla v_{m}$ to $\nabla v_{0}$. Therefore~\eqref{cc3} is proved. we deduce from~\eqref{cc3} and~\eqref{cc2} that 
    \begin{equation*}
        \int_{B_{\sigma}^{+}\setminus B_{\varrho}^{+}}\frac{1}{H_{m}}(\partial_{1}v_{m})^{2}\eta\,dx \to \int_{B_{\sigma}^{+}\setminus B_{\varrho}^{+}}\frac{1}{\bar{\rho}_{0}}(\partial_{1}v_{0})^{2}\eta\,dx,
    \end{equation*}
    for each $\eta\in C_{0}^{0}(B_{\sigma}^{+}\setminus\bar{B}_{\varrho}^{+})$. Similarly, we can prove that 
    \begin{align*}
        &\int_{B_{\sigma}^{+}\setminus B_{\varrho}^{+}}\frac{1}{x_{1}H_{m}}\left( \partial_{1}v_{m}\partial_{2}v_{m}x_{1}+\partial_{2}v_{m}\partial_{2}v_{m}x_{2} \right)\eta\,dx\\
        &\to\int_{B_{\sigma}^{+}\setminus B_{\varrho}^{+}}\frac{1}{x_{1}\bar{\rho}_{0}}\left( \partial_{1}v_{0}\partial_{2}v_{0}x_{1}+\partial_{2}v_{0}\partial_{2}v_{0}x_{2} \right)\eta\,dx.
    \end{align*}
    Using~\eqref{cc2} once more yields that 
    \begin{equation*}
        \int_{B_{\sigma}^{+}\setminus B_{\varrho}^{+}}\frac{x_{2}}{x_{1}H_{m}}(\partial_{2}v_{m})^{2}\eta\,dx \to \int_{B_{\sigma}^{+}\setminus B_{\varrho}^{+}}\frac{x_{2}}{x_{1}\bar{\rho}_{0}}(\partial_{2}v_{0})^{2}\eta\,dx.
    \end{equation*}
    Using the strong $L^{2}$ convergence of $H_{m}$ to $\bar{\rho}_{0}$, one has $\nabla v_{m}$ converges strongly in $L_{\mathit{w},\mathit{loc}}^{2}(B_{\sigma}^{+}\setminus\bar{B}_{\varrho}^{+})$. Since $0<\varrho<\sigma<1$ were arbitrary, we have that $\nabla v_{m}$ converges strongly in $L_{\mathit{w},\mathit{loc}}^{2}(B_{1}^{+}\setminus\{0\})$. As a consequence of the strong convergence, we see that 
    \begin{equation*}
        \int_{B_{1}^{+}}\frac{1}{x_{1}\bar{\rho}_{0}}\nabla\left( \eta v_{0} \right)\cdot\nabla v_{0}\,dx =0\quad\text{ for each }\eta\in C_{0}^{1}(B_{1}^{+}\setminus\{0\}).
    \end{equation*}
    Combined with the fact that $v_{0}=0$ in $B_{1}^{+}\cap\{x_{2} \leqslant 0\}$, we proves that 
    \begin{equation*}
        v_{0}\operatorname{div}\left(  \frac{1}{x_{1}\bar{\rho}_{0}}\nabla v_{0}\right)=0
    \end{equation*}
    in the sense of Radon measures on $B_{1}^{+}$.
\end{proof}
As an application of~\propref{prop:cc}, we obtained the existence of a nontrivial homogeneous solution of degree $\geqslant \frac{5}{2}$ at the origin, which can also be regarded as a compressible counterpart of Theorem 7.1 in~\cite{MR3225630}. 
\begin{corollary}
    Let $u$ be a subsonic weak solution so that the origin is a degenerate original point. Suppose that the free boundary $B_{1}^{+}\cap\partial\{u>0\}$ is in a neighborhood of the origin a continuous injective curve $\sigma:I\to \mathbb{R} ^{2}$, where $I$ is an interval of $\mathbb{R} $ containing $0$, such that $\sigma=(\sigma_{1},\sigma_{2})$ and $\sigma(0)=0$. Then $\sigma_{1}(t)\neq 0$ in $(0,t_{1})$, 
    \begin{equation*}
        \lim_{t\to 0^{+}}\frac{\sigma_{2}(t)}{\sigma_{1}(t)}=0,
    \end{equation*}
    and 
    \begin{equation*}
        \frac{u(rx)}{\|u\|_{L_{\mathit{w}^{2}}(\partial B_{1}^{+})}}\to \beta x_{1}^{2}x_{2}^{+}\quad\text{ as }t\to 0^{+},
    \end{equation*}
    strongly in $W_{\mathit{w},\mathit{loc}}^{1,2}(B_{1}^{+}\setminus\{0\})$ and weakly in $W_{\mathit{w}}^{1,2}(B_{1}^{+})$. Here $\beta$ is a constant given by 
    \begin{equation*}
        \beta:=\frac{1}{\sqrt{\int_{\partial B_{1}^{+}}x_{1}^{3}(x_{2}^{+})^{2}\,d\mathcal{H}^{1}}}.
    \end{equation*}
\end{corollary}
Since the blow-up limit $v_{0}$ satisfies $v_{0}\operatorname{div}(\frac{1}{x_{1}}\nabla v_{0})=0$ on $B_{1}^{+}$, the result follows by adapting the steps from Theorem 7.1 in~\cite{MR3225630}. We therefore omit it.
\appendix
\section{The proof of the first variation formula}\label{app1}
Given a vector field $\phi\in C_{0}^{1}(\Omega)$, for small $\varepsilon>0$ we consider the ODE flow $y=y(\varepsilon;x)$ defined by
\begin{equation*}
    y(0;x)=x,\qquad \partial_{\varepsilon}y(\varepsilon;x)=\phi(y(\varepsilon;x)),
\end{equation*}
It follows that for $\varepsilon>0$ small, $y(\varepsilon;x)=x+\varepsilon\phi(y(\varepsilon;x))+o(\varepsilon)=x+\varepsilon\phi(x)+o(\varepsilon)$ and 
\begin{equation*}
    \det D_{x}y(\varepsilon;x)=\mathit{Id}+\varepsilon D\phi(x)+o(\varepsilon).
\end{equation*}
For variational solution $u$ let us define the perturbation $u_{\varepsilon}(x):=u(y(\varepsilon;x))$. It follows by the change of variable $y=y(\varepsilon;x)$ and a direct calculation that
\begin{align*}
    &\int_{\Omega}x_{1}F\left( \frac{|\nabla u_{\varepsilon}(x)|^{2}}{x_{1}^{2}} ;x_{2}\right)\,dx\\
    &=\int_{\Omega}(y_{1}-\varepsilon\phi_{1}(y(\varepsilon;x)))\\
    &\qquad\qquad F \begin{pmatrix}
        \frac{|\nabla u(y;\varepsilon)|^{2}+2\varepsilon\nabla u(y(\varepsilon;x))D\phi(y(\varepsilon;x))\nabla u(y(\varepsilon;x))}{(y_{1}-\varepsilon\phi_{1}(y;\varepsilon))^{2}}+o(\varepsilon)\\
        y_{2}-\varepsilon\phi_{2}(y(\varepsilon;x))+o(\varepsilon) 
    \end{pmatrix}\,dx\\
    &=\int_{\Omega}(y_{1}-\varepsilon\phi_{1}(y))\\
    &\qquad \qquad F \begin{pmatrix}
        \frac{|\nabla u(y)|^{2}+2\varepsilon\nabla u(y)D\phi(y)\nabla u(y)+o(\varepsilon)}{(y_{1}-\varepsilon\phi_{1}(y))^{2}}\\
        y_{2}-\varepsilon\phi_{2}(y)+o(\varepsilon)
    \end{pmatrix}(1-\varepsilon \operatorname{div}\phi(y)+o(\varepsilon))\,dy.
\end{align*} 
Define $f(\varepsilon)=\frac{1}{(y_{1}-\varepsilon\phi_{1}(y))^{2}}$ and we have 
\begin{align*}
    f(\varepsilon)=f(0)+f'(0)\varepsilon+o(\varepsilon)=\frac{1}{y_{1}^{2}}+\frac{2\phi_{1}(y)}{y_{1}^{3}}\varepsilon+o(\varepsilon).
\end{align*}
This implies that 
\begin{align}\label{fv1}
    \begin{split}
        &\int_{\Omega}x_{1}F\left( \frac{|\nabla u_{\varepsilon}(x)|^{2}}{x_{1}^{2}} ;x_{2}\right)\,dx\\
        &=\int_{\Omega}y_{1}F\left( \tfrac{|\nabla u(y)|^{2}}{y_{1}^{2}};y_{2} \right)\,dy-\varepsilon\int_{\Omega}y_{1}F\left( \tfrac{|\nabla u(y)|^{2}}{y_{1}^{2}};y_{2} \right)\operatorname{div}\phi(y)\,dy\\
        &-\varepsilon \int_{\Omega}y_{1}\partial_{2}F\left( \tfrac{|\nabla u(y)|^{2}}{y_{1}^{2}};y_{2} \right)\phi_{2}(y)\,dy-\varepsilon \int_{\Omega}F\left( \tfrac{|\nabla u(y)|^{2}}{y_{1}^{2}};y_{2} \right)\phi_{1}(y)\,dy\\
        &+2\varepsilon\int_{\Omega}\partial_{1}F\left( \tfrac{|\nabla u(y)|^{2}}{y_{1}^{2}};y_{2}-\varepsilon\phi_{2}(y)+o(\varepsilon) \right)\frac{\nabla u(y)D\phi(y)\nabla u(y)}{y_{1}}\,dy\\
        &+2\varepsilon\int_{\Omega}\partial_{1}F\left( \tfrac{|\nabla u(y)|^{2}}{y_{1}^{2}};y_{2}-\varepsilon\phi_{2}(y)+o(\varepsilon) \right)\frac{\phi_{1}(y)|\nabla u(y)|^{2}}{y_{1}^{2}}\,dy+o(\varepsilon).
    \end{split}
\end{align}
Similarly, one has 
\begin{align}\label{fv2}
    \begin{split}
        &\int_{\Omega}x_{1}\lambda(x_{2})\chi_{\left\{ u_{\varepsilon}>0 \right\} }\,dx\\
        &=\int_{\Omega}y_{1}\lambda(y_{2})\chi_{\left\{ u>0 \right\} }\,dy-\varepsilon\int_{\Omega}y_{1}\lambda(y_{2})\chi_{\left\{ u>0 \right\} }\operatorname{div}\phi(y)\,dy\\
        &-\varepsilon\int_{\Omega}y_{1}\lambda'(y_{2})\phi_{2}(y)\chi_{\left\{ u>0 \right\} }\,dy-\varepsilon\int_{\Omega}\lambda(y_{2})\phi_{1}(y)\chi_{\left\{ u>0 \right\} }\,dy.
    \end{split}
\end{align}
Then~\eqref{fv} follows from~\eqref{fv1} and~\eqref{fv2} and the definition of the first domain variation.
\subsubsection*{Acknowledgement}
\hyphenpenalty=10
\sloppy
Du is supported by National
Nature Science Foundation of China under Grants 12125102, 12526202, Nature Science Foundation of Guangdong Province under Grant 2024A1515012794, and Shenzhen Science and Technology Program
(JCYJ20241202124209011).
\subsubsection*{Conflict of interests}
The authors declare no conflicts of interests.
\subsubsection*{Data availability} No data was used in this research.


\begin{thebibliography}{99}
    \bibitem[AM20]{MR4163981}J.~Abrantes Santos and S.~H.~Monari Soares, A limiting free boundary problem for a degenerate operator in Orlicz-Sobolev spaces, {\it Rev. Mat. Iberoam.} {\bf 36} (2020), no.~6, 1687--1720. 

    \bibitem[AC81]{MR618549}H.~W. Alt and L.~A. Caffarelli, Existence and regularity for a minimum problem with free boundary, {\it J. Reine Angew. Math.} {\bf 325} (1981), 105--144.
    
    \bibitem[ACF82]{MR647374}H.~W. Alt, L.~A. Caffarelli and A. Friedman, Jet flows with gravity, {\it J. Reine Angew. Math.} {\bf 331} (1982), 58--103.
    
    \bibitem[ACF83]{MR682265}H.~W. Alt, L.~A. Caffarelli and A. Friedman, Axially symmetric jet flows, {\it Arch. Rational Mech. Anal.} {\bf 81} (1983), no.~2, 97--149. 

    \bibitem[ACF84]{MR752578}H.~W. Alt, L.~A. Caffarelli and A. Friedman, A free boundary problem for quasilinear elliptic equations, {\it Ann. Scuola Norm. Sup. Pisa Cl. Sci.} (4) {\bf 11} (1984), no.~1, 1--44. 

    \bibitem[ACF85]{MR772122}H.~W. Alt, L.~A. Caffarelli and A. Friedman, Compressible flows of jets and cavities, {\it J. Differential Equations} {\bf 56} (1985), no.~1, 82--141. 

    \bibitem[AFT82]{MR666110}C.~J. Amick, L.~E. Fraenkel and J.~F. Toland, On the Stokes conjecture for the wave of extreme form, {\it Acta Math.} {\bf 148} (1982), 193--214. 

    \bibitem[AL12]{MR2915865}D. Arama and G. Leoni, On a variational approach for water waves, {\it Comm. Partial Differential Equations} {\bf 37} (2012), no.~5, 833--874. 

    \bibitem[BF24]{MR4773610}M. Bayrami and M. Fotouhi, Regularity in the two-phase Bernoulli problem for the $p$-Laplace operator, {\it Calc. Var. Partial Differential Equations} {\bf 63} (2024), no.~7, Paper No. 183, 38 pp.

    

    \bibitem[Bers58]{MR96477} L. Bers, {\it Mathematical aspects of subsonic and transonic gas dynamics}, Surveys in Applied Mathematics, Vol. 3, Wiley, New York, 1958 Chapman \& Hall, Ltd., London, 1958. 

    \bibitem[Caf77]{MR454350}L.~A. Caffarelli, The regularity of free boundaries in higher dimensions, {\it Acta Math.} {\bf 139} (1977), no.~3-4, 155--184.
    
    \bibitem[CF82]{MR642623}L.~A. Caffarelli and A. Friedman, Axially symmetric infinite cavities, {\it Indiana Univ. Math. J.} {\bf 31} (1982), no.~1, 135--160. 

    \bibitem[Chan57]{MR92663} S. Chandrasekhar, {\it An introduction to the study of stellar structure}, Dover, New York, 1957. 

    \bibitem[CDSW07]{MR2291790}Q.~G. Chen, C.~M.~Dafermos, M.~Slemrod and D.~H.~Wang, On two-dimensional sonic-subsonic flow, {\it Comm. Math. Phys.} {\bf 271} (2007), no.~3, 635--647. 

    \bibitem[CHW16]{MR3437861}Q.~G. Chen, F. Huang and T.~Y. Wang, Subsonic-sonic limit of approximate solutions to multidimensional steady Euler equations, {\it Arch. Ration. Mech. Anal.} {\bf 219} (2016), no.~2, 719--740.

    \bibitem[CHLWW24]{MR4716737}G.~Q.~Chen, F.~Huang, T.~Li, W.~Wang, Y.~Wang, Global finite-energy solutions of the compressible Euler-Poisson equations for general pressure laws with large initial data of spherical symmetry, {\it Comm. Math. Phys.} {\bf 405} (2024), no.~3, Paper No. 77, 85 pp. 

    \bibitem[CD18]{MR3842050}J.~ Cheng and L. Du, Compressible subsonic impinging flows, {\it Arch. Ration. Mech. Anal.} {\bf 230} (2018), no.~2, 427--458. 

    \bibitem[CDW18]{MR3814594}J.~ Cheng, L.~Du and Y.~Wang, The existence of steady compressible subsonic impinging jet flows, {\it Arch. Ration. Mech. Anal.} {\bf 229} (2018), no.~3, 953--1014. 

    \bibitem[CDX20]{MR4127792}J. Cheng, L.~Du and W. Xiang, Steady incompressible axially symmetric R\'{e}thy flows, {\it Nonlinearity} {\bf 33} (2020), no.~9, 4627--4669.

    \bibitem[CDZ21]{MR4246821}J.~Cheng, L.~Du and Q. Zhang, Existence and uniqueness of axially symmetric compressible subsonic jet impinging on an infinite wall, {\it Interfaces Free Bound.} {\bf 23} (2021), no.~1, 1--58. 

    \bibitem[CF76]{MR421279}R. Courant and K.~O. Friedrichs, {\it Supersonic flow and shock waves}, Applied Mathematical Sciences, Vol. 21, Springer, New York-Heidelberg, 1976. 

    \bibitem[DRRV23]{MR4591831}J.~V. Da~Silva, G.~C.~Rampasso, G.~C.~Ricarte and H.~A.~Vivas, Free boundary regularity for a class of one-phase problems with non-homogeneous degeneracy, {\it Israel J. Math.} {\bf 254} (2023), no.~1, 155--200; 

    \bibitem[DP05]{MR2133664}D. Danielli and A. Petrosyan, A minimum problem with free boundary for a degenerate quasilinear operator, {\it Calc. Var. Partial Differential Equations} {\bf 23} (2005), no.~1, 97--124. 

    \bibitem[Del92]{MR1176677}J.-M. Delort, Une remarque sur le probl\`{e}me des nappes de tourbillon axisym\'{e}triques sur $\mathbb{R}^3$, {\it J. Funct. Anal.} {\bf 108} (1992), no.~2, 274--295.

    \bibitem[DD11]{MR2737815}L.~Du and B. Duan, Global subsonic Euler flows in an infinitely long axisymmetric nozzle, {\it J. Differential Equations} {\bf 250} (2011), no.~2, 813--847. 

    \bibitem[DD16]{MR3537905}
    L.~Du and B. Duan, Subsonic Euler flows with large vorticity through an infinitely long axisymmetric nozzle, {\it J. Math. Fluid Mech.} {\bf 18} (2016), no.~3, 511--530. 

    \bibitem[DHP23]{MR4595616}L.~ Du, J. Huang and Y. Pu, The free boundary of steady axisymmetric inviscid flow with vorticity $I$: near the degenerate point, {\it Comm. Math. Phys.} {\bf 400} (2023), no.~3, 2137--2179. 

    \bibitem[DY24a]{MR4808256}L.~ Du and C. Yang, The free boundary of steady axisymmetric inviscid flow with vorticity $II$: near the non-degenerate points, {\it Comm. Math. Phys.} {\bf 405} (2024), no.~11, Paper No. 262, 58 pp. 

    \bibitem[DY24]{du2024proofstokesconjecturecompressible}L.~Du and C.~Yang, Proof of the Stokes conjecture for compressible gravity water waves, arXiv preprint: arXiv: \href{https://arxiv.org/abs/2410.04178}{2410.04178} (2024).

    \bibitem[EFW25]{MR4848670}S. Eberle, A. Figalli and G.~S. Weiss, Complete classification of global solutions to the obstacle problem, {\it Ann. of Math.} (2) {\bf 201} (2025), no.~1, 167--224. 

    \bibitem[Fri83]{MR702728}A. Friedman, Axially symmetric cavities in rotational flows, {\it Comm. Partial Differential Equations} {\bf 8} (1983), no.~9, 949--997. 

    \bibitem[FJ19]{MR3904453}A. Figalli and J. Serra, On the fine structure of the free boundary for the classical obstacle problem, {\it Invent. Math.} {\bf 215} (2019), no.~1, 311--366. 

    \bibitem[GS55]{MR67650}D. Gilbarg and J.~ Serrin, Uniqueness of axially symmetric subsonic flow past a finite body, {\it J. Rational Mech. Anal.} {\bf 4} (1955), 169--175;

    \bibitem[KW25]{MR4850026}D. Kriventsov and G.~S. Weiss, Rectifiability, finite Hausdorff measure, and compactness for non-minimizing Bernoulli free boundaries, {\it Comm. Pure Appl. Math.} {\bf 78} (2025), no.~3, 545--591. 


    \bibitem[Lind05]{MR2177323} H. Lindblad, Well posedness for the motion of a compressible liquid with free surface boundary, {\it Comm. Math. Phys.} {\bf 260} (2005), no.~2, 319--392.

    \bibitem[LL18]{MR3812074}H.~Lindblad and C.~Luo, A priori estimates for the compressible Euler equations for a liquid with free surface boundary and the incompressible limit, {\it Comm. Pure Appl. Math.} {\bf 71} (2018), no.~7, 1273--1333. 

    \bibitem[Luo18]{MR3887218} C. Luo, On the motion of a compressible gravity water wave with vorticity, {\it Ann. PDE.} {\bf 4} (2018), no.~2, Paper No. 20, 71 pp.

    \bibitem[LZ22]{MR4439376}C.~Luo and J.~Zhang, Local well-posedness for the motion of a compressible gravity water wave with vorticity, {\it J. Differential Equations} {\bf 332} (2022), 333--403. 

    \bibitem[Mil60]{MR112435}L.~M. Milne~Thomson, {\it Theoretical hydrodynamics}, The Macmillan Company, New York, 1960; 

    \bibitem[Mur78]{MR506997}F. Murat, Compacit\'{e}par compensation, {\it Ann. Scuola Norm. Sup. Pisa Cl. Sci.} (4) {\bf 5} (1978), no.~3, 489--507. 

    \bibitem[MN08]{MR2431665}S.~R. Mart\'{i}nez and N.~I. Wolanski, A minimum problem with free boundary in Orlicz spaces, {\it Adv. Math.} {\bf 218} (2008), no.~6, 1914--1971. 

    \bibitem[Sto80]{MR2858161}G.~G. Stokes, {\it Mathematical and physical papers. Volume 1}, reprint of the 1880 original, 
    Cambridge Library Collection, Cambridge Univ. Press, Cambridge, 2009. 

    \bibitem[Tar79]{MR584398}L. Tartar, Compensated compactness and applications to partial differential equations, in {\it Nonlinear analysis and mechanics: Heriot-Watt Symposium, Vol. IV}, pp. 136--212, Res. Notes in Math., 1979. 

    \bibitem[VW11]{MR2810856}E. V\v{a}rv\v{a}ruc\v{a} and G.~S. Weiss, A geometric approach to generalized Stokes conjectures, {\it Acta Math.} {\bf 206} (2011), no.~2, 363--403. 

    \bibitem[VW12]{MR2995099}E. V\v{a}rv\v{a}ruc\v{a} and G.~S. Weiss, The Stokes conjecture for waves with vorticity, {\it Ann. Inst. H. Poincar\'{e} C Anal. Non Lin\'{e}aire} {\bf 29} (2012), no.~6, 861--885. 

    \bibitem[VW14]{MR3225630}E. V\v{a}rv\v{a}ruc\v{a} and G.~S. Weiss, Singularities of steady axisymmetric free surface flows with gravity, {\it Comm. Pure Appl.} Math. {\bf 67} (2014), no.~8, 1263--1306. 

    \bibitem[Wei98]{MR1620644}G.~S. Weiss, Partial regularity for weak solutions of an elliptic free boundary problem, {\it Comm. Partial Differential Equations} {\bf 23} (1998), no.~3-4, 439--455. 

    \bibitem[Wei99]{MR1759450}G.~S. Weiss, Partial regularity for a minimum problem with free boundary, {\it J. Geom. Anal.} {\bf 9} (1999), no.~2, 317--326.

    \bibitem[Wei03]{MR1989835}G.~S. Weiss, A singular limit arising in combustion theory: fine properties of the free boundary, {\it Calc. Var. Partial Differential Equations} {\bf 17} (2003), no.~3, 311--340. 

    \bibitem[WZ10]{MR2748622}G.~S. Weiss and G. Zhang, Existence of a degenerate singularity in the high activation energy limit of a reaction-diffusion equation, {\it Comm. Partial Differential Equations} {\bf 35} (2010), no.~1, 185--199.

    \bibitem[Wei21]{MR4238496}G.~S. Weiss, Bernoulli type free boundary problems and water waves, in {\it Geometric measure theory and free boundary problems}, 89--136, Lecture Notes in Math. Fond. CIME/CIME Found.  Springer, Cham. 
    
    \bibitem[XX07]{MR2375709}C. Xie and Z. Xin, Global subsonic and subsonic-sonic flows through infinitely long nozzles, {\it Indiana Univ. Math. J.} {\bf 56} (2007), no.~6, 2991--3023. 

    \bibitem[XX10a]{MR2644144}C. Xie and Z. Xin, Global subsonic and subsonic-sonic flows through infinitely long axially symmetric nozzles, {\it J. Differential Equations} {\bf 248} (2010), no.~11, 2657--2683. 

    \bibitem[XX10b]{MR2607929}C. Xie and Z. Xin, Existence of global steady subsonic Euler flows through infinitely long nozzles, {\it SIAM J. Math. Anal.} {\bf 42} (2010), no.~2, 751--784. 

\end{thebibliography}
\end{document}